\documentclass[11pt]{amsart}

\usepackage{amsxtra, amssymb, amsmath, amsthm, stmaryrd}
\usepackage{mathbbol, mathtools}
\usepackage{color}
\usepackage[all]{xy}
\usepackage{hyperref}
\usepackage{appendix}
\usepackage{cancel}
\usepackage{graphicx}
\usepackage{blindtext}

\theoremstyle{plain}
 
\makeatletter
\@namedef{subjclassname@1991}{2020 Mathematics Subject Classification}
\makeatother

\usepackage{chngcntr}

\counterwithin*{equation}{section}

\renewcommand{\theequation}{\arabic{section}.\arabic{equation}}

\newtheorem{theorem}{Theorem}[section]
\newtheorem{lemma}[theorem]{Lemma}
\newtheorem{proposition}[theorem]{Proposition}

\newtheorem{corollary}[theorem]{Corollary}
\newtheorem{definition}[theorem]{Definition} \theoremstyle{definition}
\newtheorem{example}[theorem]{Example}
\newtheorem{remark}[theorem]{Remark}
\newtheorem*{theorem*}{Theorem}

\newcommand{\target}   {{\mathsf{t}}}

\newcommand{\bN}{{\bar{N}}}
\newcommand{\bp}{{\bar{\partial}}} 
\DeclareMathOperator{\CE}{CE}

\newcommand{\R}{\mathbb{R}} 
 
\newcommand{\N}{\mathbb{N}} 
\newcommand{\GG}{{\mathbb{G}}}

\newcommand{\aaar}{\substack{\longrightarrow\\[-0.85em] \longrightarrow \\[-0.85em] \longrightarrow}}

 \DeclareMathOperator{\pr}{pr}
\DeclareMathOperator{\tr}{\mathbb{T}}
\DeclareMathOperator{\Hom}{Hom}
\DeclareMathOperator{\End}{End}

\DeclareMathOperator{\Id}{Id}
\DeclareMathOperator{\id}{id}

\DeclareMathOperator{\im}{im}

\newcommand*{\emptycomment}[1]{}

\DeclareMathOperator{\coker}{coker}

\DeclareMathOperator{\algd}{algd} 

\newcommand{\huaB}{\mathcal{B}}

\newcommand{\huaA}{\mathcal{A}}
\newcommand{\huaL}{\mathcal{L}}

\newcommand{\huaF}{\mathcal{F}}
\newcommand{\huaG}{\mathcal{G}}

\newcommand{\huaO}{\mathcal{O}}
\newcommand{\huaT}{\mathcal{T}}

\newcommand{\OO}{\alpha}

\DeclareMathOperator{\codim}{codim}

\newcommand{\Z}{\ensuremath{\mathbb Z}}
\newcommand{\G}{{{\mathcal Arrow}_1(G)}} 

\newcommand{\Sh}{\mathsf{Shuff}}
\newcommand{\String}{\textup{String}}

\newcommand{\g}{\ensuremath{\mathfrak{g}}}
\newcommand{\h}{\ensuremath{\mathfrak{h}}}

\makeatletter
\let\@wraptoccontribs\wraptoccontribs
\makeatother

\title{Shifted symplectic higher Lie groupoids and  classifying spaces}
\author{Miquel Cueca}
\address{Mathematics Institute\\Georg-August-University of G\"ottingen\\Bunsenstra{\ss}e 3-5\\G\"ottingen 37073\\Germany}

\author{Chenchang Zhu}

\contrib[ \lowercase{with appendix} E \lowercase{by}]{Florian Dorsch}

\email{miquel.cuecaten@mathematik.uni-goettingen.de}
\email{chenchang.zhu@mathematik.uni-goettingen.de}
\email{florian.dorsch@stud.uni-goettingen.de}

\subjclass{53D17, 18N65, , 22E67}

\begin{document}

\begin{abstract}
We introduce the concept of $m$-shifted symplectic Lie $n$-groupoids and symplectic Morita equivalences between them. We then build various models for the 2-shifted symplectic structure on the classifying stack in this setting and construct explicit symplectic Morita equivalences between them.  
\end{abstract}

\keywords{Lie $n$-groupoids, Shifted symplectic structures, Classifying stacks, Quadratic Lie algebras.}

\maketitle
\vspace{-15mm}

\tableofcontents

\section{Introduction}

In recent years, the development of TQFTs has encouraged the study of higher versions of symplectic manifolds \cite{rog:cat, kont:for, van:dou}. Among them, on one hand, we have very concrete models of $\N$-graded symplectic $Q$-manifolds 
\cite{aksz, saf:shi, royt, sch:geo, Severa05}. This is familiar to Poisson geometers and mathematical physicists. On the other hand, there is also the beautiful universal and algebraic viewpoint of shifted symplectic stacks  \cite{ cala:der, cal:shi, PaToVaVe13, prid:shift}, which play a prominent role in a framework familiar to algebraic geometers. There is a series of works \cite{cala:der, cal:shi, saf:shi, safro, saf:poi} using the more abstract language of the latter to develop further theories in Poisson geometry and mathematical physics. At the same time,  there is also a series of works \cite{ping:shift, Xu:ShiftP} using language more familiar to differential geometers to study such shifted Poisson or symplectic structures on manifolds and Lie groupoids. As algebraic stacks are conceptually related to differentiable stacks, we may understand that Lie $n$-groupoids play the role of linking  in all these works for the objects carrying the (shifted) symplectic structures. 
Hence it is natural to study what ``symplectic" Lie $n$-groupoids are,  as posed by Getzler in \cite{Lesdiablerets}.
\begin{equation}\label{diag1}
\xymatrix{\text{Degree n } Q\text{-manifold }\ar@{-->}@<1ex>[rr]^{\quad\text{Integration}}&& \text{ Lie n-groupoid }\ar^{\quad \ \text{Differentiation}}@{-->}@<1ex>[ll]\ar@<1ex>[rr]^{\text{Morita class}\quad \ }&&\ar^{\text{Atlas}\quad}@<1ex>[ll]\text{ Differentiable n-stack}}
\end{equation}

Continuing the ideas therein, in this article, we give a careful and rigorous definition of  {\bf $m$-shifted symplectic Lie $n$-groupoids} (see Definition \ref{def:mSSLnG}). We then establish the theory of {\bf symplectic Morita equivalence} (see Section \ref{sec:me}). This is fundamental because only then can the concept of  $m$-shifted symplectic forms  descend to the stack behind the groupoid presentations, thus the concept of $m$-shifted symplectic structures becomes intrinsic.
One of the most important examples of shifted symplectic stacks is the classifying stack $\huaB G$ for a certain Lie group $G$. Over the years, people have created a zoo of forms related to a Lie group $G$, coming from Lie's integration  theory \cite{Getzler04, henriques}, loop group theory \cite{Segal:loop}, Mackenzie's double theory \cite{Mackenzie99}, and dynamical systems \cite{Lu-we1, lu-weinstein}. In this article, we construct three (higher) groups with explicit differential forms coming from the above source  (see Table \ref{tabex}). Our second main achievement is to show that all three of these  are $2$-shifted symplectic structures, and to discover the hidden differential forms which provide the explicit symplectic Morita equivalences among them (see Theorems \ref{thm-morita} and \ref{Thm-manin}).  
In the end, we sketch the outline of shifted symplectic structures in the bi-simplicial world. We construct two models of $\huaB G$ using bi-simplicial manifolds and prove how the shifted symplectic forms that they carry are equivalent. 
We hope that such a concrete language and its further development may assist more people to access the concept of higher symplectic structure more easily, and thus enrich this domain itself from several roots. 



\subsection{Theory}%
In this work we model differentiable $n$-stacks by Lie $n$-groupoids, which are simplicial manifolds satisfying suitable Kan conditions (see Definition \ref{def:lie-n-gpd}). The differentiable forms on a differentiable $n$-stack can be explicitly modeled by the normalised double complex $(\hat{\Omega}^\bullet(X_\bullet), d, \delta)$ (forms vanishing under degeneracies) on a Lie $n$-groupoid $X_\bullet$ presenting this stack (see Section \ref{sec:tan}). An {\bf $m$-shifted $k$-form  } $\alpha_\bullet$ is 
\begin{equation} \label{eq:cut-0}
    \alpha_\bullet=\sum_{i=0}^m \alpha_i\ \text{ with } \ \alpha_i\in\widehat{\Omega}^{k+m-i}(X_i).
\end{equation}
The form $\alpha_\bullet$ is {\bf closed} if $D\alpha_\bullet=0$. The cut on the range of the form in \eqref{eq:cut-0} above, that is the fact that $\alpha_\bullet$ does not contain $0$, $1$, $\dots$, $(k-1)$-forms,  was initially introduced in  \cite{Lesdiablerets} using a filtration and a fine resolution for forms to understand the closedness condition in \cite{PaToVaVe13}.

However, to define a {\bf $m$-shifted symplectic $2$-form}, we need to address the {\em non-degeneracy} in a homotopic manner. That is, this form needs to induce a quasi-isomorphism from the tangent complex to the cotangent complex.  In contrast to algebraic geometry where things are taken as limit of local models, a shifted $2$-form $\alpha_\bullet$ on a Lie $n$-groupoid
does not directly induce a map between the tangent and cotangent complex. It is only the associated IM (infinitesimal multiplicative) form $\lambda^{\alpha_\bullet}$ which builds a morphism  between the tangent and cotangent complex. In differential geometry, this link between global and infinitesimal must be created. For example, in the classical theory of forms on Lie groups and their Lie algebras, this is addressed by the Van Est theory. In other words, ideally we would have needed a higher Van Est theory to pass between $\alpha_\bullet$ and $\lambda^{\alpha_\bullet}$. However thanks to the work of Getzler \cite{Lesdiablerets}, his formula \eqref{nondeg-pairing} provides a short-cut. We prove that this $\lambda^{\alpha_\bullet}$ satisfies various properties (see Appendix \ref{Ap:IM}), which are exactly what one expects for IM forms. Therefore we believe that $\lambda^{\alpha_\bullet}$ provides the first leading term for the IM form of an $m$-shifted 2-form, which is enough for us to define the {\bf non-degeneracy} of the 2-form (see Definition \ref{def:mSSLnG}).

It turns out that when $m, n$ takes small values, e.g. $(0, 0)$, $(0, 1)$, $(1, 1)$, $(2,1)$, $(1,2)$, such $m$-shifted symplectic Lie $n$-groupoids have appeared naturally and been already studied thoroughly in Poisson geometry (see Section \ref{sec:example-m-n}). These include symplectic manifolds, 0-symplectic groupoids \cite{hoffman-sjamaar},  Weinstein symplectic groupoids \cite{cdw}, and quasi-symplectic \cite{Xu04} (aka twisted preseymplectic \cite{BCWZ}) groupoids.

Furthermore, as the diagram \eqref{diag1} suggests, another important ingredient is the relation given by {\bf symplectic Morita equivalences} between $m$-shifted symplectic Lie $n$-groupoids  (see Definition \ref{def:sympmorita}).  This is one of the major differences with the algebraic geometric approach. For the purpose of explicit calculation and interesting examples in the context of differential geometry, one cares about different higher groupoid presentations of the higher stack instead of just the higher stack itself. Therefore, Morita equivalence between these different models becomes an important concept to make sure that the internal geometric meaning is not lost in the apparently very different calculations in different models. Indeed, it is only after quotienting by symplectic Morita equivalences, $m$-shifted symplectic Lie $n$-groupoids will give us the notion of $m$-shifted symplectic differentiable $n$-stacks. 

It is well known that Morita equivalence of Lie groupoids has three equivalent definitions: it may be defined as a principal bi-bundle, or a zig-zag of weak equivalences, or a zig-zag of hypercovers. The latter two are extended to  higher Lie groupoids \cite{Rogers-Zhu:2016, z:tgpd-2}. In this work, we use a zig-zag of hypercovers to give a definition of symplectic Morita equivalence.  Let $(K_\bullet, \alpha_\bullet)$ and $(J_\bullet, \beta_\bullet)$ be two $m$-shifted symplectic Lie $n$-groupoids. They are symplectic Morita equivalent, if there is another Lie $n$-groupoid $Z_\bullet$ with an $(m-1)$-shifted 2-form $\phi_\bullet$ and hypercovers $f_\bullet$, $g_\bullet$ satisfying
\begin{equation*}\label{eq:sympme-0}
    K_\bullet\xleftarrow{f_\bullet}Z_\bullet\xrightarrow{g_\bullet}J_\bullet \quad \text{and} \quad f^*_\bullet\alpha_\bullet-g^*_\bullet\beta_\bullet=D\phi_\bullet.
\end{equation*}
The first condition tells that $K_\bullet$ and $J_\bullet$ are Morita equivalent as Lie $n$-groupoids, hence they present the same $n$-stack. The second condition tells that the forms represent the same shifted symplectic structure on the $n$-stack. The geometric meaning of symplectic Morita equivalence is that $$(Z_\bullet, \phi_\bullet) \xrightarrow{f_\bullet\times g_\bullet} (K_\bullet \times J_\bullet, (\alpha_\bullet, -\beta_\bullet))$$ presents an $m$-shifted Lagrangian $n$-stack in the sense of \cite{PaToVaVe13}. This is made more precise in the upcoming article \cite{ABC:lag}.    

For consistency, we prove (see Theorem \ref{thm:morita-1-1}) that when $m=n=1$,  our notion of symplectic Morita equivalence coincide with that of Xu given by Hamiltonian principal bimodules in \cite{Xu04}. People have discovered many examples of Hamiltonian principal bimodules carrying natural physics meanings in Poisson geometry (see \cite{Xu04} and a newly appeared work \cite{AM:22}). Therefore Theorem \ref{thm:morita-1-1} implies that each such a bimodule gives rise to a shifted Lagrangian stack.

All this theoretical framework is in Section \ref{Sec:sslg} and is the first main accomplishment of this article.

\subsection{Shifted symplectic structure on classifying stack $\huaB G$}
Recall that a quadratic Lie algebra $(\g,\langle\cdot,\cdot\rangle)$ is a Lie algebra $\g$ endowed with a symmetric pairing $\langle\cdot,\cdot\rangle:\g\times\g\to\mathbb{R}$ that is adjoint-invariant and non-degenerate. The classifying stack $\huaB G$ of the Lie group $G$ integrating such $\g$ carries a $2$-shifted symplectic structure $\omega_{\huaB G}$ induced by the pairing, see \cite{PaToVaVe13}. This is an important example of shifted symplectic structure because, as the main result of \cite{PaToVaVe13} demonstrates, $(\huaB G, \omega_{\huaB G})$ further induces other shifted symplectic structures on various derived moduli stacks, which explains many existing theories in algebraic geometry.


On the other hand, in graded geometry,  it is established \cite{royt} that $(\g,\langle\cdot,\cdot\rangle)$ is a Courant algeboid over a point, so it gives rise to the degree $2$ symplectic $NQ$-manifold $(\g[1], Q=d_{CE}, \{\cdot,\cdot\}=\langle\cdot,\cdot\rangle)$.   Thus, for any quadratic Lie algebra we shall expect a $2$-shifted symplectic Lie $n$-groupoid\footnote{Since $\g[1]$ is a degree $1$ manifold and $\huaB G$ is a $1$-stack, it is reasonable to expect $n=1$. However, a $1$-stack may also allow a presentation by a Lie 2-groupoid. In our case, such a Lie 2-groupoid model for $\huaB G$ turns out to be interesting and related to Segal's loop group theory. } $(K_\bullet,\Omega_\bullet)$  making the following diagram commute

\begin{equation}\label{diag2}
\xymatrix{(\g[1], Q, \{\cdot,\cdot\})\ \ar@{-->}@<1ex>[rr]^{\quad\text{Integration}}&&\ (K_\bullet,\Omega_\bullet)\ \ar^{\qquad \text{Differentiation}}@{-->}@<1ex>[ll]\ar@<1ex>[rr]^{\text{Morita class}\quad \ }&&\ar^{\text{Atlas}\quad}@<1ex>[ll]\ (\huaB G, \omega_{\huaB G})}.
\end{equation}
The second main result of this article is to provide three different $(K_\bullet,\Omega_\bullet)$ that fits in \eqref{diag2} and to build symplectic Morita equivalences between them.  All three are explicit and summarized in the Table \ref{tabex}. In such a way, each of this $(K_\bullet,\Omega_\bullet)$ provides an integration of the Courant algebroid $(\g,[\cdot,\cdot],\langle\cdot,\cdot\rangle)$. Therefore our work produces a complete solution for the problem of integrating Courant algebroids whose base is a point.

Perhaps the most exciting motivation to investigate this example is that $3d$-Chern-Simons field theory can be described as a $\sigma$-model with target the shifted symplectic stack $(\huaB G,\omega_{\huaB G})$ (or $(\g[1],Q,\{\cdot,\cdot\})$ for the infinitesimal theory), see \cite{cat:cs, dw, dom:cs}. From the point of view of extended TQFT, it is conjectured that 
what $3d$-Chern-Simons theory assigns to a point is a certain class of projective unitary representations of the based loop group, see \cite{hen:what}. 
We expect that those can be obtained from the geometric quantization of the 2-shifted symplectic structures constructed here. In fact, in the second model \ref{sec:2nd-model}, we already see the shadow of based loop group and Segal's symplectic form. This may be a sign that we are on the correct track. \emptycomment{Moreover, all these natural 2-shifted symplectic forms that we find on various models are symplectic Morita equivalent. This is again a good news because we may have many choices of symplectic structures to start the quantization, however the fact that  all the natural forms we find are equivalent is a strong hint that all of them are correct choices.  }

\emptycomment{As we know, $(\g[1], Q, \{\cdot, \cdot \})$ is a Courant algeboid over a point. Thus all these models in Table \ref{tabex} provide equivalent integration of Courant algebroid in this special case.}

\begin{table}[h]
\[
\begin{array}{|c|c|c|}
\hline
  \text{Model}& 2\text{-shifted symplectic form}\\
\hline
 NG_\bullet=\cdots\ G\times G\aaar G\rightrightarrows pt & 
 \def\arraystretch{3.6}
 \begin{array}{c}
 \begin{split}
\Omega_\bullet=&\Omega-\Theta,\\
\Omega=&\langle d_2^*\theta^l,d_0^*\theta^r\rangle\in\Omega^2(G\times G),\\ \Theta=&\frac{1}{6}\langle \theta^l,[\theta^l,\theta^l]\rangle\in\Omega^3(G).
\end{split}
 \end{array}\\
\hline
 \GG_\bullet=\cdots\ \Omega G\aaar P_eG\rightrightarrows pt &
 \def\arraystretch{2.8}
 \begin{array}{c}
     \begin{split}
\omega_\bullet=&\omega\in\Omega^2(\Omega G),\\
\omega_\tau(v,w)=&\int_{S^1}\langle \big(L_{\tau(t)^{-1}}v(t)\big)',L_{\tau(t)^{-1}}w(t)\rangle\ dt.\end{split}
 \end{array} \\
\hline
 \bar{\Gamma}_\bullet^{\mathfrak{h}}=\cdots H_-\times\Gamma^\mathfrak{h}\times H_+\aaar H_-\times H_+\rightrightarrows pt & 
 \def\arraystretch{2.5}
 \begin{array}{c}
\begin{split}
\bar{\omega}^{\mathfrak{h}}_\bullet=&\omega^\mathfrak{h}\in\Omega^2(\Gamma^\mathfrak{h}),\\
\omega^\mathfrak{h}=&\langle\theta^l_{H_+}, \theta^r_{H_-}\rangle-\langle\theta^l_{H_-}, \theta^r_{H_+}\rangle.\end{split}
 \end{array}
 \\
\hline
\end{array}\]
\caption{The $2$-shifted symplectic (local) Lie $2$-group models of $\huaB G$.\label{tabex}}
\end{table}

 Now, we describe the three models, outline their main properties and indicate related works. 
 
 \subsubsection{First model} \label{sec:1st-model}  The first model $(NG_\bullet, \Omega_\bullet)$ (see Section \ref{sec-findim}) is given by the nerve of $G$, the Lie group integrating $\g$. The 2-shifted  $2$-form $\Omega_\bullet$ has two pieces, $\Theta$ known as the Cartan $3$-form and $\Omega$ that we named as Brylinski-Weinstein $2$-form. We prove that it is furthermore non-degenerate as in Definition \ref{def:mSSLnG} and thus $2$-shifted symplectic (see Theorem \ref{Thm-fin-sym}). This model appeared first in \cite{brylinski} and in \cite{Alan:moduli} where it was used to compute the symplectic structure on the moduli space of flat connections on a closed surface. A beautiful calculation in \cite{safro} shows that the transgression of the 2-shifted symplectic form $\Omega_\bullet$ on   $NG_\bullet$ gives the 1-shifted symplectic form  on $G\times G\rightrightarrows G$ described in \cite{BCWZ, Xu04}.  We further show in Proposition \ref{VE} that the Van Est map defined in \cite{ArCr11} satisfies $VE(\Omega_\bullet)=-\langle\cdot,\cdot\rangle$.  There are therefore  well defined integration and differentiation processes that shows the left part of \eqref{diag2}. As $NG_\bullet$ gives an atlas for $\huaB G$ and $\omega_{\huaB G}$ in \cite{PaToVaVe13} is constructed by $\langle\cdot,\cdot\rangle$, $(NG_\bullet, \Omega_\bullet)$ gives a model for $(\huaB G, \omega_{\huaB G})$.


\subsubsection{Second model} \label{sec:2nd-model} The second model $(\GG_\bullet, \omega_\bullet)$ (see Section \ref{Sec:infdim}) is given at level $1$ by paths on $G$ starting at the identity, and at level $2$ by based loops $\Omega G$ equipped with  Segal's symplectic $2$-form $\omega$ \cite{Segal:loop}. It is  very nice to see that $\omega$ is, in addition, multiplicative and furthermore provides a $2$-shifted symplectic structure (see Theorem \ref{thm-inf-sym}). In fact, as \cite{Mein:flat} indicates, the Chern-Simon $\sigma$-model on a disk provides, via a symplectic reduction, Segal's symplectic form on $\Omega G$. This reduction shows up also in the second truncation\footnote{The truncation of simplicial manifolds is similar to the classical operation of truncation for simplicial sets in algebraic topology. See Definition \ref{def:truncation}. } of the universal integration of $\g$:  It turns out that $\GG_\bullet$ is exactly this  truncation. Moreover, when Segal's form $\omega$ is integral, $(\GG_\bullet, \omega_\bullet)$ has a $\huaB^2U(1)$ central extension (see \cite{Baez:from, henriques, waldorf:string-CS}) known as $\String(G)_\bullet$, equipped with a sort of connection. Therefore $\String(G)_\bullet$ can be regarded as a prequantization of $\huaB G$.  

\subsubsection{Third model} \label{sec:3rd-model}  In order to define the third model $(\bar{\Gamma}_\bullet^\h , \bar{\omega}^\h_\bullet )$ (see Section \ref{sec:manin-model}), we need the additional structure of a Manin triple $(\g, \h_+, \h_-)$.  It was shown in \cite{Lu-we1} that a Manin triple integrates to a double symplectic Lie group. Moreover, in \cite{MeTa11} the authors applyn the Artin-Mazur codiagonal to this double Lie group and produce a local Lie $2$-group  $\bar{\Gamma}_\bullet^\h$ with $2$-form $\bar{\omega}^\h_\bullet$. We prove that these give rise to a 2-shifted symplectic local Lie 2-group (see Proposition \ref{prop-manin-sym}). There is some obstruction to extend the local structure which comes from the fact that the integration of hamiltonian systems is often local \cite{semenov79} (see also Remark \ref{rmk:obstruction}). 

\subsubsection{Equivalence Theorems} \label{sec:equ}
The $2$-shifted symplectic group $(NG_\bullet, \Omega_\bullet)$ gives a presentation for the $2$-shifted symplectic stack  $(\huaB G, \omega_{\huaB G})$ as in \ref{sec:1st-model}. We prove two equivalence theorems (Theorems \ref{thm-morita} and \ref{Thm-manin}) summarized in the following to demonstrate that the other two models also present 
 $(\huaB G, \omega_{\huaB G})$. 

\begin{theorem}\label{thm:equi-0}
Let $(\g,[\cdot,\cdot],\langle\cdot,\cdot\rangle)$ be a quadratic Lie algebra and $G$ the connected and simply connected Lie group integrating it.
\begin{enumerate}
    \item The hypercovers $\id_\bullet$ and $ev_\bullet$ (defined in Lemma \ref{evaluation}) give the following symplectic Morita equivalence  
$$(\GG_\bullet,\omega_\bullet)\xleftarrow{\id_\bullet}(\GG_\bullet,\omega^P_\bullet)\xrightarrow{ev_\bullet}(NG_\bullet, \frac{1}{2}\Omega_\bullet), $$
with $\omega_\bullet^P=\omega^P+0$ the $1$-shifted $2$-form on $\GG_\bullet$ defined by 
$$\omega^{P}_\gamma(u, v) = \frac{1}{2} \int_{0}^1 \langle \widehat{u}'(t) , \widehat{v}(t) \rangle -  \langle \widehat{u}(t) , \widehat{v}'(t) \rangle dt\in\Omega^2(\GG_1),$$
for $\gamma\in P_eG,\ u,v\in T_\gamma P_eG,$ where  $ \widehat{u}(t)=L_{\gamma(t)^{-1}}u(t), \ \widehat{v}(t)=L_{\gamma(t)^{-1}}v(t).$
\item Let $(\g,\h_+,\h_-)$ be a Manin triple of $\g$ and denote by $H_\pm$ the connected and simply connected integrations of $\h_\pm$. The hypercovers $\id_\bullet$ and $\Phi_\bullet$ (as in Proposition \ref{prop:phi}) give the following symplectic Morita equivalence $$(\bar{\Gamma}^\h_\bullet, \bar{\omega}_\bullet^\h)\xleftarrow{\id_\bullet}(\bar{\Gamma}^\h_\bullet, \beta_\bullet)\xrightarrow{\Phi_\bullet}(NG_\bullet, \Omega_\bullet),$$
with $\beta_\bullet=\beta+0.$  The $1$-shifted $2$-form on  $\bar{\Gamma}^\h_\bullet$ defined by
$$\beta=\langle \theta^l_{H_-},\theta^r_{H_+}\rangle\in\Omega^2(\bar{\Gamma}^\h_1=H_-\times H_+).$$
\end{enumerate}
\end{theorem}

As we mentioned before, one of the most significant motivations to study the symplectic structure on $\huaB G$  is its quantization via 3d-Chern-Simons theory. Thus the fact that all these natural 2-shifted symplectic forms that we find on various models are symplectic Morita equivalent provides a solid foundation for the quantization process ahead. We may use any of them for our convenience of calculation. 

Actually, Morita theory is itself an important concept beyond the theory of stacks.  In Poisson geometry, Morita theory was first introduced  for symplectic groupoids  \cite{x1} and then extended to $1$-shifted symplectic groupoids in \cite{Xu04} motivated by that for  $C^*$-algebras. Today Morita equivalence is deeply embedded in many important theories in Poisson geometry, such as Picard group for Poisson manifolds \cite{bur-rui18, bu-we:poisson}, moment map theory \cite{ AM:22, Xu04}, deformation quantization \cite{bur:mor12, bur:mor20}. We expect the concrete Morita equivalences constructed in Theorem \ref{thm:equi-0} can assist the explicit calculations in the above context.

\subsection{Shifted double theory, a concise outline and an invitation} There is a big community in differential geometry and mathematical physics who prefer to work with bi-simplicial objects. In Section \ref{Sec:double}, we explore shifted symplectic structures in the bi-simplicial context through explicit models for $\huaB G$. 

Strict Lie $2$-groups can be expressed in terms of internal categories or double Lie groups \cite{baez:2gp, Mackenzie99} which give rise to bi-simplicial manifolds. 
We define double shifted forms in the de Rham triple complex of a bi-simplicial manifold.   It turns out that our first two models of $\huaB G$ are both equivalent to strict Lie 2-groups,   so they transfer into double Lie groups. We further transfer the two shifted symplectic forms into the double setting. In the end, we prove (see Theorem \ref{Thm-dou-ev}) that these two double shifted forms are equivalent in a similar way as in Section \ref{sec:equ}. This equivalence theorem in double theory is closely related to the descent equations studied in \cite{ANXZ}.

\subsection{Future Outlook} 

The present work motivates to further explore the applications of 2-shifted lagrangians on $\huaB G$ in an upcoming work-in-progress \cite{ABC:lag}.  Quasi-Poisson $(\mathfrak{d}, \g)$-manifolds \cite{libland-severa1},  Poisson homogeneous spaces \cite{Henrique-lu}, and dynamical $r$-matrices \cite{et-var} all provide examples of 2-shifted lagrangians on $\huaB G$. It is proven in \cite{PaToVaVe13} that intersections of $m$-shifted Lagrangians are $(m-1)$-shifted symplectic. Using this principle,  the new interpretation of these objects allows us to provide integration for many Dirac structures. 

One more meaningful work is to figure out the explicit appearance of a cotangent bundle for a Lie $n$-groupoid. Just as the cotangent bundle of a manifold has a natural symplectic structure, \cite{cal:cot, PaToVaVe13} shows that shifted cotangent bundles of a derived higher stack carries a shifted symplectic structure. In contrast to algebraic geometry, where a cotangent bundle can be simply defined as spectrum of symmetric algebra of tangent sheaves \cite[Def 1.20]{PaToVaVe13}, the concept of cotangent bundle of a higher Lie groupoid as higher VB (vector bundle) groupoid can not be directly obtained. There are many Lie $n$-groupoids presenting the same $n$-stack. But we believe there is a canonical VB Lie $n$-groupoid presenting the cotangent bundle as soon as we choose a fixed Lie $n$-groupoid $X_\bullet$ presenting the $n$-stack, just as the tangent bundle is canonically given by the tangent VB $n$-groupoid $TX_\bullet$ (notice that the naive choice $T^*X_\bullet$ does not make sense at all \cite{cotang-bd-mathoverflow}). To figure out all this and the canonical shifted symplectic structure on the cotangent bundle is the aim of the future work \cite{stefano-thesis}. 

Another immediate application of this work is in the famous open problem  of integration of Courant algebroids, which is partially solved in special situations \cite{Sev:exin, MeTa11, MeTa18a, Severa05, sev:int, ShZh17}. In fact, before our work it was not even clear what the integration object should be. 
Now we understand the nature of a Courant algebroid is  a 2-shifted symplectic Lie 2-algebroid \cite{saf:shi, royt}. Thus  we expect a {\em 2-shifted symplectic Lie 2-groupoid} as the integration result of a Courant algebroid.  It was very mysterious in the previous works that people arrived at different integrating symplectic forms even for the same Courant algebroid. After establishing the concept of symplectic Morita equivalence, we might have a chance to explain this non-uniqueness. As Courant algebroid over a point is a quadratic Lie algebra,  this work gives three Morita equivalent solutions to the integration of this case (see Table \ref{tabex}).  We will study the integration of a general Courant algebroids and the relation between the different integrated symplectic forms in a future work. 

A relatively long term aim for us is to use the 2-shifted symplectic form to do a geometric quantization for $\huaB G$. As described in \ref{sec:2nd-model}, $\String(G)_\bullet$ gives the prequantization $S^1$-gerbe of $\huaB G$. The next step will be to equip $\huaB G$ with an associated 2-line bundle via a 2-representation theory of $\String(G)_\bullet$. There are some rough ideas to link this with a certain class of projective unitary representations of the based loop group (also see \cite{huan:22, murray-robert-wockel}). An intermediate step is to construct a TFT, that is to build a 2-functor from the cobodism 2-category to that of 2-shifted symplectic groupoids, 2-shifted Lagrangians, and 1-shifted Lagragians as in \cite{cal-schaumbauer:21, moore-tachikawa}.  The next step is to build a quantizing functor $Q$ to the 2-category of 2-vector spaces via a 2-representation theory of $\String(G)_\bullet$.

Finally, another direction to explore is the shifted Poisson structures on Lie $n$-groupoids. The 1-shifted Poisson Lie groupoids and their Morita equivalences are fully explored in \cite{Xu:ShiftP}.  It is proven in \cite{BuIgSe09} that this concept contains that of 1-shifted symplectic groupoid. In other context, $m$-shifted Poisson NQ-manifolds are studied in \cite{prid:shift-nq}, and $m$-shifted Poisson derived higher stacks are studied in \cite{cal:shi, prid:shift}.  However,  what are $m$-shifted Poisson $n$-groupoids and their Morita equivalences? How do they include $m$-shifted symplectic $n$-groupoids? How do they give rise to  higher shifted Poisson brackets? These are other immediate questions that the Poisson community cares much. Explicit higher shifted Poisson brackets will certainly help for the purpose of quantization.


\subsection*{Acknowledgement}
We thank very much Daniel \'{A}lvarez, Henrique Bursztyn, Alejandro Cabrera,  Zheng Hua,  Jiang-Hua Lu, Eckhard Meinrenken and Ralf Meyer for stimulating discussions.  We are in debt with all the participants of the higher structure seminar at G\"{o}ttigen for providing useful comments in early stages of this work. We also give our warm thank Jonathan Taylor for his help with the English language of this article.   We are supported by DFG ZH 274/3-1,

\section{$m$-Shifted symplectic Lie $n$-groupoids}\label{Sec:sslg}
A {\bf simplicial manifold} $X_\bullet$ is a contravariant
functor from $\Delta$, the category of finite ordinals 
\[
[0]=\{0\}, \qquad [1]=\{0, 1\},\quad \dotsc,\quad
[l]=\{0, 1,\dotsc, l\},\quad\dotsc,
\]
with order-preserving maps, to the category of manifolds\footnote{Here, by manifolds we mean Banach manifolds as in \cite{lang}.}. More precisely, $X_\bullet$ consists of a tower of manifolds $X_l$, face maps $d^l_k: X_l \to X_{l-1}$ for $k=0, \dots, l$, and degeneracy maps $s^l_k: X_l \to X_{l+1}$. These maps satisfy the following simplicial identities
\begin{equation}\label{eq:face-degen}
    \begin{array}{lll}
        d^{l-1}_i d^{l}_j = & d^{l-1}_{j-1} d^l_i &\text{if}\; i<j,  \\
       s^{l}_i s^{l-1}_j =& s^{l}_{j+1} s^{l-1}_i & \text{if}\; i\leq j,
    \end{array}\qquad  d^l_i s^{l-1}_j =\left\{\begin{array}{ll}
    s^{l-2}_{j-1} d^{l-1}_i  & \text{if}\; i<j, \\
    \id  & \text{if}\; i=j, j+1,\\
    s^{l-2}_j d^{l-1}_{i-1} & \text{if}\; i> j+1.
 \end{array}\right.
\end{equation}
We may drop the upper indices and just write $d_i$ and $s_i$ for simplicity when the context is clear later in our article. 

Similarly, we may define other simplicial objects in other categories, such as simplicial sets and simplcial vector spaces, which we will meet in this article. 
The following simplicial sets, which may be viewed as simplicial manifolds with discrete topology (e.g. in Def. \ref{def:lie-n-gpd}), play an important role for us. They are the simplicial {\bf $l$-simplex} $\Delta[l]$ and
the {\bf horn} $\Lambda[l,j]$:
\begin{equation}\label{eq:simplex-horn}
\begin{split}
(\Delta[l])_k & = \{ f: [k] \to [l] \mid f(i)\leq
f(j),
\forall i \leq j\}, \\
(\Lambda[l,j])_k & = \{ f\in (\Delta[l])_k\mid \{0,\dots,j-1,j+1,\dots,l\}
\nsubseteq \{ f(0),\dots, f(k)\} \}.
\end{split}
\end{equation}
In fact the horn $\Lambda[l,j]$ is a simplicial set obtained from the
simplicial $l$-simplex $\Delta[l]$ by taking away its unique
non-degenerate $l$-simplex as well as the $j$-th of its $l+1$
non-degenerate $(l-1)$-simplices, as in the following picture (in
this paper all the arrows are oriented from larger numbers to
smaller numbers): 

\begin{figure}[h]
   \centering
\includegraphics[page=1,width=.75\textwidth]{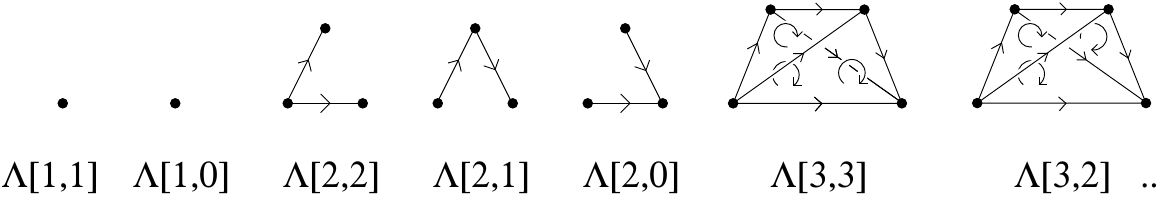}
\end{figure}

\begin{definition} \label{def:lie-n-gpd}
 A \textbf{Lie $n$-groupoid} \cite{Getzler04, henriques, z:tgpd-2} is a simplicial manifold $K_\bullet$ where the natural projections
\begin{equation}\label{eq:horn-proj}
 p_{l,j}:K_l=\Hom(\Delta[l],K_\bullet)\to \Hom(\Lambda[l,j], K_\bullet)
\end{equation}
 are surjective submersions\footnote{Submersion for Banach manifold is also introduced in \cite{lang}. It is sometimes called a split submersion in other literature, namely we need local splitting in charts.} for all $1\le l\ge j\ge 0$ and diffeomorphisms for all $0\le j\le l>n$. It is a \textbf{Lie $n$-group} if $K_0=pt$.  
 
If we replace surjective submersions by submersions, and diffeomorphisms by injective \'etale maps, in the requirement for the natural projections $p_{l,j}$ we say that $K_\bullet$ is a \textbf{local Lie $n$-groupoid}, see \cite[Defi.1.1]{zhu:kan}. Similarly, it is further a \textbf{local Lie $n$-group} if $K_0=pt$. 
\end{definition}

\begin{remark}\label{remark:1-2}
This definition is inspired by \cite{duskin, duskin2} where the author uses Kan complex to understand and define higher (set theoretic) groupoids. The fact that $\Hom(\Lambda[l,j], K_\bullet)$ is again a manifold is not obvious but can be proven inductively. 
We refer to \cite{Getzler04, henriques} or \cite[Sect.1]{z:tgpd-2} for more details on this definition.
\end{remark}

\begin{remark} \label{remark:gpd-stack-morita}
Lie $n$-groupoids may be understood as concrete chart systems for differentiable $n$-stacks \cite{Lesdiablerets, Pridham:higher-stack, prid:shift}. Two Lie $n$-groupoids present the same differentiable $n$-stack if and only if they are Morita equivalent. See Section \ref{sec:me} for details of Morita equivalence for Lie $n$-groupoids and Morita equivalence preserving shifted symplectic structures, which is in turn called symplectic Morita equivalence.  
\end{remark}

\begin{example}[Lie $1$-groupoid]\label{ex:1group}
It is well known that any Lie $1$-groupoid $K_\bullet$ is always the simplicial nerve $NK_\bullet$ of a usual Lie groupoid $K\rightrightarrows M$. Here we recall that $NK_0=M$ and for $i\geq 0$ define 
\[
NK_i := \underbrace{K\times_M K \times_M \dots \times_M K}_{i\;\text{many copies of $K$}}. 
\]
The faces and degeneracies are given by  \begin{equation*}
    \begin{array}{rlrl}
        d_0(k)=&s(k), & d_1(k)=&t(k),\\
        d_0(k_1, \dots, k_i)=&(k_2,\dots ,k_i),& d_j(k_1, \dots, k_i)=&(k_1,\dots,  k_jk_{j+1}, \dots, k_i),\\
        d_i(k_1, \dots, k_i)=&(k_1, \dots, k_{i-1}),
         &  s_j(k_1, \dots, k_i) =& (k_1, \dots, k_j, 1_{s(k_j)}, k_{j+1}, \dots, k_i).
    \end{array}
\end{equation*} 
\end{example}
\begin{example}[Lie 1-group]\label{NG}
A particular case of the previous example is when we consider a Lie group $G$. Its simplicial nerve $NG_\bullet$ is then given by $NG_i=G^{\times i}$ and the faces and generacies are
\begin{equation*}
    \begin{array}{ll}
        d_0(g_1, \dots, g_i)=(g_2,\dots ,g_i),& d_j(g_1, \dots, g_i)=(g_1,\dots,  g_jg_{j+1}, \dots, g_i),\\
        d_i(g_1, \dots, g_i)=(g_1, \dots, g_{i-1}),
         &  s_j(g_1, \dots, g_i) =(g_1, \dots, g_j, e, g_{j+1}, \dots, g_i).
    \end{array}
\end{equation*} 
\end{example}

\begin{example}[Lie $2$-groupoids]\label{ex:2groups}
Just like a Lie 1-groupoid is determined by the first three levels and the faces and degeneracy maps between them\footnote{We need $NG_2$ and $d_1: NG_2\to NG_1$ to obtain the multiplication, for example.}, a Lie 2-groupoid $K_\bullet$ is determined by the first four levels $K_0, K_1, K_2, K_3$, and face and degeneracy between them.  The rest of $K_\bullet$ is given by a coskeleton construction\footnote{See \cite[Section 2.3]{z:tgpd-2} for details of this coskeleton construction. }, that is $K_i$ is made up of those $i$-simplices whose 2-faces are elements of $K_2$ and such that each set of four 2-faces gluing together as a 3-simplex is an element of $K_3$ (see \cite[Sect.2.3]{z:tgpd-2}).   This coskeleton construction also guarantees that $K_i$ are manifolds for all $i\ge 0$ as long as $K_0,\cdots,  K_3$ are. 
\end{example}

\subsection{Tangent complex and de Rham double complex}\label{sec:tan}
We first recall some prelimenary knowledge about Dold-Kan correspondence, which can be found in standard text books (e.g. \cite{wei:hom}). Given a simplicial vector space $V_\bullet$, we recall that  the normalised (or Moore) complex $(NV_\bullet, \partial)$ is given by 
\[
NV_l := \cap_{i=0}^{l-1} \ker d^l_i \subset V_l, \quad \partial:=(-1)^l d^l_l: NV_l \to NV_{l-1}. 
\]
Then the horn projection $p_{l, j}: V_l \to \Hom(\Lambda[l, j], V_\bullet)$ (see \eqref{eq:horn-proj}) in this situation, clearly satisfies $\ker p_{l, l}=NV_l$. On the other hand, we have another natural chain complex $(\bN V_\bullet, \bp)$ associated to $V_\bullet$:
\begin{equation*}
    \bN V_l := V_l/(\im s^{l-1}_0 + \dots + \im s^{l-1}_{l-1} ), \quad \bp:= \sum_{i=0}^l (-1)^i d^l_i. 
\end{equation*}
The map $\bp$ descends to  the quotient $\bN V_l$ thanks to the simplicial identities \eqref{eq:face-degen}. We recall the classical Dold-Kan correspondence\footnote{The Dold-Kan correspondence works more generally for simplicial abelian groups. Here we adapt to the special case of vector spaces for convenience in this article.}. 
\begin{theorem}[Dold-Kan] \label{thm:doldkan} The functor sending $V_\bullet$ to $(NV_\bullet, \partial)$ is an equivalence of categories between simplicial vector spaces and chain complexes of vector spaces concentrated in positive degrees. Moreover, $(NV_\bullet, \partial) \cong (\bN V_\bullet, \bp)$ are isomorphic chain complexes. 
\end{theorem}

For any simplicial manifold $X_\bullet$, we have a simplicial vector bundle over $X_0$ given by,  
\begin{equation*}\label{eq:tang-vs}
\dots T X_2|_{X_0} \aaar    T X_1|_{X_0} \rightrightarrows T X_0 ,
\end{equation*}
where $X_0$ is a submanifold in $X_l$ via iterated degeneracy $X_0\xrightarrow{s_0} X_1 \xrightarrow{s_0} X_2 \dots \xrightarrow{s_0} X_l$ and $TX_i|_{X_0}=s_0^*\circ {\dots}\circ s_0^*TX_i$ is the pullback vector bundle to $X_0$.  The face and degeneracy maps are $T d_i$ and $T s_i$. If $X_\bullet$  is a Lie $n$-groupoid, then the horn projection (2.3) is always a surjective submersion and so $\ker T p_{l,l}$ is a vector bundle. (Note that constant rank guarantees it is a vector bundle.)

\begin{definition}\label{defi:tangent-cx}
The {\bf tangent complex} $(\huaT_\bullet K, \partial)$ of a Lie $n$-groupoid $K_\bullet$ is the following chain complex of vector bundles  over $K_0$: \begin{equation*}\label{eq:tang-cx}
    \huaT_l K := \left\{\begin{array}{ll}
        \ker Tp_{l, l}|_{K_0} & \text{for}\; l>0, \\
         TK_0 & \text{for}\; l=0,\\
         0&\text{for}\; l<0,
    \end{array}\right. 
\end{equation*} with $ \partial_l:=(-1)^l Td^l_l$. We omit the lower index of $\partial_l$ when it is clear from the context. 
We denote by $H_\bullet(\huaT K)$ the {\bf homology groups} of the tangent complex $(\huaT_\bullet K, \partial)$.
\end{definition}
\begin{remark}\label{remark:tang-cx-me}
This notion of tangent complex is invariant up to weak equivalence under Morita equivalence of higher Lie groupoids (see Corollary \ref{cor:tang-cx-we}). Thus it provides a correct notion of ``tangent complex'' for a higher stack presented by a Lie $n$-groupoid. 
\end{remark}

\begin{proposition}\label{tcx-quot}
Let $K_\bullet$ be a Lie $n$-groupoid. For all $l\ge 0$ we have
\begin{equation*}
\huaT_l K\cong TK_l|_{K_0}/\oplus_{i=0}^{l-1}\im(Ts^{l-1}_i), \quad \text{and}\quad\partial_l= \sum_{i=0}^l (-1)^i Td^l_i
\end{equation*}under this isomorphism. Moreover $\huaT_{l}K=0$ if $l>n$ and the rank of $\huaT_l K$ is  \begin{equation*} \label{eq:rank-T}
    \dim K_l- \sum_{i=0}^{l-1}(-1)^{i}\binom{l}{i+1}\dim K_{l-1-i}.
\end{equation*} 
\end{proposition}
\begin{proof}
The isomorphism follows by using Dold-Kan (Theorem \ref{thm:doldkan}) point-wise. Since $p_{l, j}$ is an isomorphism for $l>n$, it is clear that   $\huaT_{l}K=0$ if $l>n$. The rank formula involves very nice but lengthy combinatorics presented in \cite{stefano-thesis}. 
\end{proof}

\begin{remark}\label{remark:tangent-cx}
The tangent complex in quotient form of Proposition \ref{tcx-quot} also appears in \cite{Lesdiablerets, MeTa18b}. Moreover, this complex should host the structure of a Lie $n$-algebroid.  A presheaf version of the statement is proven in \cite{severa:diff}, however whether the desired presheaf is represented by the tangent complex is not proven. Some effort to prove the representibility was done in \cite{li:thesis}, however, Lemma 8.34 therein still contains a flaw in the general case. 
An article fixing all these problems is in preparation \cite{lfrwz}. 

It is clear that the concept of tangent complexes works well for a local Lie $n$-groupoid, and the tangent complex of a local Lie $n$-groupoid is still a complex of vector bundles over the base $K_0$,  since we only have used the submersion condition in Definition \ref{defi:tangent-cx}. 
\end{remark}

\begin{example}\label{eq:tan-1}
Following Example \ref{ex:1group}, let $K\rightrightarrows M$ be a Lie groupoid. Its nerve $NK_\bullet$ is then a Lie $1$-groupoid. In this case the tangent complex is concentrated in degrees $0$ and $1$, and given by its Lie algebroid $A=\ker Ts_{|M}$ and the anchor
\begin{equation*}
  \huaT_i K=\left\{\begin{array}{ll}
      A &\text{for } i=1,   \\
      TM & \text{for } i=0,\\
      0& \text{otherwise}.
  \end{array}\right. \text{with }\partial=\rho\quad \text{and}\quad H_i(\huaT K)=\left\{\begin{array}{ll}
      \ker \rho &\text{for } i=1,   \\
      \coker\rho & \text{for } i=0,\\
      0& \text{otherwise}.
  \end{array} \right. 
\end{equation*}
\end{example}

Since our main motivation is to study symplectic structures, we recall that {\bf differential forms} on a simplicial manifold $X_\bullet$ live in the de Rham simplicial double complex $
(\Omega^\bullet(X_\bullet), d, \delta)$ (see \cite{bss:derham, dup:derham}).
Here $d$ is the usual de Rham differential $d: \Omega^q(X_p)\to \Omega^{q+1}(X_p)$, and the $\delta$ is the simplicial differential $$\delta: \Omega^q(X_{p-1})\to \Omega^q(X_{p}), \quad \delta=\sum_{i=0}^{p} (-1)^i d^*_i.$$
The associated differential on the total complex is $D=\delta+(-1)^p d$. A special important sub-complex is the so called {\bf normalised double complex} 
\[
\hat{\Omega}^\bullet(X_\bullet)=\{\alpha\in\Omega^\bullet(X_\bullet)\ |\ s^*_i\alpha=0, \quad \text{for all possible $i$} \}. 
\] In other words, a differential form $\alpha$ is called {\bf normalised} if it vanishes on degeneracies. An {\bf $m$-shifted $k$-form  } $\alpha_\bullet$ is 
\begin{equation} \label{eq:cut}
    \alpha_\bullet=\sum_{i=0}^m \alpha_i\ \text{ with } \ \alpha_i\in\widehat{\Omega}^{k+m-i}(X_i).
\end{equation}
We say that $\alpha_\bullet$ is {\bf closed} if $D\alpha_\bullet=0$. 


\subsection{Shifted symplectic structures}
Following \cite{Lesdiablerets}, we express the non-degeneracy of an  $m$-shifted $2$-form $\alpha_\bullet$ on a (local) Lie $n$-groupoid $K_\bullet$ in terms of the tangent complex $(\huaT_\bullet K, \partial)$. Just  as on a Lie algebroid there are infinitesimal multiplicative (IM) forms  \cite{BCWZ, hen:mul}, there is a higher version showing up on tangent complexes of Lie $n$-groupoids. Here we introduce part of the IM-form associated to a $m$-shifted $2$-form.

For a Lie $n$-groupoid $K_\bullet$ and a closed $m$-shifted $2$-form $\alpha_\bullet$ we introduce its \textbf{IM-form}\footnote{This nice explicit formula of IM form comes from \cite{Lesdiablerets}. Nevertheless, as also indicated in \cite{BCWZ, hen:mul} in lower cases,  $\lambda^{\alpha_\bullet}$ is probably just the leading term of the full IM-form and does not contain all the infinitesimal information.} $\lambda^{\alpha_\bullet}$ in the following way: for $l\in\Z$ and $v\in (\huaT_l K)_x\subset T_x K_l$, $w\in (\huaT_{m-l}K)_x\subset T_{x} K_{m-l}$ in the tangent complex at $x\in K_0$, we define
\begin{equation}\label{nondeg-pairing}
    \lambda^{\alpha_\bullet}_x( v, w ) := \sum_{\sigma \in \Sh_{l,m-l}} (-1)^\sigma \alpha_m(T(s_{\sigma(m-1)}\dots s_{\sigma(l)} )v, T(s_{\sigma(l-1)} \dots s_{\sigma(0)})w)
\end{equation}
where $\Sh_{l, m-l}$ is the set of $(l, m-l)$-shuffles, and $(-1)^\sigma$ is the permutation sign of $\sigma$.

It is clear from the definition that $\lambda^{\alpha_\bullet}$ is graded anti-symmetric.  Furthermore, $\lambda^{\alpha_\bullet}$ is {\bf infinitesimal multiplicative}\footnote{This together with several other useful statements is stated in \cite{Lesdiablerets} without a proof. We give a proof to them in Appendix \ref{Ap:IM} for the completion of this article. These tasks were given to Florian Dorsch for his master thesis under the guide of the second author.  }, that is  for $u\in \huaT_{l+1}K$, $w\in \huaT_{m-l}K$
\begin{equation}\label{eq:inf-mul}
    \lambda^{\alpha_\bullet}(\partial u, w) + (-1)^{l+1} \lambda^{\alpha_\bullet}( u, \partial w )=0.
\end{equation}
This implies that $\lambda^{\alpha_\bullet}(\cdot,\cdot)$ descends to the homology groups $H_\bullet(\huaT K)$.  Thus we have an induced pairing 
\begin{equation}\label{eq:homo-pai}
    \lambda^{\alpha_\bullet}(\cdot,\cdot):H_l(\huaT K)\times H_{m-l}(\huaT K)\to \R_{K_0}
\end{equation}
that we denote by the same symbol. The $m$-shifted $2$-form $\alpha_\bullet$ is called {\bf non-degenerate} if  \eqref{eq:homo-pai} defines a pointwise non-degenerate pairing.

\begin{remark}\label{remark:local-sslg}
We notice that the non-degeneracy condition may easily adapt to the case of local Lie $n$-groupoids (see Def. \ref{def:lie-n-gpd}) since the non-degeneracy condition only depends on the tangent complex which works well for local Lie $n$-groupoids (see Remark \ref{remark:tangent-cx}). 
\end{remark}

\begin{definition}\label{def:mSSLnG}
Following \cite{Lesdiablerets}, we say that the pair  $(K_\bullet, \alpha_\bullet)$ is an {\bf $m$-shifted symplectic (resp. local) Lie $n$-groupoid} if $K_\bullet$ is a (resp. local) Lie $n$-groupoid and $\alpha_\bullet$ is a closed, normalized and non-degenerate $m$-shifted $2$-form on $K_\bullet$. 
\end{definition}

Since this is one of the main characters of this work, we make a couple of comments that perhaps shed light in the previous definition. 

\begin{remark}[Clarification on \eqref{nondeg-pairing} and \eqref{eq:inf-mul}]\label{remark:extreme-case}
The pairing $\lambda^{\alpha_\bullet}$ was introduced in \cite{Lesdiablerets} and can be understood as part of a Van Est map for Lie $n$-groupoids (see Proposition \ref{VE} and Remark \ref{VE2} for this interpretation in our particular case). Since $K_\bullet$ is a Lie $n$-groupoid, it follows that $\huaT_\bullet K$ is concentrated in degrees $0$ to $n$. Hence, for $l<0$ or $l>m$ we have $v=0$ or $w=0$ respectively, and we define $\lambda^{\alpha_\bullet}_x(v,w):=0$ in these cases. If $l=0$, Equation \eqref{nondeg-pairing} has the simpler expression
$$\lambda_x^{\alpha_\bullet}(v,w)=\alpha_m(T(s_{m-1}\cdots s_0)v, w),$$
and the case of $l=m$ is analogous. The extreme case of $l=m=0$ in \eqref{nondeg-pairing} is given by  $\lambda^{\alpha_\bullet}_x=\alpha_0(v, w).$ Notice also that if $u\in \huaT_0 K$ then \eqref{eq:inf-mul} reads as $\lambda^{\alpha_\bullet}( u, \partial w )=0$, and analogous for $w\in \huaT_0 K.$
\end{remark}

\begin{remark}[Non-degeneracy condition and values of $m$ and $n$] \label{rmk:m-n}
For $K_\bullet$ a Lie $n$-groupoid,  we know  that the only possible non-zero homology groups are concentrated between $0$ and $n$. Suppose that $H_0(\huaT K)\neq 0\neq H_n(\huaT K)$.  Then, due to the non-degeneracy condition from \eqref{eq:homo-pai}, $K_\bullet$ only admits $n$-shifted symplectic structures. This will be the generic situation but other cases can happen.

The cases when $m>n$ are included in $m$-shifted symplectic $m$-groupoids because any Lie $n$-groupoid is also a Lie $m$-groupoid for $m\ge n$. The cases when $m<n$ can be understood as ``adding singularities" to the case of $m$-shifted symplectic $m$-groupoids (see Example \ref{ep:1-2}).
\end{remark}

\begin{remark}
The homology groups $H_\bullet(\huaT K)$ form a bundle of vector spaces over $K_0$, but it might not have a constant rank, and thus is not a vector bundle in general. But this does not prevent us from
describing the above concept of non-degeneracy. In \eqref{eq:homo-pai} the trivial bundle over $K_0$ is denoted by $\R_{K_0}$ and pointwise non-degenerate means  exactly that for each $x\in K_0$ the pairing restricted to the fibre of $x$ is non-degenerate.
\end{remark}

\begin{remark}[Cotangent complex]
A way to reformulate the non-degeneracy is by introducing the {\bf cotangent complex}
$(\huaT^*_\bullet K, \partial^*)$ with 
$$\huaT^*_i K=(\huaT_{-i}K)^*\qquad\text{and}\qquad \partial^*_i=(-1)^{-i+1}\partial^t_{-i+1}.$$
Then $\lambda^{\alpha_\bullet}$ being non-degenerate can be rephrased by saying that the morphism 
\begin{equation*}
    \begin{array}{rc}
         \lambda^{\alpha_\bullet,\sharp}:&(\huaT_\bullet K,\partial)\to (\huaT^*_\bullet K[-m],\partial^t), 
         \quad v \mapsto \lambda^{\alpha_\bullet,\sharp}(v)=\lambda^{\alpha_\bullet}(v,\cdot)
    \end{array}
\end{equation*}
is a quasi-isomorphism of complexes. 
\end{remark}

\subsection{Examples} \label{sec:example-m-n}

Although the focus of this article is not the general theory of $m$-shifted symplectic Lie $n$-groupoids, we provide a couple of examples that illustrate the definition for small $m$ and $n$. The examples will show that $m$-shifted symplectic Lie $n$-groupoids incorporate already well known objects.

\begin{example}[$m=0$, $n=0$]
Recall that a Lie $0$-groupoid is the same as a manifold $M$, with $M_\bullet$ given by $M_i=M$ with faces and degeneracies all equal to the identity. The tangent complex is concentrated in degree $0$ and given by $\huaT_0 M=TM$. Hence $M_\bullet$ only admits $0$-shifted symplectic structures. It is trivial to see that $\omega_\bullet=\omega\in\Omega^2(M)$ is a $0$-shifted symplectic structure if and only if it is symplectic. Thus we obtain the following correspondence
$$\{ \text{ Symplectic manifolds }\ (M,\omega) \}\rightleftharpoons\{ \text{ 0-Shifted symplectic Lie 0-groupoids }\ (M_\bullet,\omega_\bullet) \}.$$

\end{example}

\begin{example}[$m=0$, $n=1$]
By Example \ref{ex:1group} we have a Lie groupoid $K\rightrightarrows M$.
A $0$-shifted 2-form is simply a 2-form $\omega\in \Omega^2(M)$, and $D\omega=0$ if and only if $d\omega=0$ and $s^*\omega=t^*\omega$. Here the normalized condition is vacuous. In Example \ref{eq:tan-1} we computed the tangent complex, and the non-degeneracy is equivalent to 
$$\ker\rho=0 \quad\text{ and }\quad \omega_{|\coker\rho} \text{ non-degenerate},$$ according to Remark \ref{rmk:m-n} and \eqref{eq:homo-pai}.

Since $\ker\rho=0$, there is  a regular distribution on $M$ given by $A=D\subseteq TM$ that carries a non-degenerate $2$-form on the transversal directions. In other words, a $0$-shifted symplectic Lie $1$-groupoid gives a precise mathematical formulation for the sentence: ``The leaf space $F=M/D$ of a regular distribution $D$ on a manifold $M$ is symplectic'', and this is what we shall expect for a symplectic differentiable stack, because $0$-shifted symplectic Lie groupoid should provide an atlas for a symplectic differentiable stack. This concept is studied in \cite{hoffman-sjamaar} under the name of $0$-symplectic groupoids. 
\end{example}

\begin{example}[$m=1$, $n=1$] \label{ep:pre-symp-gpd}
Recall from \cite{BCWZ} that a {\bf twisted presymplectic Lie groupoid}  (a.k.a. quasi-symplectic Lie groupoid in \cite{Xu04}) is a triple $(K\rightrightarrows M, \omega, H)$ where $K\rightrightarrows M$ is a Lie groupoid, and $\omega\in\Omega^2(K)$ and $H\in\Omega^3(M)$ satisfy
\begin{equation} \label{eq:1-1}
    \delta\omega=0, \quad d\omega=\delta H,\quad dH=0\quad\text{and}\quad \ker\omega_x\cap \ker T_xs\cap \ker T_xt=0\quad \forall x\in M.
\end{equation}
As in the previous example we fix $s=d_0$, $t=d_1$, $A=\ker Ts_{|M}$ and $\rho=Tt_{|A}$.
The equation $\delta\omega=0$ implies that $\omega_x=0$ for $x\in M$. Hence $\omega_\bullet:=\omega+H$ is a normalized $1$-shifted $2$-form. The three equations from the left are equivalent to $D\omega_\bullet=0$. Finally, the last condition in \eqref{eq:1-1} can be reformulated as 
\begin{equation}\label{eq:non11}
    \ker\omega_x\cap  A_x\cap \ker \rho_x=0, \quad \forall x\in M.
\end{equation}
Since the tangent complex is also given in Example \ref{eq:tan-1}, we see that \eqref{eq:non11} is equivalent to the non-degeneracy condition for $\omega_\bullet$. So we obtain the following correspondence
$$\{ \text{ Twisted presymplectic Lie groupoids } \}\rightleftharpoons\{ \text{ 1-shifted symplectic Lie 1-groupoids } \}.$$
In particular, the classical notion of {\bf symplectic groupoid} is a particular case of a 1-shifted symplectic Lie 1-groupoid.    
\end{example}

\begin{remark}
In \cite{mat:vb} it was already noticed that the non-degeneracy condition \eqref{eq:non11} is equivalent to saying that  the tangent complex and the cotangent complex are quasi-isomorphic. Moreover they show that \eqref{eq:non11} is equivalent to $\omega^\sharp: TK\to T^*K$ being a Morita equivalence between the VB-groupoids $TK$ and $T^*K$. For more detail see \cite[Section 5.2]{mat:vb}.
\end{remark}

\begin{example}[$m=2$, $n=1$]
In this case, we have a Lie groupoid $K\rightrightarrows M$ endowed with $\omega_\bullet=\omega_2+\omega_1+\omega_0$, where $\omega_i\in\Omega^{4-i}(K_i)$ satisfies $D\omega_\bullet=0$ and $\omega_\bullet$ is normalized. Since the tangent complex is again \eqref{eq:tan-1}, the non-degeneracy condition implies that  $H_0(\huaT K)=0$ (which is equivalent to the fact that $\rho$ is surjective) and that $\lambda^\omega|_{\ker \rho}$ is non-degenerate. In particular, the Lie algebroid must be transitive. 

In the study of Heterotic Courant algebroids, see e.g. \cite{bar:tra}, a particular case of this already appeared. They consider the case when $\omega_0=0$ and the Lie algebroid is the Atiyah algebroid of a principal bundle. In Sections \ref{sec-findim} and \ref{sec:manin-model} we will consider another particular case, namely when $M=pt$. 
\end{example}

\begin{example}[m=1, n=2] \label{ep:1-2}
 Recall that a Poisson manifold $(P,\pi)$ may not always integrate to a symplectic groupoid, but it can always integrate to a Lie 2-groupoid $K_\bullet$ with a symplectic form $\omega\in\Omega^2(K_1)$ as demonstrated in \cite{tz2}. Now we show that this is in fact a 1-shifted symplectic Lie 2-groupoid. 

Recall from \cite{tz2} that for a Poisson manifold $(P, \pi)$, the so-called $A_0$-path space $P_0 T^*P$ is
\[
P_0 T^*P=\{ a\in C^1(I, T^*P) | \pi ( a)= \gamma', \text{with}\; \gamma=p(a) \; \text{the based path of}\; a, a|_{\partial I}=a'|_{\partial I}=0,  \}, 
\]where $p: T^*P\to P$ is the natural projection. Moreover, $P_0 T^*P$ it has a foliation $\huaF$ of codimension $2\dim P$.  Following \cite{tz2}, we construct a Lie  2-groupoid as follows.  Take complete local transversals $U_i$ with respect to $\huaF$, that is, $\sqcup U_i$ intersects with all the leaves. Let $Mon$ denote the monodromy groupoid of the foliation $\huaF$ on $P_0 T^*P$, and $Mon|_{P_0 T^*P \times K_1}$ denotes its source fibre over $K_1=\sqcup U_i$, and $\odot$ denotes the concatenation of two paths.  Then let $K_\bullet$ denote the Lie 2-groupoid, 
$$K_\bullet=\cdots K_1\times_{K_0} K_1 \times_{\odot, P_0 T^*P, \target} Mon|_{P_0 T^*P \times K_1}\aaar\sqcup U_i\rightrightarrows P.$$
We have $\dim K_1 = 2\dim P$, $\dim K_2 = 3 \dim P$.  

On the path space $P_0 T^*P$, we have the 2-form 
\[
\omega(X, Y ):= \int^1_0 \omega_c(X(t), Y(t)), \quad \forall X, Y \in T(P_0T^*P),
\]where $\omega_c$ is the canonical 2-form on $T^*P$. This 2-form restricted on $K_1$ becomes a symplectic 2-form, moreover $\delta \omega =0$ (\cite[Theorem 1.1]{tz2}). It is clear that $\omega$ is normalized. The fact that $\omega$ is also non-degenerate in the sense that its IM form is non-degenerate   follows from Example \ref{ep:pre-symp-gpd} as the IM form only depends on the behavior of $\omega$ in a small neighborhood $U$ of $K_0$ in $K_1$, and there $(U, \omega)$ is a symplectic local Lie groupoid. 
\end{example}

\begin{remark}
 We end this section by pointing that a similar object was introduced in \cite{sev:int}. A  strictly $m$-symplectic Lie $n$-groupoid is a pair $(K_\bullet, \omega)$ where $K_\bullet$ is a Lie $n$-groupoid and $\omega\in\Omega^2(K_m)$ satisfies $D\omega=0$ and $\omega^\sharp: TK_m\to T^*K_m$ isomorphism. Observe that $\omega$ is not  required to be normalized. Moreover, the non-degeneracy condition described by $\omega^\sharp: TK_m\to T^*K_m$ being an isomorphism is neither contained nor included by our non-degeneracy condition.  This can be clearly seen already in the case of $0$-shifted Lie 1-groupoids. 
\end{remark}

\subsection{Symplectic Morita equivalences} \label{sec:me}
In order to introduce  symplectic Morita equivalences among $m$-shifted symplectic Lie $n$-groupoids we need to recall hypercovers and Morita equivalences for Lie $n$-groupoids. We follow \cite[Sect.2]{z:tgpd-2}. 

The {\bf boundary space} or {\bf $i$-matching space} of a simplicial manifold $X_\bullet$ is \begin{equation} \label{eq:boundary}
    \partial_i(X_\bullet):=\Hom(\partial \Delta[i], X_\bullet)=\{(x_0,\cdots, x_i)\in X_{i-1}^{\times i+1}| d_a(x_b)=d_{b-1}(x_a)\ \forall a<b\}.
\end{equation}
 If $K_\bullet, J_\bullet$ are two Lie $n$-groupoids,  then a simplicial morphism $f_\bullet:K_\bullet\to J_\bullet$ is called a {\bf hypercover} if the maps 
\begin{equation} \label{eq:def-hypercover}
     q_i:=((d_0,\cdots, d_i), f_i):K_i\to \partial_i(K_\bullet)\times_{\partial_i(J_\bullet)}J_i
\end{equation}
are surjective submersions for $0\leq i<n$ and an isomorphism for $i=n$. Then $q_i$ is automatically an isomorphism for $i\ge n+1$ \cite[Lemma 2.5]{z:tgpd-2}. Two Lie $n$-groupoids $K_\bullet$ and $J_\bullet$ are called {\bf Morita equivalent} if and only if there is a third Lie $n$-groupoid $Z_\bullet$ with hypercovers $f_\bullet:Z_\bullet\to K_\bullet$ and $g_\bullet:Z_\bullet\to J_\bullet$.  

\begin{remark}
In general, $\partial_i(X_\bullet)$ might not be a manifold any more, even if $X_\bullet$ is a Lie $n$-groupoid, and the most-right-hand-side of \eqref{eq:boundary} is a set-theoretical description for it.  Nevertheless, the right-hand-side of \eqref{eq:def-hypercover} is always a manifold if $f_\bullet$ is a hypercover---even though it seems to be logically dependent, this is something one can prove level-wise inductively (see \cite[Sect.2.1]{z:tgpd-2}) just like the case of horn spaces mentioned in Remark \ref{remark:1-2}. Therefore hypercover is a well-defined concept for Lie $n$-groupoids. 
\end{remark}

Now we prove a statement on the relation between Morita equivalence of Lie n-groupoids and weak equivalence between their tangent complexes. Let $A_\bullet\to M$ and $B_\bullet\to N$ two chain complexes of vector bundles. A chain map $f_\bullet: A_\bullet\to B_\bullet$ consists of  degree-wise vector bundle morphisms $f_i: A_i \to B_i$ over the same base map $f: M\to N$, which in addition commute with the differentials of $A_\bullet$ and $B_\bullet$. A chain map $f_\bullet: A_\bullet\to B_\bullet$ covering a surjective submersion $f:M\to N$ is called a {\bf quasi-isomorphism} if for all points $x\in N$ and any of its preimages $y\in M$, $f_\bullet|_{y}:A_\bullet|_y \to B_\bullet|_x$ is a quasi-isomorphism. We say that $A_\bullet\to M$ and $B_\bullet\to N$ are {\bf weakly equivalent} if there exists another chain complex of vector bundles $C_\bullet\to M'$ and quasi-isomorphisms $f_\bullet:C_\bullet\to A_\bullet$ and $g_\bullet:C_\bullet\to B_\bullet$.

\begin{lemma}\label{lem:h-w}
Given a hypercover $f_\bullet:K_\bullet\to J_\bullet$ of Lie $n$-groupoids, the induced map on their tangent complexes $Tf_i:\huaT_i K\to\huaT_i J$ forms a quasi-isomorphism. In fact, each $Tf_i$ is surjective. 
\end{lemma}
With this lemma (the first part only)\footnote{In the end, a weak equivalence of Lie $n$-groupoids should correspond to a quasi-isomorphism of their corresponding tangent complexes. But hypercovers are special weak equivalences, the second property proven in the above lemma might be the condition that hypercovers additionally correspond to.}, we immediately gain the following. 
\begin{corollary}\label{cor:tang-cx-we}
If two Lie $n$-groupoids $K_\bullet$ and $J_\bullet$ are Morita equivalent, then their tangent complexes $\huaT_\bullet K$ and $\huaT_\bullet J$ are weakly equivalent. 
\end{corollary}

\begin{proof}[Proof of Lemma \ref{lem:h-w}]
Since  $K_0\xrightarrow{f_0}J_0$ is a surjective submersion, we take a point $x_0\in J_0$ and one of its pre-image $y_0\in K_0$. Clearly, after taking tangent functor and fixing base points, $W_\bullet:=T_{y_0}K_\bullet$ and $V_\bullet:=T_{x_0}J_\bullet$ are still Lie $n$-groupoids. In fact, since all face and degeneracy maps come from taking tangent, they are in particular linear. Thus $W_\bullet$ and $V_\bullet$ are actually $n$-groupoid objects in the category of vector spaces with surjective maps as covers (see \cite[Definition 1.3]{z:tgpd-2}), therefore also simplicial vector spaces.   By taking tangent spaces of \eqref{eq:def-hypercover}, we see that
\[
Tq_i|_{y_0}: W_i \to \partial_i(W_\bullet)\times_{\partial_i(V_\bullet)} V_i,
\] is a surjective linear map between vector spaces. Notice that in the category of vector spaces, all finite limits exist, thus $\partial_i(W_\bullet)$ and  $\partial_i(V_\bullet)$ as finite limits are again vector spaces.  Therefore $T_{y_0}f_\bullet: W_\bullet\to V_\bullet$ is a hypercover. By Lemma \ref{lem:hypercover-weak-eq}, $T_{y_0} f_\bullet$ is further a weak equivalence of simplicial vector spaces. Hence, under Dold-Kan correspondence, $NW_\bullet \to NV_\bullet$ is a linear map preserving homology, thus a quasi-isomorphism. Therefore $\huaT_\bullet K\to \huaT_\bullet J$ is a quasi-isomorphism.

Moreover, by \cite[Lemma 2.5(3)]{z:tgpd-2}, $f_{l,l}: \Hom(\Lambda[l, l], K_\bullet) \to \Hom(\Lambda[l, l], J_\bullet) $ is a surjective submersion. Together with other surjective submersions in the condition of hypercover and Lie $n$-groupoids, we have on the tangent level, 
\begin{equation*}
\xymatrix{
    W_l \ar[r]^{Tp^K_{l,l}} \ar[d]_{T_{y_0}f_l}& W_{l,l} \ar[d]^{T_{y_0}f_{l,l}}\\
    V_l \ar[r]^{Tp^J_{l,l}} & V_{l,l}, }
\end{equation*}
all where four maps are surjective. Here $W_{l,l}:=T_{y_0} \Hom(\Lambda[l, l], K_\bullet)$, and $V_{l,l}:=T_{y_0} \Hom(\Lambda[l, l], J_\bullet)$. An easy linear algebra exercise shows that $T_{y_0} f_l: \ker Tp^K_{l,l}\to \ker Tp^J_{l,l}$ is also surjective. 
\end{proof}

Now we give the technical lemma which involves some machinery developed in \cite{Rogers-Zhu:2016} but outside of this article. As these techniques are only used here, we will not recall all the concepts but rather give precise references to \cite{Rogers-Zhu:2016}. 
\begin{lemma}\label{lem:hypercover-weak-eq}
If $f: W_\bullet\to V_\bullet$ is a hypercover of $n$-groupoid objects in vector spaces with surjective maps as covers, then $f$ is a weak equivalence of simplicial vector spaces. 
\end{lemma}
\begin{proof}
Such a hypercover $f$ is a stalk-wise weak equivalence (see \cite[Def.5.1, Cor. 6.3]{Rogers-Zhu:2016}). Then $f$ being a stalk-wise weak equivalence implies that the induced map $\pi_k(f): \pi_k(W, w_0)\to \pi_k(V, v_0)$ is an isomorphism for $v_0\in V_0$, $w_0\in W_0$ (this is an easy implication of \cite[Prop.5.3]{Rogers-Zhu:2016}. There, the pointed case is discussed, however, if we fix base points $w_0$ and $v_0$ similar proofs go through). 
\end{proof}

With all this preliminaries we can return to the world of $m$-shifted symplectic Lie $n$-groupoids. We start with the following observation. 
\begin{lemma}\label{lem:hypercover-pullback}
Let $(K_\bullet, \alpha_\bullet)$ be an $m$-shifted symplectic Lie $n$-groupoid and $f_\bullet: J_\bullet\to K_\bullet$ a hypercover of Lie $n$-groupoids. Then $(J_\bullet, f_\bullet^*\alpha_\bullet)$ is also a $m$-shifted symplectic Lie $n$-groupoid.
\end{lemma}
\begin{proof}
Since $f_\bullet$ is a simplicial morphism we have that  $Df^*_\bullet =f^*_\bullet D$ and $f_\bullet^*s^*_i=s_i^*f^*_\bullet$. Therefore $f^*_\bullet\alpha_\bullet$ is closed and normalized,  thus we only need to show that $f^*_\bullet \alpha_\bullet$ is non-degenerate. As $f_0: J_0\to K_0$ is a surjective submersion, we take again $x_0\in K_0$ and one of its preimage $y_0\in J_0$.  Since $f_\bullet$ is a simplicial morphism, it commutes with all face and degeneracy maps and we have
\begin{equation*}
    \lambda^{f^*_\bullet \alpha_\bullet}_{y_0} (v, w) = \lambda^{\alpha_\bullet }_{x_0} (T_{y_0} f (v), T_{y_0}f (w)).
\end{equation*}
By Lemma \ref{lem:h-w}, the induced map $f_*: H_\bullet(\huaT(K)) \to H_\bullet(\huaT(J))$ is an isomorphism, thus the induced pairing of $\lambda^{f^*_\bullet \alpha_\bullet}_{y_0} $ is the same as that of $\lambda^{\alpha_\bullet }_{x_0} $ under this isomorphism. Therefore $\alpha_\bullet$ is non-degenerate if and only if $f^*_\bullet \alpha_\bullet$ is non-degenerate. 
\end{proof}

\begin{definition}\label{def:sympmorita}
Let $(K_\bullet,\alpha_\bullet)$ and $(J_\bullet,\beta_\bullet)$ be two $m$-shifted symplectic Lie $n$-groupoids. We say that $(K_\bullet,\alpha_\bullet)$ and $(J_\bullet,\beta_\bullet)$ are {\bf symplectic Morita equivalent} if there exists  another Lie $n$-groupoid $(Z_\bullet,\phi_\bullet)$ with $\phi_\bullet$ an $(m-1)$-shifted $2$-form and hypercovers $f_\bullet$, $g_\bullet$  satisfying
\begin{equation}\label{eq:sympme}
    K_\bullet\xleftarrow{f_\bullet}Z_\bullet\xrightarrow{g_\bullet}J_\bullet, \quad \text{and} \quad f^*_\bullet\alpha_\bullet-g^*_\bullet\beta_\bullet=D\phi_\bullet.
\end{equation}
\end{definition}

\begin{remark} \label{rmk:lag}
Notice that Morita equivalence for Lie $n$-groupoids is defined as a zig-zag of hypercovers. Therefore our definition says that the underlying Lie $n$-groupoids are Morita equivalent and we add an extra condition for the match of the symplectic form. This extra condition means that the map $f_\bullet\times g_\bullet:(Z_\bullet, \phi_\bullet) \to (K_\bullet\times J_\bullet, (\alpha_\bullet, -\beta_\bullet) )$ is Lagrangian\footnote{Here we appreciate the private communication with Pavel Safronov.}, as defined in \cite{safro}. Therefore symplectic Morita equivalence provides us examples of {\em shifted Lagrangian} stacks. This aspect is further explored in \cite{ABC:lag}. 
\end{remark}

\begin{proposition}
Symplectic Morita equivalence is an equivalence relation.
\end{proposition}
\begin{proof}
An $m$-shifted symplectic Lie $n$-groupoid $(K_\bullet, \alpha_\bullet)$ is symplectic Morita equivalent to itself via $K_\bullet\xleftarrow{\id} K_\bullet\xrightarrow{\id}K_\bullet$ with $\phi_\bullet=0$. Reflexivity is also clear, therefore transitivity is the only condition left to be checked. In \cite[Section 2.1]{z:tgpd-2}, it is proven that zig-zag of hypercovers is transitive. Hence using result therein, we have that if $(K^1_\bullet, \alpha^1_\bullet)\xleftarrow{f^1_\bullet} (Z^1_\bullet, \phi^1_\bullet)\xrightarrow{f_\bullet^2}(K^2_\bullet, \alpha^2_\bullet)$ and $(K^2_\bullet, \alpha^2_\bullet)\xleftarrow{g^1_\bullet} (Z^2_\bullet, \phi^2_\bullet)\xrightarrow{g^2_\bullet}(K^3_\bullet, \alpha_\bullet^3)$ are two symplectic Morita equivalences then we can form the following diagram, 
\begin{equation*}
    \xymatrix{&&Z^1_\bullet\times_{K^2_\bullet}Z^2_\bullet \ar[dl]_{\widetilde{g}^1_\bullet}\ar[dr]^{\widetilde{f}^2_\bullet}\\
    &Z^1_\bullet\ar[dr]^{f^2_\bullet}\ar[dl]_{f^1_\bullet}&&Z^2_\bullet\ar[dr]^{g^2_\bullet}\ar[dl]_{g^1_\bullet}\\
    K^1_\bullet&&K^2_\bullet&&K^3_\bullet}
\end{equation*}
with all maps being hypercovers. A direct computation shows that 
$$(K^1_\bullet, \alpha^1_\bullet)\xleftarrow{f^1_\bullet\circ \widetilde{g}^1_\bullet} (Z^1_\bullet\times_{K^2_\bullet}Z^2_\bullet, (\widetilde{g}_\bullet^1)^*\phi_\bullet^1+(\widetilde{f}_\bullet^2)^*\phi_\bullet^2)\xrightarrow{g^2_\bullet\circ \widetilde{f}^2_\bullet}(K^3_\bullet, \alpha_\bullet^3)$$
is a symplectic Morita equivalence.
\end{proof}

\begin{remark}
    In fact, more is true:  $m$-shifted sympelctic structures can be transported by Morita equivalences of Lie $n$-groupoids. For this, we need to show that Morita equivalences induce  quasi-isomorphisms between the simplicial graded de Rham complexs. See \cite{weiershausen2025}\footnote{This is based on a sketch of proof in  \cite{Lesdiablerets}.} for a complete proof. 
\end{remark}

\begin{example}[Strict morphisms] \label{ep:strict-morp}
 If $f_\bullet: J_\bullet \to K_\bullet$ is a hypercover of Lie $n$-groupoids, and $\alpha_\bullet$ is an $m$-shifted symplectic form on $K_\bullet$, then  by Lemma \ref{lem:hypercover-pullback} $(J_\bullet, f^*_\bullet \alpha_\bullet)$ is also an $m$-shifted symplectic Lie $n$-groupoid. Moreover  the two $m$-shifted symplectic Lie $n$-groupoids are symplectic Morita equivalent via $(J_\bullet, 0)$ with one leg $f_\bullet$ and the other leg the identity morphism
 \[
 (J_\bullet, f^* \alpha_\bullet) \xleftarrow{\id_\bullet} ( J_\bullet, 0)  \xrightarrow{f_\bullet} (K_\bullet, \alpha_\bullet).
 \]
\end{example}

\begin{proposition}\label{gauge-transformation}
Let $(K_\bullet,\alpha_\bullet)$ be an $m$-shifted symplectic Lie $n$-groupoid and $\phi_\bullet$ an $(m-1)$-shifted $2$-form. Then $(K_\bullet,\alpha_\bullet+D\phi_\bullet)$ is again an $m$-shifted symplectic Lie $n$-groupoid and 
 \[
 (K_\bullet, \alpha_\bullet+D\phi_\bullet) \xleftarrow{\id_\bullet} ( K_\bullet, \phi_\bullet)  \xrightarrow{\id_\bullet} (K_\bullet, \alpha_\bullet)
 \] is a symplectic Morita equivalence. We call this sort of symplectic Morita equivalence a {\bf gauge transformation}. 
\end{proposition}
\begin{proof}
We only need to show that  $\alpha_\bullet+ D\phi_\bullet$ is again $m$-shifted symplectic, and the rest follows similarly to Example \ref{ep:strict-morp}. Obviously $\alpha_\bullet+ D\phi_\bullet$ is still $m$-shifted, closed and normalised. The non-degeneracy condition follows from Lemma \ref{lem:im-form} item \eqref{itm:gauge}.  
\end{proof}

An immediate consequence of Lemma \ref{lem:hypercover-pullback}, Example \ref{ep:strict-morp} and Proposition \ref{gauge-transformation} is that a symplectic Morita equivalence always decomposes into three symplectic Morita equivalences of fixed special types: the first is given by a strict morphism, the second by a gauge transformation and the third by  another strict morphism in the opposite direction. That is, 
\begin{equation}\label{dec-morita}
   \begin{split}
     (K_\bullet, \alpha_\bullet) \xleftarrow{f_\bullet} ( Z_\bullet, \phi)  \xrightarrow{g_\bullet} (J_\bullet, \beta_\bullet)=&\quad (K_\bullet, \alpha_\bullet) \xleftarrow{f_\bullet} ( Z_\bullet, 0)  \xrightarrow{\id_\bullet} (Z_\bullet, f^*\alpha_\bullet)\\
     &\circ (Z_\bullet, f^*\alpha_\bullet) \xleftarrow{\id_\bullet} ( Z_\bullet, \phi)  \xrightarrow{\id_\bullet} (Z_\bullet, f^*\alpha_\bullet+D\phi=\beta_\bullet)\\
     &\circ (Z_\bullet, \beta_\bullet) \xleftarrow{\id_\bullet} ( Z_\bullet, 0)  \xrightarrow{g_\bullet} (J_\bullet, \beta_\bullet).
\end{split} 
\end{equation}

Notice that in \cite{Xu04}, Xu has given a version of Morita equivalence between 1-shifted symplectic Lie 1-groupoids (see Example \ref{ep:pre-symp-gpd}) via Hamiltonian bimodules. The definition is as follows.

\begin{definition}[See Def 3.1, 3.13 and 4.1 in \cite{Xu04}]\label{pingdef}
Two $1$-shifted symplectic Lie $1$-groupoids $(K_\bullet, \alpha_\bullet)$ and $(J_\bullet,\beta_\bullet)$ are symplectic Morita equivalent if there exists a principal bibundle $P$ and $\phi\in\Omega^2(P)$ that makes $P$ into a Hamiltonian space for $K_\bullet$ and $J_\bullet$.
\end{definition}

To closed this section, we prove that our Morita equivalence coincides with the one defined by Xu in \cite{Xu04} when $m=n=1$. 

\begin{theorem}\label{thm:morita-1-1}
The two possible definitions of symplectic Morita equivalence for $1$-shifted symplectic Lie $1$-groupoids define the same equivalence relation.
\end{theorem}
\begin{proof}
Suppose that $(K_\bullet, \alpha_\bullet)$ and $(J_\bullet,\beta_\bullet)$ are symplectic Morita equivalent according to Definition \ref{pingdef}. Then Proposition 3.14 in \cite{Xu04} shows that they are symplectic Morita equivalent according to Definition \ref{def:sympmorita}.

For the converse statement, we use the decomposition \eqref{dec-morita}. It is enough to show that the type of strict morphisms and that of gauge transformations both give rise to symplectic Morita equivalences defined in Definition \ref{pingdef}.  The type of strict morphisms follows from \cite[Proposition 4.8]{Xu04}. The type of gauge transformations follows from \cite[Proposition 4.6]{Xu04}, because our gauge transformations for $m=n=1$ coincide with  gauge transformations of the second type therein. 
\end{proof}

\begin{remark} Symplectic Morita equivalence in Definition \ref{def:sympmorita} provides a shifted Lagragian $n$-stack (see Remark \ref{rmk:lag}).  There is a zoo of natural examples of Hamiltonian principal bimodules in Poisson geometry. They all carry some important physical meaning (see \cite{Xu04, AM:22}).  The above theorem thus implies that each such Hamiltonian principal bimodule gives rise to a shifted Lagrangian stack (c.f. Remark \ref{rmk:lag}). 
\end{remark}

\section{Simplicial models of the classifying stack $\huaB G$ }\label{Sec:3}

Let $G$ be a Lie group with Lie algebra $\g$. We say that $(\g, \langle\cdot, \cdot\rangle)$ is a {\bf quadratic Lie algebra} if $\langle\cdot, \cdot\rangle:\g\times \g\to \R$ is a symmetric and non-degenerate pairing satisfying 
$$\langle[a,b],c\rangle+\langle b,[a,c]\rangle=0\qquad \forall a,b,c\in\g.$$
When $\g$ is quadratic, it was proved in \cite{PaToVaVe13} that the classifying stack $\huaB G$ carries a $2$-shifted symplectic structure.
In this section we present two different shifted symplectic Lie groupoids that encode $\huaB G$ and its shifted symplectic structure. The first one, which we denote by $(NG_\bullet, \Omega_\bullet),$ is a $2$-shifted symplectic Lie $1$-group and it is a finite dimensional model.  The second one, denoted here by $(\GG_\bullet, \omega_\bullet)$, is a $2$-shifted symplectic Lie $2$-group and it is infinite dimensional. In Theorem \ref{thm-morita} we construct an explicit symplectic Morita equivalence between them.

\subsection{Finite dimensional model}\label{sec-findim}
\emptycomment{
If $G$ is a Lie group we have that its simplicial nerve $NG_\bullet$ is a Lie $1$-group. Here we recall that $NG_k=G^{\times k}$ and the faces and degeneracy are 
\begin{equation*}
    \begin{array}{ll}
        d_0(g_1, \dots, g_k)=(g_2,\dots ,g_k),& d_i(g_1, \dots, g_k)=(g_1,\dots,  g_ig_{i+1}, \dots, g_k),\\
        d_k(g_1, \dots, g_k)=(g_1, \dots, g_{k-1}),
         &  s_i(g_1, \dots, g_k) = (g_1, \dots, g_i, e, g_{i+1}, \dots, g_k).
    \end{array}
\end{equation*} 
Moreover it is well known that if $G$ is good enough then the simplicial nerve $NG_\bullet$ appear as a truncation of the universal integration of the Lie algebra $\g$, see also Appendix \ref{ap:int-trun}. More concretely
\begin{proposition}[\cite{henriques} Example 7.2]\label{1-trun}
Let $G$ be the connected and simply connected Lie group integrating $\g$. Then the first truncation $\tau_1(\int \g_\bullet)=NG_\bullet$
\end{proposition}  }

Let $G$ be a Lie group, by Example \ref{NG} we have that $NG_\bullet$ is a Lie $1$-group. If $(\g, \langle\cdot,\cdot\rangle)$ is a quadratic Lie algebra then we can define the following differential forms on $NG_\bullet$:
\begin{equation}\label{Omegadot}
  \Omega=\langle d_2^*\theta^l, d_0^*\theta^r\rangle\in\Omega^2(NG_2)\quad\text{and}\quad    \Theta=\frac{1}{6}\langle\theta^l,[\theta^l,\theta^l]\rangle\in\Omega^3(NG_1)
\end{equation}
where $\theta^l,\theta^r\in\Omega^1(G; \g)$ are the left and right Maurer-Cartan $1$-forms on $G$.

The $2$-form $\Omega$ appeared in the works of Brylinski \cite{brylinski} and Weinstein \cite{Alan:moduli} and we therefore  call $\Omega$ the Brylinski-Weinstein $2$-form.  More precisely, for $(v_1,v_2),(w_1,w_2)\in T_{(g,h)}(G\times G)$ we can write $\Omega$ as
\begin{equation}\label{exOme}
    \Omega_{(g, h)}( (v_1, v_2), (w_1,w_2)) = \langle L_{g^{-1}} v_1, R_{h^{-1}} w_2 \rangle -  \langle L_{g^{-1}} w_1, R_{h^{-1}} v_2 \rangle.
\end{equation}
The $3$-form $\Theta$ is known as the Cartan $3$-form (see \cite{AlBuMe09, dw}). For $u,v,w\in T_g G$ we can express it as
\begin{equation*}
    \Theta_g(u,v,w)=\langle L_{g^{-1}}u, [L_{g^{-1}}v, L_{g^{-1}}w]\rangle.
\end{equation*}

Using the language of shifted symplectic groupoids, the results of \cite{brylinski, Alan:moduli} can be summarized in the following.

\begin{theorem}\label{Thm-fin-sym}
Let $G$ be a Lie group with a quadratic Lie algebra $(\g,\langle\cdot,\cdot\rangle)$. The pair $(NG_\bullet, \Omega_\bullet)$ is a $2$-shifted symplectic Lie $1$-group where $NG_\bullet=\cdots G\times G\aaar G\rightrightarrows pt$ is the nerve of the Lie group $G$ and $\Omega_\bullet=\Omega-\Theta+0$ is given by \eqref{Omegadot}.
\end{theorem}
\begin{proof}
We need to check that $\Omega_\bullet$ is normalized, non-degenerate and closed.  Clearly,   $s^*_0\Theta=0$ and $$(s^*_0\Omega)_g(v,w)=\Omega_{(e,g)}((0_e,v),(0_e,w))=0, \quad \forall\ g\in G,\ v,w\in TG$$
where the zero was deduced by the explicit formula \eqref{exOme}.  The computation for $s_1$ is analogous.

The tangent complex of $NG_\bullet$ is concentrated in degree $1$ and given by $\huaT_1 NG=\g$. Therefore the only non trivial homology group is $H_1(\huaT NG)=\g$ itself. By equation \eqref{nondeg-pairing}, we have
$$\lambda^{\Omega_\bullet}( v,w)=\Omega_{(e,e)}((v,0_e),(0_e,w))-\Omega_{(e,e)}((0_e,v),(w,0_e))=2\langle v,w\rangle,$$ which is non-degenerate due to the non-degeneracy of $\langle\cdot,\cdot\rangle$.

In order to show that $\Omega_\bullet$ is closed we recall that the Maurer-Cartan forms satisfy $$\theta^r=\phi\theta^l,\quad d\theta^l=-\frac{1}{2}[\theta^l,\theta^l], \quad d\theta^r=\frac{1}{2}[\theta^r,\theta^r] \quad\text{and}\quad d_1^*\theta^l=d_0^*\phi^{-1}d_2^*\theta^l+d_0^*\theta^l, $$ 
where $\phi\in\Omega^0(G,\End(\g))$ is the adjoint representation. We  then have that
$$d\Theta=\frac{1}{12}\langle[\theta^l,\theta^l],[\theta^l,\theta^l]\rangle=0$$
by $ad$-invariance of the pairing and the graded version of Jacobi identity, and 
\begin{equation*}
        \delta\Theta=\frac{1}{6}\sum_{i=0}^2(-1)^i\langle d_i^*\theta^l,[d_i^*\theta^l,d_i^*\theta^l]\rangle
        =-\frac{1}{2}\langle d^*_2[\theta^l,\theta^l],d_0^*\theta^r\rangle-\frac{1}{2}\langle d^*_2\theta^l,d^*_0[\theta^r,\theta^r]\rangle=d\Omega, 
\end{equation*}
 where we have used that the pairing is ad-invariant and finally
\begin{equation*}
    \delta\Omega=\sum_{i=0}^3(-1)^i\langle d^*_id_2^*\theta^l,d_i^*d_0^*\theta^r\rangle=0
\end{equation*}
by the simplicial identities.
\end{proof}

It is well known that the Morita equivalence class of the Lie 1-groupoid $NG_\bullet$ gives rise to the differentiable stack $\huaB G$ which plays the role of classifying space for a Lie group $G$. In fact, the geometric realisation of $NG_\bullet$ is a topological classifying space $BG$ of $G$ (see e.g. \cite{bss:derham}), and the stack $\huaB G$ has analogous properties of $BG$ in the world of stacks. Therefore the previous Theorem \ref{Thm-fin-sym} gives a geometric model for a symplectic structure on the classifying stack $\huaB G$. We also point out that the de Rham simplicial double complex $(\Omega^\bullet(NG_\bullet),\delta, d)$  calculates the singular cohomology of the topological classifying space $BG$ (see e.g.  \cite{bss:derham, Segal:loop}), that is 
\[
H^k(Tot(\Omega^\bullet(NG_\bullet)),D)=H^k(BG, \mathbb{R}).
\]

\begin{proposition}
If $G$ is compact, connected, and simple, then $H^4(BG, \R)=\R$, and the $4$-cocycle $\Omega_\bullet$ represents a generator. 
\end{proposition}
\begin{proof}
It is a classical result that $H^4(BG, \R)=H^3(G, \R)=\R$ if $G$ is compact, connected, and simple (see e.g. \cite[Theorems 1.47, 1.81]{tanre}). The cocycle $\Omega_\bullet$ is not exact because the Cartan $3$-form $\Theta$  is not an exact 3-form on $G$. In fact, $\Theta$ represents a generator of $H^3(G, \R)=\R$ in this case. Therefore   $\Omega_\bullet$ represents a generator. 
\end{proof}
\begin{remark}
According to \cite{Lesdiablerets}, in his thesis \cite{shulman:phd} Shulman gives an approach to Chern-Weil
theory using the de Rham complex of $NG_\bullet$. These classes are explicitly calculated therein. 
\end{remark}

\emptycomment{
If we pick $G$ to be the connected and simply connected integrationg of $\g$ then $NG_\bullet$ appears as the first truncation of the universal integration of $\g$,  i.e. $\tau_1(\int \g_\bullet)=NG_\bullet$ (see Appendix \ref{ap:int-trun} for details). This property characterize $NG_\bullet$ among all the Morita equivalent groupoids that represent $BG$.  }

Another relevant property of the shifted symplectic structure $\Omega_\bullet$ is that it provides an integration of the quadratic structure on the Lie algebra. More concretely, in \cite{ArCr11} it was proved that the classical Van Est map relating Lie group cohomology and Lie algebra cohomology can be extended to a map of double complexes
\begin{equation}\label{VE-map}
    VE:\big(\widehat{\Omega}^q(NG_p), \delta, d\big)\to\big(W^{p,q}\g, d^h, d^v\big)
\end{equation}
where $W^{p,q}\g=\wedge^{q-p}\g^*\otimes S^q\g^*$ is a bi-graded version of the Weil algebra of $\g$ (see \cite{ArCr11} for more details).

\begin{remark}\label{graded1}
The Van Est map has a nice interpretation in terms of graded manifolds. Let us denote by $(\g[1], Q)$ the degree $1$ $Q$-manifold given by the Lie algebra $\g$, where the function ring $\huaO_{\g[1]}^i=\wedge^i\g^*$ and $Q=d_{CE}$ is the Chevalley-Eilenberg differential. Then the Weil algebra becomes the algebra of differential forms on $\g[1]$. In fact, we have $$W^{p,q}=\Omega^q(\g[1])_p,\quad d^h=\huaL_Q,\quad d^v=d$$ where $\huaL_Q$ is the Lie derivative with respect to the vector field $Q$ and $d$ is the de Rham differential on $\g[1]$ (see \cite{mad:mod, rajQalg} for more details).   
\end{remark}

\begin{proposition}\label{VE}
We use the same notation in Theorem \ref{Thm-fin-sym}.  The pairing $\langle \cdot,\cdot\rangle$ is an element in $S^2\g^* = W^{2, 2} \g$, and the Van Est map \eqref{VE-map} satisfies
$$VE(\Omega_\bullet)=-\langle \cdot,\cdot\rangle. $$
\end{proposition}
\begin{proof}
Since $\Theta\in\Omega^3(NG_1)$ we have that $VE(\Theta)=0$ for dimensional reasons. Following  \cite{ArCr11} we see that if $v,w\in\g$ then $VE(\Omega)$ is explicitly given by
\begin{equation*}
        VE(\Omega)(v,w)=  - s^*_0i_{\widehat{v}} s_0^*i_{\widehat{w}}\Omega=-(s_0^*i_{\widehat{w}}\Omega)_e(v)=-\Omega_{(e,e)}((w,0),(0,v))=-\langle w,v\rangle, 
\end{equation*}
where $\widehat{v}$ and $\widehat{w}$ denote the associted left invariant vector fields on $G$ and $G^{\times 2}$ respectively.
\end{proof}

\begin{remark}
As indicated in the introduction, the pairing $\langle\cdot,\cdot\rangle$ on $\g$ can be interpreted as a degree $2$ symplectic 
form in the graded manifold $\g[1]$. Combining the previous Proposition \ref{VE}, and using the graded description of the Weil algebra (see Remark \ref{graded1}), we obtain that $(\g[1], Q=d_{CE}, \langle\cdot,\cdot\rangle)$ is a degree $2$ symplectic $Q$-manifold. Therefore, we have proved that $(NG_\bullet, \Omega_\bullet)$ is an integration for $(\g[1], Q=d_{CE}, \langle\cdot,\cdot\rangle)$. 
\end{remark}

We finish this section by answering the following question: Can we find a symplectic form $\Omega'\in\Omega^2(NG_2)$ satisfying $VE(\Omega')=\langle\cdot,\cdot\rangle$ and $\delta\Omega'=0$?

If $\g$ is semisimple and $G$ is a compact and connected integration, then it is impossible. Since $G$ is compact then $G\times G$ admits a symplectic structures if and only if $H^2(G\times G)\neq 0$. But by the compactness and connectedness, we have $H^i(G)=H^i_{CE}(\g)$. Moreover, since $\g$ is semisimple,  Whitehead's Lemma implies that $H^1_{CE}(\g)=H^2_{CE}(\g)=0$, therefore $H^2(G\times G)=0$. Thus there does not even exist a symplectic form $\Omega'$ on $G\times G=NG_2$. Nevertheless, on a natural infinite dimensional model of $\huaB G$, we do find a single 2-form to realise our desired 2-shifted symplectic 2-form. 

\subsection{Infinite dimensional model}\label{Sec:infdim}
There is a natural infinite dimensional Lie $2$-group $\GG_\bullet$ for a given (finite dimensional) Lie group $G$, which is the 2-truncation of the universal integration of the Lie algebra $\g$. We refer to Appendices \ref{ap-sob}, \ref{ap:int-trun} for the technical aspects of infinite dimensional spaces and the details of the truncation procedure. Here we simply describe $\GG_\bullet$ in an explicit way. Recall from Example \ref{ex:2groups} that, in order to define a Lie $2$-group, it is enough to specify the first four levels and the faces and degeneracy between them. 
Let $G$ be a Lie group. Then the spaces $\GG_k$ are given by
\begin{itemize}
    \item $\GG_0=pt,$
    \item $\GG_1=P_e G$,  so $\GG_1$ consists of paths of a fixed Sobolev class $r$ on $G$ starting at $e$,
    \item $\GG_2= \Omega G$, so elements of $\GG_2$ are based loops of Sobolev class $r$,
    \item $\GG_3=\{(\tau_0,\tau_1,\tau_2)\in \GG^{\times 3}_2 \ | \ \tau_0(t)=\tau_1(t),\ \tau_1(\frac{2}{3}+t)=\tau_2(\frac{2}{3}+t),\ \tau_0(1-t)=\tau_2(t)\quad t\in[0,\frac{1}{3}]\}$. Visually,  elements of $\GG_3$ are empty tetrahedrons on $G$ based at $e$. 
\end{itemize}
A more compact notation for those spaces is
\begin{equation*}
    \GG_\bullet=\cdots\ \Omega G\aaar P_eG\rightrightarrows pt.
\end{equation*}
The  face maps $d_i:\Omega G \to P_eG$ for $i=0,1,2$ are defined as
\begin{eqnarray}\label{2faces}
d_0(\tau)(t)=\tau(\frac{t}{3}), \quad d_1(\tau)(t)=\tau(1-\frac{t}{3}), \quad d_2(\tau)(t)=\tau(\frac{1}{3}+\frac{t}{3})\cdot \tau(\frac{1}{3})^{-1}, 
\end{eqnarray} for $t\in [0, 1]$ and  $d_i:\GG_3\to \Omega G$ for $i=0,1,2,3$ are given by
\begin{equation}\label{3faces}
    d_i(\tau_0,\tau_1,\tau_2)=\tau_i\quad\text{with}\quad \tau_3(t)=\left\{\begin{array}{rl}
         \tau_0(t+\frac{1}{3})\cdot\tau_0(\frac{1}{3})^{-1} & \text{for } t\in[0,\frac{1}{3}],\\
         \tau_2(t)\cdot\tau_0(\frac{1}{3})^{-1}&\text{for } t\in[\frac{1}{3},\frac{2}{3}],\\
         \tau_1(\frac{4}{3}-t)\cdot\tau_0(\frac{1}{3})^{-1}&\text{for } t\in[\frac{2}{3},1].
    \end{array}\right.
\end{equation}
The degeneracy maps $s_i: P_eG \to \Omega G$  for $i=0, 1$ are defined as follows:
\begin{equation}\label{1degeneraci}
    s_0(\gamma)(t)=\left\{
    \begin{array}{ll}
         \gamma(3t) & t\in[ 0,\frac{1}{3}]  \\
          \gamma(1) & t\in[\frac{1}{3},\frac{2}{3}]\\
          \gamma(3-3t)&t\in[\frac{2}{3},1], 
     \end{array}\right.
\quad 
    s_1(\gamma)(t)=\left\{\begin{array}{ll}
         e &t\in[ 0,\frac{1}{3}]  \\
          \gamma(3t-1)&  t\in[\frac{1}{3},\frac{2}{3}]\\
          \gamma(3-3t)& t\in[\frac{2}{3},1].
     \end{array}\right.
\end{equation}

\begin{remark}
Notice that even if we begin with smooth maps, our construction of face or degeneracy maps via concatenation does not end up always with a smooth map again, but rather a map in the Sobolev completion.
\end{remark}

\emptycomment{
In order to ensure that $\GG_\bullet$ is a Lie $2$-group we use the following result on truncations of the universal integration of $\g$, see Appendix \ref{ap:int-trun} for more details.
\begin{proposition}[\cite{henriques}]\label{2truncation-thm}
Let $G$ be the connected and simply connected Lie group integrating $\g$. Then the Lie 2-group $\tau_2(\int\g_\bullet)=\GG_\bullet$ where $\GG_k=\Hom_e( sk_1 \Delta^k, G)$, which is the space of maps of Sobolev class  $r+1$ sending $(0, \dots, 0, 1)\in sk_1 \Delta^k$ to $e$.  The face and degeneracy maps are the one induced from the simplices.
\end{proposition}
\begin{proof}
The calculation is more or less done in \cite[Sect.8]{henriques}, we summarise it here. Denote by $\Hom_e(\Delta^k, G)$ the space of maps of Sobolev class ${r+1}$ which send $(0, \dots, 0, 1)\in \Delta^k$ to $e\in G$. When $k=1$ we obtain $P_eG =\Hom_e(\Delta^1, G)$. According to \cite[Example 5.5]{henriques} or \cite[Remark 3.8]{brahic-zhu}, we have
$\Hom_{\algd}(T\Delta^k, \g) = \Hom_e(\Delta^k, G) $, thus $\int \g_1=P_e G$. Finally, the equivalence relation $\sim_k$ that we take in the truncation is exactly the homotopy relative to the boundary in $G$. Thus $\GG_k=\Hom_e( sk_1 \Delta^k, G)$. 
\end{proof}
}
If $G$ is a Lie group, then the tangent bundle $TG$ is also a Lie group. The same is true for Lie $n$-groupoids: if $X_\bullet$ is a Lie $n$-groupoid, then $TX_\bullet$ equipped with tangent face and degeneracy maps is again a Lie $n$-groupoid because  tangent functor keeps the property of being a surjective submersion and an isomorphism. Hence  $T\GG_\bullet$ is a Lie $2$-group
$$T\GG_\bullet=\cdots \Omega TG\ \aaar P_{(e,0_e)} TG\rightrightarrows pt$$
with faces and degeneracy maps $Td_i$ and $Ts_i$. In Appendix \ref{GGtanget} we give explicit formulas for the faces and degeneracies since they are important for computations. The tangent complex $(\huaT_\bullet \GG,\partial)$ is given by
\begin{equation}\label{TcomplexGG}
    \huaT_k \GG=\left\{
    \begin{array}{ll}
        P_0\g=\{u:I\to \g\ | \ u(0)=0\}& \text{for } k=1, \\
        \Omega \g=\{a: S^1\to \g \ | \ a(0)=0\} &\text{for } k=2, \\
        0&\text{otherwise}, 
    \end{array}\right.
\end{equation}
and $\partial:\Omega\g\to P_0\g$ the natural inclusion. To compute the homology groups $H_*(\huaT_\bullet\GG)$, we observe that  $0\to \Omega\g\xrightarrow{\partial} P_0\g\xrightarrow{ev_1}\g\to 0$ is an exact sequence, where $ev_1$ denotes the evaluation at time 1.   Thus
\begin{equation}\label{homgg}
H_1(\huaT_\bullet \GG)=\g \qquad\text{and}\qquad H_i(\huaT_\bullet \GG)=0 \ \text{for all $i\neq 1$.}    
\end{equation}

\begin{remark}
In this article we do not need the Lie $2$-algebra structure on the tangent complex but we included for completeness (see e.g. \cite{Baez:from}). It is given by $\mathbb{g}=(\Omega\g\xrightarrow{\partial} P_0\g, [\cdot,\cdot],[\cdot,\cdot,\cdot])$, where
$$[u,v](t)=[u(t),v(t)], \quad [u,a](t)=[u(t),a(t)],\quad \text{and} \quad [u,v,w](t)=0,$$
for $u,v,w\in P_0\g$ and $a\in \Omega\g.$
\end{remark}

In the work of Segal (see e.g. \cite{Segal:loop}), it is  shown that if in addition $(\g, \langle\cdot,\cdot\rangle)$ is a quadratic Lie algebra then the based loop group $\Omega G$ carries a $2$-form given by
\begin{equation}\label{segal-2form}
    \omega_\tau(a, b)=\int_{S^1}\langle \widehat{a}'(t),\widehat{b}(t)\rangle\ dt\ \in \Omega^2(\Omega G),
\end{equation}
where $\tau\in\Omega G,\ a, b\in T_\tau\Omega G,\ \widehat{a}(t)=L_{\tau(t)^{-1}}a(t), \widehat{b}(t)=L_{\tau(t)^{-1}}b(t)$ and the prime denotes time derivative as usual. So we have all the necessary ingredients to show the following result. 

\begin{theorem}\label{thm-inf-sym}
Let $G$ be a Lie group with quadratic Lie algebra $(\g,\langle\cdot,\cdot\rangle)$. The pair $(\GG_\bullet,\omega_\bullet)$ is a $2$-shifted symplectic Lie $2$-group where $\GG_\bullet=\cdots \Omega G\aaar P_eG\rightrightarrows pt$ is the $2$-group described by \eqref{2faces}, \eqref{3faces} and \eqref{1degeneraci} and $\omega_\bullet=\omega+0+0$ with $\omega$ given by \eqref{segal-2form}.
\end{theorem}
\begin{proof}
We start by showing that $\omega_\bullet$ is normalized. Here we prove that $s_0^*\omega=0$; the proof for $s_1$ is analogous. Using the equations \eqref{Tdegeneracies},  for $\gamma\in P_eG,\ u,v \in T_\gamma P_eG$, and $\widehat{u}(t)=L_{\gamma(t)^{-1}}u(t), \ \widehat{v}(t)=L_{\gamma(t)^{-1}}v(t)$ we have
\begin{equation*}
    \begin{split}
        (s_0^*\omega)_\gamma(u,v)=& \ \omega_{s_0(\gamma)}(Ts_0(u),Ts_0(v)) \\
        =& \int_0^{\frac{1}{3}}\langle \widehat{u}'(3t),\widehat{v}(3t)\rangle dt
        +\int^{\frac{2}{3}}_{\frac{1}{3}}\langle \widehat{u}'(1),\widehat{v}(1)\rangle dt +\int_{\frac{2}{3}}^1\langle \widehat{u}'(3-3t), \widehat{v}(3-3t)\rangle dt \\
        =&\frac{1}{3}\int_0^1\langle \widehat{u}'(s),\widehat{v}(s)\rangle ds-\frac{1}{3}\int_0^1\langle \widehat{u}'(s),\widehat{v}(s)\rangle ds=0.\\
    \end{split}
\end{equation*}
In order to see that $\omega_\bullet$ is non-degenerate, we first observe that for $u,v\in P_0\g=T_{\underline{e}}P_eG$, with $\underline{e}$ the constant loop at the identity $e\in G$,  
\begin{equation*}
\begin{split}
    \omega_{\underline{e}}( Ts_0(u), Ts_1(v))=&\int_0^{\frac{1}{3}}\langle u'(3t), v(0)\rangle dt+\int_{\frac{1}{3}}^{\frac{2}{3}}\langle u'(1), v(3t-1)\rangle dt\\
    &+\int_{\frac{2}{3}}^1\langle u'(3-3t), v(3-3t)\rangle dt=\int_{\frac{2}{3}}^1\langle u'(3-3t), v(3-3t)\rangle dt.
\end{split}
\end{equation*}
Now we use the tangent complex \eqref{TcomplexGG} and its homology groups \eqref{homgg} to directly  apply \eqref{nondeg-pairing}. Then  for $u,v\in\huaT_1\GG= P_0\g$, the IM-form reads,
\begin{equation*}
    \begin{split}
        \lambda^{\omega_\bullet}( u,v)= & \omega_{\underline{e}}( Ts_0(u), Ts_1(v))- \omega_{\underline{e}}( Ts_1(u), Ts_0(v))\\
        =&\int_{\frac{2}{3}}^1\langle u'(3-3t), v(3-3t)\rangle dt+\int_{\frac{2}{3}}^1\langle v'(3-3t), u(3-3t)\rangle dt
        =\langle u(1),v(1)\rangle.
    \end{split}
\end{equation*}
Since the projection $P_0\g \xrightarrow[]{ev_1} H_1(\huaT\GG) = \g$ is evaluation at time 1,  we conclude that $\lambda^{\omega_\bullet}(-, -)$ is non-degenerate in the homology groups.

Finally the fact that $D\omega_\bullet=0$ will follow from Theorem \ref{thm-morita} which states $$\omega_\bullet=\frac{1}{2}ev^*_\bullet\Omega_\bullet+D\omega^P_\bullet,$$
and the fact that $D\Omega_\bullet=0$ by Theorem \ref{Thm-fin-sym}.
\end{proof}

\begin{remark}
One can  directly compute $D\omega_\bullet=0$ by showing that $d\omega=0$ and $\delta\omega=0$. The equation $d\omega=0$ it is well known and goes back to the works of Segal \cite{Segal:loop}. The second equation $\delta\omega=0$ was never computed (as far as we know) and can be directly verified by using the face maps \eqref{3faces} and \eqref{T3faces}. One divides each loop into $3$ pieces, and an explicit calculation shows that 
\begin{equation*}
    \delta\omega_{(\tau_0,\tau_1,\tau_2)}\big((a_0,a_1,a_2),(b_0,b_1,b_2)\big)=
     \int_0^1\langle  \widehat{a}_3'(s),R_{\tau_0(\frac{1}{3})^{-1}}b_0(\frac{1}{3})\rangle ds,
 \end{equation*}
which is in turn 0 due to  Stokes and the fact that $\widehat{a}_3(0)=\widehat{a}_3(1)=0.$ We include the explicit steps in  Prop. \ref{P-Closed} of the appendix. 
\end{remark}

Before ending this section we point out some properties of $(\GG_\bullet,\omega_\bullet)$. As far as we know, Theorem \ref{thm-inf-sym}  was not appearing in the literature before. In other words, the fact that Segal's $2$-form $\omega$ is multiplicative was not detected till now.

As we know classically,  $(\Omega G, \omega)$ is a symplectic manifold, in particular  following from  left-invariance of the $2$-form $\omega$. If we compare with the definition of a shifted symplectic structure, the fact that $\omega$ is directly symplectic is much stronger. Moreover if we further have an integral condition, that is $[\omega]\in H^2(\Omega G,\mathbb{Z})$, then the canonical prequantum $S^1$ bundle $\widetilde{\Omega G}$ also gives rise to a Lie $2$-group. Explicitly, the $S^1$ principal bundle over $\Omega G$ with curvature $\omega$ is the $S^1$-central extension $\widetilde{\Omega G}$ of $\Omega G$ classified by $[\omega]$. As proved in \cite{Baez:from, henriques}, there is a Lie 2-group structure on the simplicial manifold
\begin{equation} \label{eq:stringG}
\String(G)_\bullet=...\widetilde{\Omega G} \aaar P_eG\rightrightarrows  pt,
\end{equation}
and it gives a model of the string group $\String(G)$ associated to $G$. Thus in this way, $\String(G)_\bullet$ provides a prequantization of $(\GG_\bullet, \omega_\bullet)$. In Theorem \ref{thm-morita},  we will show that $(\GG_\bullet, \omega_\bullet)$ is a model for $\huaB G$ and its symplectic structure. Therefore $\String(G)_\bullet$ may be viewed as a prequantization of the symplectic structure on the classifying space of $G$. It is expect that the quantization of $\huaB G$ and the fully extendend $3d$ Chern-Simons theory are related \cite{freed:rmk, hen:what}. For the relation between $3d$ Chern-Simons theory and  $\String(G)$ see \cite{car:cs, dom:cs, stolz:ell, wal:mul, waldorf:string-CS}.

Another interesting property of $(\GG, \omega_\bullet)$ is the following analogy. Denote by $\Delta^i$ the geometric $i$-simplex and identify $\Delta^1\cong [0,1]$ and $\Delta^2\cong D$ the unit disk on $\R^2$. Let $(P,\pi)$ be an integrable Poisson manifold with symplectic groupoid $(K\rightrightarrows P, \omega_\pi)$. Using symplectic reduction techniques, the symplectic manifold $(K,\omega_\pi)$ (the first level of the 1-shifted Lie 1-groupoid) was obtained in \cite{CaFe01} as the reduced phase space of the $2d$ Poisson sigma model on the manifold $\Delta^1\times [0,1]$.   Now as shown in \cite[Section 2.1]{sei:cs} and \cite{Mein:flat}, the manifold $(\Omega G, \omega)$ (the second level of a 2-shifted Lie 2-groupoid) is obtained via symplectic reduction as the reduced phase space of the $3d$ Chern-Simons  sigma model on $\Delta^2\times [0,1]$. We shortly recall the $3d$ case here: For a $3$-dimensional manifold $N$, the Chern-Simons action is given by $$S(A)=\int_N \frac{1}{2}\langle A, dA\rangle+\frac{1}{6}\langle A, [A,A]\rangle\quad \text{with}\quad A\in\Omega^1(N,\g).$$
For $\Sigma$ a surface with boundary, the Chern-Simons theory on $N=\Sigma\times [0,1]$ gives a symplectic manifold $\huaA_{\Sigma}=\{$Connections on $G$-principal bundles with base $\Sigma\}$ with a Hamiltonian action of the gauge group $\huaG_\Sigma=\{g:\Sigma\to G\ |\ g(\partial\Sigma)=e\}$ and moment map $\mu:\huaA_\Sigma\to(\g_\Sigma)^*$
    \begin{equation*}
       \mu(A)(\alpha)=\int_\Sigma \langle\alpha, dA+\frac{1}{2}[A,A]\rangle, \qquad A\in\Omega^1(\Sigma;\g), \ \alpha\in\Omega^0(\Sigma,\g).
    \end{equation*}
The symplectic quotient $\huaA_\Sigma//\mu^{-1}(0)$ is the moduli space of flat connections on $\Sigma$. If  the surface is a disk, i.e. $\Sigma=D$ it was shown in \cite[Example 3.1]{Mein:flat} that the moduli space of flat connections is the based loop group with the Segal's $2$-form, more precisely $$\huaA_D//\mu^{-1}(0)=(\Omega G,\omega).$$


\subsection{The symplectic Morita equivalence}
For a Lie group $G$ with a quadratic Lie algebra $(\g,\langle\cdot,\cdot\rangle)$ we have constructed two different $2$-shifted symplectic Lie $n$-groups: the Lie $1$-group $(NG_\bullet, \Omega_\bullet)$ and the infinite dimensional Lie $2$-group $(\GG_\bullet,\omega_\bullet)$. In this section we will show that if $G$ is connected and simply connected then  $(NG_\bullet, \Omega_\bullet)$ and $(\GG_\bullet,\omega_\bullet)$ are symplectic Morita equivalent. 

\begin{lemma}\label{evaluation}
The evaluation map $ev_1: P_eG \to G$ defined by evaluating at time $1$, that is $ev_1(\gamma)=\gamma(1)$, extends to a simplicial morphism 
\begin{equation*} \label{eq:ev}
    ev_\bullet: \GG_\bullet \to NG_\bullet.
\end{equation*}
\end{lemma}
\begin{proof}
Define $ev_2: \Omega G \to G\times G$ and $ev_3: \GG_3 \to G^{\times 3}$ by
\begin{equation*}
    ev_2(\tau)=\big(\tau(\frac{2}{3})\tau(\frac{1}{3})^{-1}, \tau(\frac{1}{3}) \big), \quad ev_3(\tau_0, \tau_1, \tau_2)=(\tau_2(\frac{2}{3})\tau_2(\frac{1}{3})^{-1}, \tau_0(\frac{2}{3})\tau_0(\frac{1}{3})^{-1}, \tau_0(\frac{1}{3})).
\end{equation*}
It is a direct forward calculation to verify that $ev_1, ev_2$ and $ev_3$ commute with the face and degeneracy maps, for example $$ev_1\circ d_2(\tau)=ev_1(\tau(\frac{t}{3}+\frac{1}{3})\tau(\frac{1}{3})^{-1})= \tau(\frac{2}{3})\tau(\frac{1}{3})^{-1}=d_2\circ ev_2(\tau).$$ 
Since both $NG_\bullet$ and $\GG_\bullet$ are 3-coskeletal\footnote{$NG_\bullet$ is actually 2-coskeletal.} (see e.g. \cite[Sect.2.3]{z:tgpd-2})
, the levels higher than 3 are determined by the lower levels. Hence $ev_1$, $ev_2$ and $ev_3$ naturally extend to a simplicial map. 
\end{proof}

\begin{proposition}\label{Thm-ev-hyper}
Let $G$ be a connected and symply connected Lie group. Then the evaluation map $ev_\bullet: \GG_\bullet\to NG_\bullet$ is an hypercover.
\end{proposition}

\begin{proof}
By definition of hypercover (see \eqref{eq:def-hypercover}), we need to show that  $ev_1:P_e G\to G$ is a surjective submersion of Banach manifolds, and that 
\begin{equation} \label{eq:d-ev-2}
    ((d_0,d_1,d_2), ev_2):\Omega G\to (P_eG)^3\times_{G^{\times 3}}G^{\times 2} 
\end{equation}
is an isomorphism.

For the frist statement, it suffices to solve the problem locally and component-wise, that is, the same problem for $pr_\R \circ ev_1: H_r(\R^n) \to \R$.  Then for a small perturbation $(\epsilon^0, \epsilon^1)$, we may always solve differential equations $\epsilon(x)=\epsilon^0, \epsilon'(x)=\epsilon^1$ at least locally and add this perturbation to any function $f$ to obtain a section\footnote{In the case of Banach manifolds, it is not enough to show that $Tev_1$ is surjective locally, but one needs to show the existence of a splitting section.} of $Tev_1$ near $f$. 

For the second statement, we observe that
\begin{equation*}
\begin{split}
    (P_eG)^3\times_{G^{\times 3}}G^{\times 2}=&\{(\gamma_0,\gamma_1,\gamma_2, g_1, g_2)\in  (P_eG)^3\times G^{\times 2}\ | \ \gamma_0(1)=g_2,\ \gamma_2(1)=g_1,\\
    &\gamma_1(1)=g_1g_2\}
    =\{(\gamma_0,\gamma_1,\gamma_2)\in (P_eG)^3 |  \gamma_2(1)\gamma_0(1)=\gamma_1(1)\}=X.
\end{split}
\end{equation*}
Hence $\Phi:X\to \Omega G$ defined by
\begin{equation*}
     \Phi(\gamma_0,\gamma_1,\gamma_2)(t)=\left\{\begin{array}{ll}
         \gamma_0(3t) & t\in[0,\frac{1}{3}], \\
          \gamma_2(3t-1)\cdot \gamma_0(1)& t\in[\frac{1}{3},\frac{2}{3}],\\
          \gamma_1(3-3t)& t\in[\frac{2}{3}, 1],
     \end{array}\right.
\end{equation*}
provides the inverse of \eqref{eq:d-ev-2}. Therefore \eqref{eq:d-ev-2} is an isomorphism.
\end{proof}

\begin{remark}\label{inf}
For completeness, we notice here that the evaluation map $ev_\bullet:\GG_\bullet\to NG_\bullet$ induces a map between the Lie $2$-algebras $\mathbb{g}=(\Omega\g\xrightarrow{\partial}P_0\g)$ and  $\g=(0\to \g)$ given by $$Lie(ev)_\bullet:\mathbb{g}\to \g,  \quad Lie(ev)_1(u)=u(1)\quad \text{and}\quad Lie(ev)_2(a)=a(1)=0,$$ for $u\in \Omega\g$ and $a\in P_0\g$.   Moreover, if we define $F_\bullet:\g\to \mathbb{g}$ by $F(v)(t)=tv\in P_0\g$ then we can easily check that $$Lie(ev)_\bullet\circ F_\bullet=\id_\g,\quad F_\bullet\circ Lie(ev)_\bullet\sim_{H}\id_\mathbb{g}\quad \text{with}\quad H(u)=\left\{\begin{array}{ll}
    u(2t) & t\in[0,\frac{1}{2}], \\
    (2-2t)u(1) & t\in[\frac{1}{2},1].
\end{array}\right.$$ Therefore $F_\bullet$ is a quasi-inverse for $Lie(ev)_\bullet$ with $2$-morphism given by $H$.
\end{remark}

Before proving that our two models are Morita equivalent we include a computation that was stated in \cite[5.4.2.Prop., 5.4.8.Prop.]{brylinski}  without proofs. 

\begin{lemma}[See \cite{brylinski}]\label{lemma-brylinski}
The transgression\footnote{See Appendix \ref{Ap-trans}  for definition of the transgression map $\tr$.} $\tr:\Omega^\bullet(G)\to \Omega^\bullet(\GG_1)$ of the Cartan $3$-form $\Theta$ introduced in \eqref{Omegadot} satisfies
$$d\tr(\Theta)=ev^*_1\Theta\quad\text{and}\quad \tr(\Theta)=d\alpha^P-2\omega^P\quad\text{with}\quad $$
$$\alpha^P_\gamma(u)=\int_0^1\langle   \widehat{\gamma}(t) \widehat{u}(t)\rangle dt\in\Omega^1(\GG_1),$$ $$\omega^{P}_\gamma(u, v) = \frac{1}{2} \int_{0}^1 \langle \widehat{u}'(t) , \widehat{v}(t) \rangle -  \langle \widehat{u}(t) , \widehat{v}'(t) \rangle dt\in\Omega^2(\GG_1),$$
for $\gamma\in P_eG,\ u,v\in T_\gamma P_eG$ and $\widehat{\gamma}(t)=L_{\gamma^{-1}(t)}\gamma'(t),\ \widehat{u}(t)=L_{\gamma(t)^{-1}}u(t), \ \widehat{v}(t)=L_{\gamma(t)^{-1}}v(t).$
\end{lemma}
\begin{proof}
The first equality follows from a general result on the transgression map that we include in Proposition \ref{Trans-DR}. For the second equality we embed $G$ into a matrix Lie group in order to simplify the calculation. Then for the integral curve $\gamma_{\epsilon}(t):= \gamma(t) \cdot \exp (\epsilon \widehat{u}(t))$ we have 
\begin{equation*}
    \begin{split}
       u(\alpha^P(v))=&\frac{d}{d\epsilon}\Big\rvert_{\epsilon=0}\Big(\int_0^1 \langle \gamma_\epsilon^{-1}(t)\gamma'_\epsilon(t),\widehat{v}(t)\rangle dt\Big)\\
        =&\int_0^1\langle-\widehat{u}(t)\gamma^{-1}(t)\gamma'(t)+\gamma^{-1}(t)\big(\gamma'(t)\widehat{u}(t)+\gamma(t)\widehat{u}'(t)\big),\widehat{v}(t)\rangle dt\\
        =&\int_0^1\langle[\widehat{\gamma}(t),\widehat{u}(t)],\widehat{v}(t)\rangle+ \langle \widehat{u}'(t),\widehat{v}(t)\rangle dt.
    \end{split}
\end{equation*}
Therefore we obtain
\begin{equation*}
\begin{split}
     d\alpha^P(u,v)=&u(\alpha^P(v))-v(\alpha^P(u))-\alpha^P([u,v])
     =\int_0^1\Big(\langle[\widehat{\gamma}(t),\widehat{u}(t)],\widehat{v}(t)\rangle+ \langle \widehat{u}'(t),\widehat{v}(t)\rangle\\ -&\langle[\widehat{\gamma}(t),\widehat{v}(t)],\widehat{u}(t)\rangle- \langle \widehat{v}'(t),\widehat{u}(t)\rangle -\langle \widehat{\gamma}(t),[\widehat{u}(t),\widehat{v}(t)]\rangle\Big) dt\\
     =&\int_0^1\Big( \langle \widehat{u}'(t),\widehat{v}(t)\rangle - \langle \widehat{v}'(t),\widehat{u}(t)\rangle \Big) dt+\int_0^1\langle \widehat{\gamma}(t),[\widehat{u}(t),\widehat{v}(t)]\rangle dt\\
     =&2\omega^P(u,v)+\tr(\Theta)(u,v),
\end{split}
\end{equation*}
where in the middle step we have used the ad-invariance of the pairing and in the last step the explicit formula for the transgression given in Proposition \ref{trans-expli}.
\end{proof}

\begin{theorem}\label{thm-morita}
Let $G$ be a connected and simply connected Lie group with quadratic Lie algebra $(\g,\langle\cdot,\cdot\rangle)$. The $2$-shifted symplectic Lie $2$-groups $(\GG_\bullet,\omega_\bullet)$ and $(NG_\bullet,\frac{1}{2}\Omega_\bullet)$ are symplectic Morita equivalent via
$$(\GG_\bullet,\omega_\bullet)\xleftarrow{\Id_\bullet}(\GG_\bullet,\omega^P_\bullet)\xrightarrow{ev_\bullet}(NG_\bullet, \frac{1}{2}\Omega_\bullet), $$
where $\omega_\bullet^P=\omega^P+0$ is the $1$-shifted $2$-form on $\GG_\bullet$ defined in Lemma \ref{lemma-brylinski}.
\end{theorem}

\begin{proof}
To prove this, we need to show 
\begin{equation*}\label{eq:Omega-omega}
    \frac{1}{2}ev^*_1 \Theta =- d\omega^{P}\qquad\text{ and}\qquad  \omega-\frac{1}{2}ev^*_2 \Omega= \delta \omega^{P}.
\end{equation*}
The first equality follows directly from the previous Lemma \ref{lemma-brylinski},  since we have
$$ \frac{1}{2}ev^*_1\Theta=\frac{1}{2}d\tr(\Theta)=-d\omega^P , $$
as we want. For the second equality,  we need to compute 
\[
\delta \omega^{P}  = d^*_0 \omega^{P}  - d^*_1 \omega^{P}  + d^*_2 \omega^{P},\quad\text{ and }\quad ev_2^*\Omega . 
\]
Take  $\tau\in\Omega G,\ a, b\in T_\tau\Omega G$, and $\widehat{a}(t)=L_{\tau(t)^{-1}}a(t), \widehat{b}(t)=L_{\tau(t)^{-1}}b(t)$ and recall from Lemma \ref{evaluation} that the evaluation map is
$ev_2(\tau)=(\tau(\frac{2}{3})\tau(\frac{1}{3})^{-1},\tau(\frac{1}{3}))$,  hence its tangent reads    $$Tev_2(a)=(R_{\tau(\frac{1}{3})^{-1}}a(\frac{2}{3}) -L_{\tau(\frac{2}{3} )\tau(\frac{1}{3})^{-1}} R_{\tau(\frac{1}{3})^{-1}}  a(\frac{1}{3}) , a(\frac{1}{3}) ).$$  With this and equation \eqref{exOme} we compute the pullback of $\Omega$ and obtain
\begin{equation}\label{eq:ev2-Omega}
    \begin{split}
        (ev^*_2\Omega)_{\tau}(a,b)= &\Omega_{ev_2(\tau)}(Tev_2(a),Tev_2(b))  \\
        = &\langle L_{\tau(\frac{1}{3})\tau(\frac{2}{3} )^{-1}}R_{\tau(\frac{1}{3})^{-1}} a(\frac{2}{3}) - R_{\tau(\frac{1}{3})^{-1}}  a(\frac{1}{3}) , R_{\tau(\frac{1}{3})^{-1}}b(\frac{1}{3})\rangle\\
        -&\langle L_{\tau(\frac{1}{3})\tau(\frac{2}{3} )^{-1}}R_{\tau(\frac{1}{3})^{-1}} b(\frac{2}{3}) - R_{\tau(\frac{1}{3})^{-1}}  b(\frac{1}{3}) , R_{\tau(\frac{1}{3})^{-1}}a(\frac{1}{3})\rangle\\
        =&\langle \widehat{a}(\frac{2}{3}), \widehat{b}(\frac{1}{3})\rangle-\langle \widehat{b}(\frac{2}{3}), \widehat{a}(\frac{1}{3})\rangle. 
    \end{split}
\end{equation}

For $\delta\omega^P$ we use the face maps and their tangents as given in equations \eqref{2faces} and \eqref{T2faces}. We have
\begin{equation}\label{eq:d2-omegaP}
    \begin{split}
        &(d^*_2\omega^P)_{\tau}(a,b)=\omega^P_{d_2(\tau)}(Td_2(a),Td_2(b))\\
        =&\frac{1}{2}\int_0^1\langle\big(L_{\tau(\frac{1}{3})\tau(\frac{1+t}{3} )^{-1}}R_{\tau(\frac{1}{3})^{-1}} a(\frac{1+t}{3})\big)',L_{\tau(\frac{1}{3})\tau(\frac{1+t}{3} )^{-1}}R_{\tau(\frac{1}{3})^{-1}} b(\frac{1+t}{3})-R_{\tau(\frac{1}{3})^{-1}}b(\frac{1}{3})\rangle\\
        -&\frac{1}{2}\int_0^1\langle L_{\tau(\frac{1}{3})\tau(\frac{1+t}{3} )^{-1}}R_{\tau(\frac{1}{3})^{-1}} a(\frac{1+t}{3})-R_{\tau(\frac{1}{3})^{-1}}a(\frac{1}{3}), \big(L_{\tau(\frac{1}{3})\tau(\frac{1+t}{3} )^{-1}}R_{\tau(\frac{1}{3})^{-1}} b(\frac{1+t}{3})\big)'\rangle\\
        =&\frac{1}{2}\int_{\frac{1}{3}}^\frac{2}{3}\langle \widehat{a}'(s), \widehat{b}(s)\rangle-\frac{1}{2}\int_{\frac{1}{3}}^\frac{2}{3}\langle \widehat{a}(s), \widehat{b}'(s)\rangle
        -\frac{1}{2}\int_{\frac{1}{3}}^\frac{2}{3}\langle \widehat{a}'(s), \widehat{b}(\frac{1}{3})\rangle+\frac{1}{2}\int_{\frac{1}{3}}^\frac{2}{3}\langle \widehat{a}(\frac{1}{3}), \widehat{b}'(s)\rangle\\
        =&\frac{1}{2}\int_{\frac{1}{3}}^\frac{2}{3}\langle \widehat{a}'(s), \widehat{b}(s)\rangle-\frac{1}{2}\int_{\frac{1}{3}}^\frac{2}{3}\langle \widehat{a}(s), \widehat{b}'(s)\rangle
        -\frac{1}{2}\langle \widehat{a}(\frac{2}{3}), \widehat{b}(\frac{1}{3})\rangle+\frac{1}{2}\langle \widehat{b}(\frac{2}{3}), \widehat{a}(\frac{1}{3})\rangle, 
    \end{split}
\end{equation}
where in the last step we are using Stokes' Theorem. Similarly, through a much easier calculation, we have 
$\tau_0(t)^{-1} Td_0 (a)(t)=\tau(\frac{t}{3})^{-1} a(\frac{t}{3})=\widehat{a}(\frac{t}{3})$, 
$\tau_1(t)^{-1} Td_1 (a)(t) =\widehat{a}(1-\frac{t}{3})$, and similar for $b(t)$. Thus $d^*_0 \omega^{P} = \omega^P|_{[0, \frac{1}{3}]}$ and $d^*_1 \omega^{P} = - \omega^P|_{[\frac{2}{3}, 1]}$. Combining these and \eqref{eq:ev2-Omega}, \eqref{eq:d2-omegaP}, we see that
$$(\delta\omega^P)_\tau(a,b)=\omega_\tau^P(a,b)-\frac{1}{2}(ev^*_2\Omega)_\tau(a,b)=\omega_\tau(a,b)-\frac{1}{2}(ev^*_2\Omega)_\tau(a,b)$$
where  $\omega^P|_{\Omega G} = \omega$ by  Stokes' theorem.
\end{proof}

\begin{remark}
A relation between the $2$-forms $\omega$ and $\Omega$ was also noticed in \cite[Section 5.5]{Ale-dirac}  using Courant algebroids and forwards maps of Dirac structures. They also use some of the multiplicative properties of this forms. 
\end{remark}

\begin{remark}\label{VE2}
In the work in progress \cite{cam-jo} we introduce a Van Est map for Lie $2$-groups as a map of double complexes analogous to \eqref{VE-map}. In special case of $\GG_\bullet$,  we obtain a map $VE:\Omega^q(\GG_p)\to \Omega^q(\mathbb{g}[1])_p$. Using the map $Lie(ev)_\bullet$ introduced in Remark \ref{inf} it is not difficult to show that
$$VE(\omega_\bullet)-\frac{1}{2}Lie(ev)^*_\bullet\big( VE(\Omega_\bullet)\big)=D \big(VE(\omega^P_\bullet)\big).$$
This may be viewed as an infinitesimal version of Theorem \ref{thm-morita}. 
\end{remark}

\section{Manin triples and $\huaB G$}\label{Sec:4}

Recall that a {\bf Manin triple } $(\g, \h_+,\h_- )$ consists of a quadratic Lie algebra $(\g,\langle\cdot,\cdot\rangle)$  together with two Lie subalgebras $\h_+, \h_-\subset \g$ satisfying
$$\langle \h_{+},\h_{+}\rangle=\langle \h_{-},\h_{-}\rangle=0\quad\text{and}\quad \g=\h_+\oplus \h_-.$$
The decomposition $\g=\h_+\oplus\h_-$ makes $\h_+$, $\h_-$ and $\g$ into {\bf Lie bialgebras} and it is well known \cite{kos:bia} that there is a correspondence between Lie bialgebras and Manin triples. Lie bialgebras also plays an important role here because they are the infinitesimal counterpart of Poisson-Lie groups, for more details see e.g. \cite{EtSc02, kos:bia, lu-weinstein}.

\begin{example}
Some examples of Manin triples are the following. Given a Lie algebra $\h$,  $(\g=\h\ltimes\h^*, \h_+=\h, \h_-=\h^*)$ is a Manin triple,  where the semidirect product is with respect to the coadjoint representation. If we start with a Lie bialgebra $(\h, [\cdot,\cdot], d^*)$,  there is also a Manin triple, $(\g=\h\bowtie\h^*, \h_+=\h, \h_-=\h^*)$,  given by the Drinfeld's double.   A particular interesting case is when the coalgebra is given by an $r$-matrix $r\in \h\otimes \h$. For complex reductive Lie algebras, a classification of Manin triples is given in \cite{del:cla}.
\end{example}

Given a Manin triple $(\g, \h_+,\h_- )$, a symplectic double Lie group $(\Gamma^\h_{\bullet,\bullet}, \omega^\h)$ is constructed in \cite{Lu-we1}. 
Using $(\Gamma^\h_{\bullet,\bullet}, \omega^\h)$ and the Artin-Mazur codiagonal procedure \cite{artin-mazur}, we obtain a (local) Lie $2$-group $\bar{\Gamma}^\h_\bullet$ and a $2$-form $\bar{\omega}^\h_{\bullet}\in\Omega^2(\bar{\Gamma}^\h_2)$ as in \cite{MeTa11}.  In this section, we further show that  $(\bar{\Gamma}^\h_\bullet, \bar{\omega}^\h)$ is a $2$-shifted symplectic (local) Lie $2$-group  and more importantly, we give in Theorem \ref{Thm-manin} an explicit symplectic Morita equivalence between $(\bar{\Gamma}^\h_\bullet, \bar{\omega}^\h_{\bullet})$ and $(NG_\bullet, \Omega_\bullet).$

\subsection{Double Lie groups and Artin-Mazur codiagnoal construction}\label{se:dlg}
Here we recall some preliminaries on double Lie groups and the Artin-Mazur codiagonal construction. For more details see \cite{artin-mazur, MK2, MeTa11}.
\begin{equation}\label{dlg}
   \xymatrix{\huaG\ar@<-0.5 ex>[r]_{s^h}\ar@<0.5 ex>[r]^{t^h}\ar@<-0.5 ex>[d]_{t^v}\ar@<0.5 ex>[d]^{s^v}& B\ar@<-0.5 ex>[d]\ar@<0.5 ex>[d]\\
     A\ar@<-0.5 ex>[r]\ar@<0.5 ex>[r]&pt}
\end{equation}
The diagram \eqref{dlg} is a {\bf double Lie group} if the two Lie groupoid structures of $\huaG$ satisfy that $s^h,t^h, s^v,t^v$ are groupoid morphisms and the multiplications commute, i.e.
$$m^v\big(m^h(g_{11}, g_{12}), m^h(g_{21}, g_{22})\big)=m^h\big(m^v(g_{11}, g_{21}), m^v(g_{12}, g_{22})\big)$$
for all $g_{ij}\in \huaG$ such that $s^h(g_{i1})=t^h(g_{i2})$ and $s^v(g_{1j})=t^h(g_{2j}).$

In order to avoid confusion, we denote a double Lie group by $\huaG_{\bullet,\bullet}$. This notation is motivated by the fact that for any double Lie group \eqref{dlg}, there is an associated {\bf bisimplicial manifold $\huaG_{\bullet,\bullet}$} (see \cite[Prop 3.10]{MeTa11}) given by 
$$      \huaG_{j,i}=\{ (\dots, g_{qp}, \dots)_{ 1\leq q\leq j, 1\leq p\leq i}| g_{qp} \in \huaG, \ s^h(g_{qp})=t^h(g_{q(p+1)}),\ s^v(g_{qp})=t^v(g_{(q+1)p})\},  $$
for $i, j \ge 1$, and $\huaG_{j,0}=B^{\times j}$, $\huaG_{0,i}=A^{\times i}$.  Faces and degeneracies are induced by the multiplications and the units,  see \cite{MeTa11} for explicit formulas. 


In the 60s it was shown in \cite{artin-mazur} how to take a codiagonal construction on a bisimplicial manifold $X_{\bullet,\bullet}$ to obtain a simplicial manifold $\bar{X}_\bullet$. This method was used in \cite{MeTa11} to prove that if the bisimplicial manifold is given by a double Lie group\footnote{The theorem holds in general for double Lie groupoids.} $\huaG_{\bullet,\bullet}$,  then $\bar{\huaG}_\bullet$ is a local Lie $2$-group (see \cite[Theorem 4.5]{MeTa11}). Moreover, \cite{MeTa11} shows that if the double source map $(s^v,s^h):\huaG\to A\times B$ is surjective then $\bar{\huaG}_\bullet$ is a Lie $2$-group.

\subsection{The Manin triple model}\label{sec:manin-model}

Given a Manin triple $(\g, \h_+, \h_-)$,  denote by $H_\pm$ and $G$ the connected and simply connected Lie groups integrating $\h_\pm$ and $\g$ respectively. By the Lie II theorem, we can also integrate the inclusions and obtain two group morphisms $\varphi_{\pm}:H_\pm\to G$. Denote by $\Gamma^\h$ the following set
$$\Gamma^\h=\{(h_2,a_2,a_1,h_1)\in H_+\times H_-\times H_-\times H_+\ |\ \varphi_+(h_2)\varphi_-(a_1)=\varphi_-(a_2)\varphi_+(h_1)\}.$$
The main result of \cite[Theorem 3]{Lu-we1} is the following:
\begin{proposition}[See \cite{Lu-we1}]\label{propwe}
Given a Manin triple $(\h_+, \h_-,\g)$, there is a symplectic double Lie group $(\Gamma^\h_{\bullet,\bullet}, \omega^\h)$ given by the square
\begin{equation*}
     \xymatrix{\Gamma^\h\ar@<-0.5 ex>[r]_{s^h}\ar@<0.5 ex>[r]^{t^h}\ar@<-0.5 ex>[d]_{t^v}\ar@<0.5 ex>[d]^{s^v}& H_+\ar@<-0.5 ex>[d]\ar@<0.5 ex>[d]\\
     H_-\ar@<-0.5 ex>[r]\ar@<0.5 ex>[r]&pt}
 \end{equation*}
 with horizontal source, target, and multiplication given by
  \begin{equation*}
     s^h(\xi)=h_1, \ t^h(\xi)=h_2, \ m^h(\xi, \xi')=(h_2, a_2a_2', a_1a_1', h_1'), 
 \end{equation*}
 and vertical source, target, and multiplication given by
  \begin{equation*}
     s^v(\xi)=a_1, \ t^v(\xi)=a_2, \ m^v(\xi, \xi')=(h_2h_2', a_2, a_1', h_1h_1'),  
 \end{equation*}
 for $\xi=(h_2,a_2,a_1,h_1)$ and  $\xi'=(h_2',a_2',a_1',h_1')\in\Gamma^\h$. The symplectic form $\omega^\h$ is given by
 \begin{equation*}
     \omega^\h:=\langle (t^h)^*\theta^l_{H_+},(s^v)^*\theta_{H_-}^r\rangle-\langle(t^v)^*\theta^l_{H_-},(s^h)^*\theta^r_{H_+}\rangle\in\Omega^2(\Gamma^\h).
 \end{equation*}
\end{proposition}
\begin{remark}
The explicit form of the symplectic structure $\omega^\h$ is taken  from \cite[Theorem 3.12]{Henrique-lu} and  also appears in \cite{libland-severa1}.
\end{remark}

With all these preliminaries,  we are now ready to equip the (local) Lie $2$-group $\bar{\Gamma}^\h_\bullet$ and the $2$-form $\bar{\omega}^\h$ constructed in \cite{MeTa11} with a 2-shifted symplectic local Lie 2-group structure.  We notice that the non-degeneracy condition \eqref{eq:homo-pai} for the $2$-form $\bar{\omega}^\h$ is therefore different from the one demanded in \cite{MeTa11}. A modification of their result leads to the following.

\begin{proposition}\label{prop-manin-sym}
Let $(\g,\h_+,\h_-)$ be a Manin triple and $(\Gamma^\h_{\bullet,\bullet}, \omega^\h)$ the symplectic double group of Proposition \ref{propwe}. Then $(\bar{\Gamma}^\h_\bullet, \bar{\omega}^\h)$  is a  $2$-shifted symplectic local Lie $2$-group given by 
$$\bar{\Gamma}^\h_\bullet=\cdots H_-\times\Gamma^\h\times H_+\aaar H_-\times H_+\rightrightarrows pt, \quad\text{and}\quad \bar{\omega}^\h=\pr_{\Gamma^\h}^*\omega^\h\in\Omega^2(\bar{\Gamma}^\h_2), $$
with faces$$ d_0(\lambda)=(a_1, h_1), \quad d_1(\lambda)=(a_3a_2, h_2h_1),\quad \text{and}\quad d_2(\lambda)=(a_3, h_3),$$ 
for $\lambda=(a_3,(h_3, a_2, a_1, h_2),h_1)\in\bar{\Gamma}^\h_2=H_-\times\Gamma^\h\times H_+$, and degeneracies, $$s_0(a,h)=(e,(e,a,a,e),h) \quad\text{and}\quad s_1(a,h)=(a,(h,e,e,h),e), \quad \text{for}\; (a,h)\in H_-\times H_+.$$
\end{proposition}
\begin{proof}
The fact that $\bar{\Gamma}^\h$ is a local Lie $2$-group follows from \cite[Theorem 4.5]{MeTa11}. The $2$-form being closed, that is $D(\bar{\omega}^\h)=0$, is equivalent to $$d\bar{\omega}^\h=0 \qquad \text{and}\qquad\delta\bar{\omega}^\h=0.$$
The first equality follows directly from $d\omega^\h=0$, while the second was shown in \cite[Proposition 6.2]{MeTa11}. Therefore it remains to see that $\bar{\omega}^\h$ is normalized and non-degenerate. That $\omega$ is normalized is trivial since
$$s_0^*\bar{\omega}^\h_{(a,h)}((\widetilde{v},v),(\widetilde{w},w))=\omega^\h_{(e,a,a,e)}((0,\widetilde{v},\widetilde{v},0),(0,\widetilde{w},\widetilde{w},0))=0, $$
and the other degeneracy is similar. 

For the non-degeneracy condition,  we need the tangent complex of $\bar{\Gamma}^\h_\bullet$. By Section \ref{sec:tan} this is given by $$0\to\huaT_2(\bar{\Gamma}^\h)\xrightarrow{\partial} \huaT_1(\bar{\Gamma}^\h)\to 0\quad \text{with}\quad \huaT_i(\bar{\Gamma}^\h)=\cap_{j=0}^{i-1} \ker T_ed^i_j\quad\text{and}\quad \partial=T_ed^2_2.$$
Clearly $\huaT_1(\bar{\Gamma}^\h)=\h_-\oplus\h_+=\g$ and we claim that $\huaT_2(\bar{\Gamma}^\h)=0$. This follows from the fact that if $V=(\widetilde{v}_3,(v_3,\widetilde{v}_2,\widetilde{v}_1,v_2),v_1)\in T_e\bar{\Gamma}^\h_2$ then
$$T_ed_0(V)=(\widetilde{v}_1, v_1),\quad T_ed_1(V)=(\widetilde{v}_3+\widetilde{v}_2, v_1+v_2),  \quad\text{and}\quad v_3+\widetilde{v}_1=\widetilde{v}_2+v_2.$$
Therefore for $(\widetilde{v}, v),(\widetilde{w}, w)\in\h_-\oplus\h_+=T_e\bar{\Gamma}^\h_1$,  we have,
\begin{equation*}
    \begin{split}
        \lambda^{\bar{\omega}^\h} \big((\widetilde{v}, v),(\widetilde{w}, w)\big)=&\bar{\omega}^\h_e(Ts_0(\widetilde{v}, v), Ts_1(\widetilde{w}, w))-\bar{\omega}^\h_e(Ts_1(\widetilde{v}, v), Ts_0(\widetilde{w}, w))\\
        =&\bar{\omega}^\h_e((0,\widetilde{v},\widetilde{v},0),(w,0,0, w))+\bar{\omega}^\h_e((0,\widetilde{w},\widetilde{w},0),(v,0,0, v))\\
        =&-2\langle\widetilde{v}, w\rangle-2\langle\widetilde{w}, v\rangle, 
    \end{split}
\end{equation*}
which is non-degenerate since $(\g, \h_+,\h_-)$ is a Manin triple.
\end{proof}

\subsection{The equivalence}
We end by comparing the $2$-shifted symplectic local Lie $2$-group $(\bar{\Gamma}^\h_\bullet, \bar{\omega}^\h_\bullet)$ with the $2$-shifted symplectic Lie $1$-group\footnote{Notice that a $2$-shifted symplectic Lie $1$-group is in particular a $2$-shifted symplectic local Lie $2$-group.} $(NG_\bullet, \Omega_\bullet)$ constructed in Section \ref{sec-findim}.

We first observe that the concept of hypercover may easily adapt to the more general case of local Lie $n$-groupoids. That is, a simplicial morphism $f_\bullet: K_\bullet \to J_\bullet$ between two local Lie $n$-groupoids is called a {\bf hypercover} if the map $q_i$ in \eqref{eq:def-hypercover} is a submersion when $0\le i \le n-1$ and injective \'etale when $i=n$. Then similar to Remark \ref{remark:local-sslg}, we may easily adapt what we have developed in Section \ref{sec:me} on symplectic Morita equivalence to the contex of $m$-shifted local Lie $n$-groupoids, by simply adding the adjective ``local'' in the appropriate places. For example, similar to Definition \ref{def:sympmorita}, two local $m$-shifted Lie $n$-groupoid $(K_\bullet, \alpha_\bullet)$ and $(J_\bullet, \beta)$ are {\bf symplectic Morita equivalent}, if there is another local Lie $n$-groupoid $Z_\bullet$ together with an $(m-1)$-shifted 2-form $\phi_\bullet$ and hypercovers $f_\bullet$, $g_\bullet$ satisfying \eqref{eq:sympme}.


\begin{proposition}\label{prop:phi}
Let $(\g,\h_+,\h_-)$ be a Manin triple. There is a hypercover of local Lie 2-groups  $\Phi_\bullet:\bar{\Gamma}^\h_\bullet\to NG_\bullet$ given by 
\begin{equation*}
     \Phi_1: \bar{\Gamma}_1^\h\to NG_1, \quad  \Phi_1(a,h)=\varphi_-(a)\varphi_+(h), \quad \text{ and }
\end{equation*}
\begin{equation*}
   \Phi_2: \bar{\Gamma}_2^\h\to NG_2, \quad   \Phi_2(\lambda)=\big(\varphi_-(a_3)\varphi_+(h_3), \ \varphi_-(a_1)\varphi_+(h_1)\big),
\end{equation*}
with $\lambda=(a_3,(h_3,a_2,a_1,h_2),h_1)\in\bar{\Gamma}_2^\h=H_-\times\Gamma^\h\times H_+$.
\end{proposition}
\begin{proof}
By definition $\Phi_\bullet$ is an hypercover of local Lie $2$-groups if $q_1$ is a submersion and $q_2$ is an injective \'{e}tale map. In our case, $q_1=\Phi_1:\bar{\Gamma}_1^\h\to NG_1$  is a local diffeomorphism as proved in \cite{Lu-we1}, and therefore a submersion.  The map $q_2:\bar{\Gamma}_1^\h\to  \bar{\Gamma}_1^\h\times \bar{\Gamma}_1^\h \times \bar{\Gamma}_2^\h \times_{G\times G\times G} NG_2$ is given by
$$q_2(\lambda)=((a_1,h_1),(a_3a_2,h_2h_1),(a_3,h_3), \varphi_-(a_3)\varphi_+(h_3), \varphi_-(a_1)\varphi_+(h_1)).$$
A short computation shows that $S: \bar{\Gamma}_1^\h\times \bar{\Gamma}_1^\h \times \bar{\Gamma}_2^\h \times_{G\times G\times G} NG_2\to \bar{\Gamma}_1^\h$ defined as
$$S((a_1, h_1),(a_2,h_2),(a_3,h_3), g_1, g_2)=(a_3,(h_3, a_3^{-1}a_2, a_1, h_2h_1^{-1}), h_1)$$
is well defined and gives an inverse for $q_2$. Therefore $q_2$ is an diffeomorphism and in particular it is injective \'{e}tale.
\end{proof}

\begin{remark}\label{rmk:obstruction}
Recall that since $(\g,\h_+,\h_-)$ is a Manin triple, then $\h_-$ acts on $\h_+$ and therefore also acts on $H_+$. This action  is known as the infinitesimal dressing action \cite{lu-weinstein}. We say that $H_-$  is {\bf complete} if the vector fields that generate the action $\h_-\curvearrowright H_+$ are complete. In \cite{lu-weinstein} it is proved that the following are equivalent:
\begin{enumerate}
    \item $H_-$ is complete,
    \item There is an action of $H_-$ on $H_+$,
    \item $H_+$ is complete,
    \item $\Phi_1:H_-\times H_+\to G$ is a diffeomorphism.
\end{enumerate}
\end{remark}
An immediate consequence of the preceding remark is the following.
\begin{corollary}
Let $(\g,\h_+,\h_-)$ be a Manin triple and assume that $H_-$ is complete. Then $\Phi_\bullet:\bar{\Gamma}^\h_\bullet\to NG_\bullet$ is a diffeomorphism and therefore $\bar{\Gamma}^\h_\bullet$ is a Lie $1$-group.
\end{corollary}

\begin{theorem}\label{Thm-manin}
Let $(\g,\h_+,\h_-)$ be a Manin triple. Then the $2$-shifted symplectic local Lie $2$-groups $(\bar{\Gamma}^\h_\bullet, \bar{\omega}^\h_\bullet)$ and $(NG_\bullet, \Omega_\bullet)$ are symplectic Morita equivalent via
$$(\bar{\Gamma}^\h_\bullet, \bar{\omega}_\bullet^\h)\xleftarrow{\id_\bullet}(\bar{\Gamma}^\h_\bullet, \beta_\bullet)\xrightarrow{\Phi_\bullet}(NG_\bullet, \Omega_\bullet),$$
where $\beta_\bullet=\beta+0$ is the $1$-shifted $2$-form on  $\bar{\Gamma}^\h_\bullet$ defined by
$$\beta=\langle \theta^l_{H_-},\theta^r_{H_+}\rangle\in\Omega^2(\bar{\Gamma}^\h_1=H_-\times H_+).$$
\end{theorem}
\begin{proof}
The result follows by showing 
$$\bar{\omega}^\h-\Phi_2^*\Omega=\delta\beta  \quad \text{and}\quad \Phi_1^*\Theta=-d\beta.$$
For the second equation,  observe that $(\Phi_1^*\theta^l)_{(a,h)}=\theta^l_{H_+}+Ad_{h^{-1}}(\theta^l_{H_-})$, for $(a, h) \in H_-\times H_+$. Therefore by using the explicit form of the Cartan 3-form given in \eqref{Omegadot}, we can see that
\begin{equation*}
    \begin{split}
        (\Phi^*_1\Theta)=&
        \frac{1}{6}\langle \theta^l_{H_+}+Ad_{h^{-1}}(\theta^l_{H_-}),[\theta^l_{H_+}+Ad_{h^{-1}}(\theta^l_{H_-}),\theta^l_{H_+}+Ad_{h^{-1}}(\theta^l_{H_-})]
        \rangle\\
        =&\frac{1}{2}\langle \theta^r_{H_+},[\theta^l_{H_-},\theta^l_{H_-}]
        \rangle+\frac{1}{2}\langle \theta^l_{H_-},[\theta^r_{H_+},\theta^r_{H_+}]
        \rangle=-d\langle\theta^l_{H_-},\theta^r_{H_+}\rangle.
    \end{split}
\end{equation*}
Here we use the adjoint invariance of the pairing and the fact that $\h_+$ and $\h_-$ are isotropic Lie algebras. 
The second equation $\bar{\omega}^\h-\Phi_2^*\Omega=\delta\beta$ is more complicated. In order to make the equations shorter, we use the notation $$\bar{h}:=\varphi_+(h), \ \bar{a}:=\varphi_-(a)\quad\text{and}\quad \bar{v}:=T\varphi_+(v),\ \bar{\widetilde{v}}:=T\varphi_-(\widetilde{v})$$
for $h\in H_+, \ a\in H_-, \ v\in T_hH_+, \widetilde{v}\in T_aH_-$. In this notation, we have the identities
$$ L_{\bar{h}^{-1}}\bar{v}=i_+(L_{h^{-1}}v)\quad \text{and}\quad L_{\bar{a}^{-1}}\bar{\widetilde{v}}=i_-(L_{a^{-1}}\widetilde{v}), $$ where $i_\pm:\h_\pm\to \g$ denotes the inclusion.  Pick $\lambda=(a_3,(h_3,a_2,a_1,h_2),h_1)\in\bar{\Gamma}^\h_2$ and $V=(\widetilde{v}_3,(v_3,\widetilde{v}_2,\widetilde{v}_1,v_2),v_1),$ $W=(\widetilde{w}_3,(w_3,\widetilde{w}_2,\widetilde{w}_1,w_2),w_1)\in T_\lambda\Gamma^\h_2$. From the definition of $\Gamma^\h$, we have
\begin{equation*}
    \bar{h}_3\bar{a}_1=\bar{a}_2\bar{h}_2,\quad L_{\bar{h}_3}\bar{\widetilde{v}}_1+R_{\bar{a}_1}\bar{v}_3=L_{\bar{a}_2}\bar{v}_2+R_{\bar{h}_2}\bar{\widetilde{v}}_2,\quad\text{and}\quad (v\leftrightarrow w),
\end{equation*}
where $(v\leftrightarrow w)$ is a shorthand notation which denotes the same equation but with  $v$ and $w$ exchanged. Hence it follows that the two terms of the left hand side are
\begin{equation*}
    \begin{split}
        \bar{\omega}^\h_\lambda(V,W)=&\langle L_{h_3^{-1}}v_3, R_{a_1^{-1}}\widetilde{w}_1\rangle-\langle L_{a_2^{-1}}\widetilde{v}_2, R_{h_2^{-1}}w_2\rangle- (v\leftrightarrow w)\\
        (\Phi_2^*\Omega)_\lambda(V,W)=&\Omega_{\Phi_2(\lambda)}(T\Phi_2(V),T\Phi_2(W)) \\
        =&\langle L_{\bar{h}_3^{-1}\bar{a}_3^{-1}}R_{\bar{h}_3}\bar{\widetilde{v}}_3+L_{\bar{h}_3^{-1}}\bar{v}_3, R_{\bar{a}_1^{-1}}\bar{\widetilde{w}}_1+L_{\bar{a}_1}R_{\bar{h}_1^{-1}\bar{a}_1^{-1}}\bar{w}_1\rangle- (v\leftrightarrow w)\\
        =&\langle L_{a_2^{-1}a_3^{-1}}R_{a_2}\widetilde{v}_3, R_{h_2^{-1}}w_2\rangle-\langle L_{a_3^{-1}}\widetilde{v}_3, R_{h_3^{-1}}w_3\rangle+\langle L_{a_2^{-1}a_3^{-1}}R_{a_2}\widetilde{v}_3,L_{h_2}R_{h_1^{-1}h_2^{-1}}w_1\rangle\\
        &+\langle L_{h_3^{-1}}v_3, R_{a_1^{-1}}\widetilde{w}_1\rangle+\langle L_{a_2^{-1}}\widetilde{v}_2,L_{h_2} R_{h_1^{-1}h_2^{-1}}w_1\rangle -\langle L_{a_1^{-1}}\widetilde{v}_1, R_{h_1^{-1}}w_1\rangle- (v\leftrightarrow w)
    \end{split}
\end{equation*}
while the term on the right hand side is
\begin{equation*}
    \begin{split}
        \delta\beta_\lambda(V,W)=&\beta_{d_0(\lambda)}(Td_0(V), Td_0(W)+\beta_{d_2(\lambda)}(Td_2(V), Td_2(W)-\beta_{d_1(\lambda)}(Td_1(V), Td_1(W)\\
        =&\langle L_{a_1^{-1}}\widetilde{v}_1, R_{h_1^{-1}}w_1\rangle+\langle L_{a_3^{-1}}\widetilde{v}_3, R_{h_3^{-1}}w_3\rangle\\
        &-\langle L_{a_2^{-1}a_3^{-1}}R_{a_2}\widetilde{v}_3+L_{a_2^{-1}}\widetilde{v}_2, R_{h_2^{-1}}w_2+L_{h_2}R_{h_1^{-1}h_2^{-1}}w_1\rangle   - (v\leftrightarrow w).
    \end{split}
\end{equation*}
Therefore, by comparing them we see $\bar{\omega}^\h-\Phi_2^*\Omega=\delta\beta$ as desired.
\end{proof}

\begin{remark}
It was proved in \cite{saf:poi}  that for any Manin triple $(\g,\h_+. \h_-)$ there is a $2$-shifted Lagrangian correspondence $$\xymatrix{& pt\ar[rd]\ar[ld]\\ \huaB H_+\ar[rd]&&\huaB H_-\ar[ld]\\& \huaB G}$$ This is closely related to our result in this section. 
\end{remark}

\section{Double Lie group models of $\huaB G$}\label{Sec:double}
The simplicial picture introduced in the previous sections is not the only available approach to describe $\huaB G$. In this section we will use the language of double Lie groups to define other models for $\huaB G$ and its symplectic structure.  

\subsection{Strict Lie 2-groups} \label{sec:str-lie-2}
A {\bf strict Lie $2$-group} \cite{baez:2gp} is a group object in the category of Lie groupoids and (strict) Lie groupoid morphisms, that is, it is a Lie groupoid $G_1\Rightarrow G_0$ which equipped with a multiplication functor, an inverse functor, and an identity functor, satisfying (strictly) associativity, and other expected axioms for groups.  It is well known (and not hard to see) that a strict Lie $2$-group $G_1\Rightarrow G_0$ gives rise to a double Lie group (defined in \eqref{dlg})
\[
\xymatrix{G_1\ar@<-0.5 ex>[r]\ar@<0.5 ex>[r]\ar@<-0.5 ex>[d]\ar@<0.5 ex>[d]& pt\ar@<-0.5 ex>[d]\ar@<0.5 ex>[d]\\
     G_0\ar@<-0.5 ex>[r]\ar@<0.5 ex>[r]&pt.}
\]  Such a strict Lie 2-group is also a special case of  a Lie 2-group which is defined previously in Definition \ref{def:lie-n-gpd} using simplicial manifolds. A strict Lie 2-group $G_1\Rightarrow G_0$ gives rise to a Lie 2-group 
\[
 \cdots G_0 \times G_0 \times_{m_0, G_0, \target} G_1 \aaar G_0\rightrightarrows pt,
\] where $m_0: G_0\times G_0\to G_0$ is the 0-th level of the multiplication functor. We refer to \cite{z:tgpd-2} for more details. 

\subsection{The de Rham triple complex of a double Lie group}
Recall from Section \ref{se:dlg} that a double Lie group has an associated bisimplicial manifold $\huaG_{\bullet,\bullet}$. Therefore {\bf differential forms} on $\huaG_{\bullet,\bullet}$ live in the {\bf de Rham triple complex} $(\Omega^\bullet(\huaG_{\bullet,\bullet}), d, \delta^v, \delta^h)$ where $d:\Omega^k(\huaG_{j,i})\to\Omega^{k+1}(\huaG_{j,i}) $ is the usual de Rham differential and    $\delta^v:\Omega^k(\huaG_{j,i})\to\Omega^k(\huaG_{j+1,i}),\ \delta^h:\Omega^k(\huaG_{j,i})\to\Omega^{k}(\huaG_{j,i+1}) $ are the simplicial differentials
$$\delta^h=\sum_{l=0}^{i+1} d^{h*}_l,\quad \delta^v=\sum_{l=0}^{j+1} d^{v*}_l, \quad \text{ and }\quad \widetilde{D}=\delta^h+(-1)^i\delta^v+(-1)^{i+j}d$$
is the differential on the total complex.

\begin{definition}
A {\bf $(q,p)$-shifted $k$-form} on a double Lie group $\huaG_{\bullet, \bullet}$ is
$$\alpha_{\bullet,\bullet}=\sum_{j=0}^{q}\sum_{i=0}^p \alpha_{j,i} \quad\text{with}\quad \alpha_{j,i}\in\Omega^{k+p+q-i-j}(\huaG_{j,i}).$$
We say that $\alpha_{\bullet,\bullet}$ is {\bf closed} if $\tilde{D}\alpha_{\bullet, \bullet}=0$. 
\end{definition}

\begin{remark}
Following \cite{Lu-we1} (see also \cite{Mackenzie99}), we call the pair $(\huaG_{\bullet,\bullet}, \omega_{\bullet,\bullet})$ a {\bf symplectic double Lie group} if $\omega_{\bullet,\bullet}$ is a $(1,1)$-shifted $2$-form satisfying  $$\omega_{\bullet,\bullet}=\omega_{1,1}\in\Omega^2(\huaG_{1,1}), \quad \widetilde{D}\omega_{\bullet,\bullet}=0\quad \text{and}\quad \omega^\sharp_{1,1}:T^*\huaG_{1,1}\xrightarrow{\sim}T\huaG_{1,1}.$$  As we will see, the models presented in this section are not symplectic double Lie groups. Therefore a more general definition (that we will not introduce) seems to be missing. 
\end{remark}

\subsection{The models}
If $G$ is a Lie group then the unit groupoid of $G$ as a strict Lie 2-group gives rise to a double Lie group as described in Section \ref{sec:str-lie-2}. More concretely, we have that
\begin{equation} \label{eq:G-double}
   \xymatrix{G\ar@<-0.5 ex>[r]\ar@<0.5 ex>[r]\ar@<-0.5 ex>[d]\ar@<0.5 ex>[d]& pt\ar@<-0.5 ex>[d]\ar@<0.5 ex>[d]\\
     G\ar@<-0.5 ex>[r]\ar@<0.5 ex>[r]&pt}
\end{equation}
is a double Lie group, which we denote by $G_{\bullet,\bullet}$. The vertical structure is given by the unit groupoid of $G$ and horizontal multiplication is given by the multiplication on the Lie group $G$.

If $(\g, \langle\cdot,\cdot\rangle)$ is a quadratic Lie algebra, Theorem \ref{Thm-fin-sym} states that $(NG_\bullet, \Omega_\bullet)$ is a $2$-shifted symplectic Lie $1$-group. Clearly we can define a $(0,2)$-shifted $2$-form $\Omega_{\bullet,\bullet}$ on $G_{\bullet,\bullet}$ by $$\Omega_{\bullet,\bullet}=\begin{pmatrix}0&0&0\\ \Omega&-\Theta&0 \end{pmatrix}, \quad\text{with}\quad \Omega_{0,2}=\Omega\in\Omega^2(G^{\times 2})\text{ and } \Omega_{0,1}=-\Theta\in\Omega^3(G).$$ 

\begin{proposition}\label{2sclosed}
The $(0,2)$-shifted $2$-form $\Omega_{\bullet,\bullet}$ is closed.
\end{proposition}
\begin{proof}
This follows directly form the fact that $\Omega_\bullet$ is closed in $NG_\bullet$, $\delta^h|_{j=0}$ is the simplicial differential of $NG_\bullet$,  and $\delta^v|_{i=0}=0$ since $s^v=t^v=\id$. Hence
$$\widetilde{D}(\Omega_{\bullet,\bullet})=(-1)^i \delta^v(\Omega_{0,\bullet})+D(\Omega_\bullet)=0.$$
\end{proof}


For a given Lie group $G$ we define the Lie groupoid $\Omega G\rightrightarrows P_eG$  with structure maps
\begin{eqnarray}
&s(\tau)(t)=\tau(\frac{t}{2}), \quad t(\tau)(t)=\tau(1-\frac{t}{2}),\quad i(\tau)(t)=\tau(1-t), \label{dou-st}  \\
   & m(\tau_1, \tau_2)(t)=\left\{\begin{array}{ll}
        \tau_2(t) & t\in [0,\frac{1}{2}], \\
         \tau_1(t) & t\in [\frac{1}{2},1], 
    \end{array}\right.\quad \text{and}\quad u(\gamma)(t)=\left\{\begin{array}{ll}
        \gamma(2t) & t\in [0,\frac{1}{2}],  \\
         \gamma(1-2t) & t\in [\frac{1}{2},1].
    \end{array}\right.\label{dou-mul}
\end{eqnarray}
As in the previous section $\Omega G$ and $P_eG$ are completed in an appropriate Sobolev norm so  that the structure maps are smooth. Moreover, it is not hard to verify that $\Omega G$ and $P_eG$ are infinite dimensional Lie groups under the point-wise multiplication, and this makes $\Omega G\rightrightarrows P_eG$ into a strict Lie 2-group. 
\begin{proposition}
For a Lie group $G$ we have the double Lie group $\GG_{\bullet,\bullet}$ given by the square  
\begin{equation} \label{eq:omega-g-double}
   \xymatrix{\Omega G\ar@<-0.5 ex>[r]\ar@<0.5 ex>[r]\ar@<-0.5 ex>[d]\ar@<0.5 ex>[d]& pt\ar@<-0.5 ex>[d]\ar@<0.5 ex>[d]\\
     P_eG\ar@<-0.5 ex>[r]\ar@<0.5 ex>[r]&pt}
\end{equation}
with vertical structure maps defined by \eqref{dou-st} and \eqref{dou-mul} and horizontal multiplication given by
$$m^h(\tau_1, \tau_2)(t)=\tau_1(t)\tau_2(t), \quad \forall \tau_1, \tau_2 \in \Omega G \; \text{or} \; P_e G.$$
\end{proposition}
\begin{proof}
The fact that $\GG_{\bullet,\bullet}$ is a double Lie group, i.e. that the vertical source and target are group morphisms and that the multiplications commute, follows by inspection.
\end{proof}

\begin{remark} \label{remark:equivalence-db}
When $G$ is connected and simply connected, the quotient stack $[P_e G/\Omega G]\cong G$ is representable. Thus, the strict Lie 2-group  $\Omega G\rightrightarrows P_eG$ is Morita equivalent to the Lie 1-group $NG_\bullet$. We may hence view both \eqref{eq:G-double} and \eqref{eq:omega-g-double} as double Lie group models for $\huaB G$.  More explicitly, the Morita equivalence is given through 
\begin{equation}\label{eq:ev-db}
    \begin{array}{ccc}
        ev_{1,1}:\Omega G\to G, &\quad & ev_{0,1}:P_eG\to G,  \\
        ev_{1,1}(\tau)=\tau(\frac{1}{2}), && ev_{0,1}(\gamma)=\gamma(1), 
    \end{array}
\end{equation} which may be further extended to a double Lie group morphism $ev_{\bullet,\bullet}:\GG_{\bullet,\bullet}\to G_{\bullet,\bullet}$. 
\end{remark}

\emptycomment{
With this second model we can also compute $NG$ and hence gives also another model for $BG$ in the following way.
\begin{proposition}
Let $G$ be a connected and simply connected Lie group. Then the differentiable stack associated to the groupoid  $\Omega G\rightrightarrows P_eG$ is the manifold $G$, the map $ev_1:P_eG\to G$ given by $ev_1(\gamma)=\gamma(1)$ is an atlas for the stack and the multiplication $m^h:\Omega G\times \Omega G\to \Omega G$ descends to the multiplication $m:G\times G\to G$.
\end{proposition}
} 
When $(\g, \langle\cdot,\cdot\rangle)$ is a quadratic Lie algebra, the double Lie group $\GG_{\bullet,\bullet}$ is endowed with the Segal's $2$-form  $\omega\in\Omega^2(\Omega G)=\Omega^2(\GG_{1,1})$ defined in \eqref{segal-2form}. But $\omega$ can not be multiplicative with respect to the group structure (otherwise it will be an example of a symplectic group). Therefore,  
in order to obtain a closed form on $\GG_{\bullet,\bullet}$ we need to introduce a new term. Define the $1$-form
\begin{equation}\label{eq:eta}
    \eta_{(\tau_1,\tau_2)}((a_1,a_2))=\int_0^1\langle R_{\tau_2(t)^{-1}}\tau_2'(t),\widehat{a}_1(t)\rangle dt\ \in\Omega^1(\Omega G^{\times 2})=\Omega^1(\GG_{1,2}), 
\end{equation}
where $\tau_i\in\Omega G, \ a_i\in T_{\tau_i}\Omega G$,  and $\widehat{a}_1(t)=L_{\tau_1(t)^{-1}}a_1(t).$

\begin{proposition}
The double Lie group $\GG_{\bullet,\bullet}$ has a closed $(1,2)$-shifted $1$-form $\omega_{\bullet,\bullet}$  defined by 
$$\omega_{\bullet,\bullet}=\begin{pmatrix}-\eta&\omega&0\\ 0&0&0 \end{pmatrix} \quad \text{with}\quad \omega_{1,2}=-\eta\in\Omega^1(\GG_{1,2})\ \text{ and }\ \omega_{1,1}=\omega\in\Omega^2(\GG_{1,1}).$$
\end{proposition}

\begin{proof}
The result will follow from Theorem \ref{Thm-dou-ev} which states that
$$\omega_{\bullet,\bullet}=-\frac{1}{2}\big(\widetilde{D}(\alpha_{\bullet,\bullet})+ev^*_{\bullet,\bullet}\Omega_{\bullet,\bullet}\big), $$
and the fact that $\Omega_{\bullet,\bullet}$ is closed by Proposition \ref{2sclosed}.
\end{proof}

It is quite surprising that in the double picture,  we obtain a $1$-form \eqref{eq:eta} instead of a $2$-form as in Theorem \ref{thm-morita} of the simplicial picture, to bridge the finite and infinite models. We believe that the term $\eta$ should be related to the descent equations computed in \cite{ANXZ}.

\subsection{The equivalence}
Given a Lie group with a quadratic Lie algebra we create two different double Lie groups endowed with differential forms $(G_{\bullet,\bullet}, \Omega_{\bullet,\bullet})$ and $(\GG_{\bullet,\bullet}, \omega_{\bullet,\bullet})$. Here we will show that they are equivalent. 

As stated in Remark \ref{remark:equivalence-db}, there is a double Lie group morphism $ev_{\bullet,\bullet}:\GG_{\bullet,\bullet}\to G_{\bullet,\bullet}$ given by \eqref{eq:ev-db}. As in the simplicial case, the forms $\omega_{\bullet,\bullet}$ and $ev^*_{\bullet,\bullet}\Omega_{\bullet,\bullet}$ do not agree.  So we need to introduce another form on $\GG_{\bullet,\bullet}$, which we denote $\alpha_{\bullet,\bullet}$. The $(1,2)$-shifted $0$-form  $\alpha_{\bullet,\bullet}$ is defined by 
$$\alpha_{\bullet,\bullet}=\begin{pmatrix}0&-\alpha&0\\-\tr(\Omega)&-\tr(\Theta)&0\end{pmatrix} \quad \text{where}\quad \alpha=\int_0^1\langle L_{\tau(t)^{-1}}\tau'(t),L_{\tau(t)^{-1}}v(t)\rangle\in\Omega^1(\Omega G)$$ 
and $\tr(\Omega)\in\Omega^1(P_eG^{\times 2})=\Omega^1((P_eG)^{\times 2})=\Omega^1(\GG_{0,2})$, and $\tr(\Theta)\in\Omega^2(P_eG)=\Omega^2(\GG_{0,1})$ are the transgressions of the forms on $NG_\bullet$ introduced in \eqref{Omegadot}.
\begin{remark}
The $1$-form $\alpha\in\Omega^1(\Omega G)$ also has a nice interpretation in terms of the $S^1$-action on $\Omega G$. The based loop group $\Omega G$ carries an $S^1$-action given by rotation of loops and we denote its infinitesimal generator by $X_{S^1}(\tau)=\tau'.$ Then the fixed points of the action correspond to the critical points of the function $\alpha(X_{S^1})$.  
\end{remark}

\begin{theorem}\label{Thm-dou-ev}
The evaluation map $ev_{\bullet,\bullet}:\GG_{\bullet,\bullet}\to G_{\bullet,\bullet}$ satisfy
$$-\omega_{\bullet,\bullet}-\frac{1}{2}ev^*_{\bullet,\bullet}\Omega_{\bullet,\bullet}=\widetilde{D}(\frac{1}{2}\alpha_{\bullet,\bullet}).$$
\end{theorem}
\begin{proof}
In order to prove this result, we need to show the following equality between matrices, 
\begin{equation*}
    \begin{pmatrix}
    0& 0& 0& 0\\
    0&\eta &-\omega&0\\
    0&-\frac{1}{2}ev_{0,2}^*\Omega&\frac{1}{2}ev^*_{0,1}\Theta&0
    \end{pmatrix}
    =\frac{1}{2}\begin{pmatrix}
    0& 0& \delta^v\alpha& 0\\
    0&-\delta^h\alpha-\delta^v\tr(\Omega) &\delta^v\tr(\Theta)-d\alpha&0\\
    -\delta^h\tr(\Omega)&-\delta^h\tr(\Theta)-d\tr(\Omega)&d\tr(\Theta)&0
    \end{pmatrix}.
\end{equation*}
For the equality in the first row we need to compute $\delta^v\alpha\in\Omega^1(\GG_{2,1})$ and show that it is zero. Recall that $\GG_{2,1}=\Omega G\times_{P_eG}\Omega G$, hence we pick $\tau_1,\tau_2\in\Omega G$ with $\tau_1(\frac{t}{2})=\tau_2(1-\frac{t}{2})$ and $a_i\in T_{\tau_i}\Omega G$ also with $a_1(\frac{t}{2})=a_2(1-\frac{t}{2})$. Then 
\begin{equation*}
    \begin{split}
        &(\delta^v\alpha)_{(\tau_1,\tau_2)}(a_1,a_2)=\alpha_{\tau_2}(a_2)-\alpha_{m^v(\tau_1,\tau_2)}(Tm^v(a_1,a_2))+\alpha_{\tau_1}(a_1)\\
        =&\alpha_{\tau_2}(a_2)+\alpha_{\tau_1}(a_1)
        -\int_0^\frac{1}{2}\langle L_{\tau_2(t)^{-1}}\tau_2'(t), L_{\tau_2(t)^{-1}}a_2(t)\rangle-\int_\frac{1}{2}^1\langle L_{\tau_1(t)^{-1}}\tau_1'(t), L_{\tau_1(t)^{-1}}a_1(t)\rangle\\
        =&\int_\frac{1}{2}^1\langle L_{\tau_2(t)^{-1}}\tau_2'(t), L_{\tau_2(t)^{-1}}a_2(t)\rangle+\int_0^\frac{1}{2}\langle L_{\tau_1(t)^{-1}}\tau_1'(t), L_{\tau_1(t)^{-1}}a_1(t)\rangle=0.
    \end{split}
\end{equation*}
In order to verify the first equality in the middle row, we need to make explicit computations of the differentials. We start by computing $\delta^h\alpha\in \Omega^1(\GG_{1,2})$. Since $\GG_{1,2}=\Omega G^{\times 2}$, we pick $\tau_1,\tau_2\in\Omega G$ and $a_i\in T_{\tau_i}\Omega G$. Then
\begin{equation}\label{H1}
    \begin{split}
        (\delta^h\alpha)_{(\tau_1,\tau_2)}((a_1,a_2))=&\ \alpha_{\tau_2}(a_2)-\alpha_{\tau_1\tau_2}(R_{\tau_2}a_1+L_{\tau_1} a_2)+\alpha_{\tau_1}(a_1)\\
        =&-\int_0^1\langle L_{\tau_1^{-1}(t)}\tau_1'(t), R_{\tau_2^{-1}(t)}a_2(t)\rangle -\eta_{(\tau_1,\tau_2)}((a_1,a_2)).
    \end{split}
\end{equation}
An easy computation shows that $\delta^v\tr(\Omega)=\tr(\Omega)_{|\Omega G}$ and by Proposition \ref{trans-expli} and the formula \eqref{exOme} we get that 
\begin{equation}\label{H2}
    \begin{split}
        \delta^v\tr(\Omega)_{(\tau_1,\tau_2)}((a_1,a_2))=&\tr(\Omega)_{(\tau_1,\tau_2)}((a_1,a_2))=\int_0^1 \Omega_{(\tau_1,\tau_2)}((\tau_1',\tau_2'),(a_1,a_2))\\
        =&\int_0^1\langle L_{\tau_1^{-1}(t)}\tau_1'(t), R_{\tau_2^{-1}(t)}a_2(t)\rangle -\eta_{(\tau_1,\tau_2)}((a_1,a_2)).
    \end{split}
\end{equation}
Hence combining \eqref{H1} and \eqref{H2}, we obtain 
$$\frac{1}{2}(-\delta^v\alpha-\delta^h\tr(\Omega))=\eta.$$
For the second term in the middle row, we  use again the fact that $\delta^v\tr(\nu)=\tr(\nu)_{|\Omega G}$, and by Lemma \ref{lemma-brylinski} we obtain that
$$-\omega=-\omega^P|_{\Omega G}=\frac{1}{2}(\tr(\Theta)-d\alpha^P)|_{\Omega G}=\frac{1}{2}(\delta^v\tr(\Theta)-d\alpha).$$

Finally, the equality on the last row follows directly from the properties of the transgression given in Propositions \ref{Trans-DR} and \ref{trans-delta}, the fact that $ev_{0,1}(\gamma)=\gamma(1)$ and $ev_{0,2}(\gamma_1,\gamma_2)=(\gamma_1(1),\gamma_2(1))$ and that $D\Omega_\bullet=0$. Explicitly
\begin{eqnarray*}
\delta^h\tr(\Omega)&=&-\tr(\delta\Omega)=0,\\
-\frac{1}{2} \delta^h\tr(\Theta)-\frac{1}{2}d\tr(\Omega)&=&-\frac{1}{2}\tr(d\Omega)-\frac{1}{2}ev_1^*\Omega+\frac{1}{2}\tr(d\Omega)=-\frac{1}{2}ev_{0,2}^*\Omega,\\
\frac{1}{2}d\tr(\Theta)&=&\frac{1}{2}ev^*_1\Theta-\frac{1}{2}\tr(d\Theta)=\frac{1}{2}ev^*_{0,1}\Theta.
\end{eqnarray*}
Therefore the two matrices coincides in all the entries, and we have proved the statement.
\end{proof}

\appendix
\renewcommand{\theequation}{\thesection.\arabic{equation}}

\section{Sobolev spaces}\label{ap-sob}
Here we recall some analytic facts about Sobolev spaces used in this article (see also \cite[Sect.4]{Ale-dirac}, \cite[Sect.14]{ Atiyah-Bott:YM}). For a finite-dimensional compact manifold $M$, possibly with boundary and corners, let $H_r(M)$ denote the order $r$ Sobolev space of functions. These functions and their weak derivatives are $L^2$-functions. The $C^\infty$ functions are dense in $H_r(M)$. The point-wise multiplication makes $H_r(M)$ a Banach algebra when $r-\frac{1}{2}\dim M >0$ \cite[Theorem 4.39]{sob}. If $Z \subset M$ is a submanifold, and $r-\frac{1}{2}\codim(Z)> 0$, then the restriction of continuous functions from $M$ to $Z$ extends to a continuous linear map $H_r(M)\to H_{r-\frac{1}{2} \codim(Z)}(Z)$,    with a continuous right inverse.

If $N$ is another finite-dimensional manifold and $r-\frac{1}{2} \dim(M) > 0$, one defines spaces $\Hom_r(M,N)$ of maps from $M$ to $N$ of Sobolev class $r$ by choosing local charts
for $N$. In particular, if $G$ is a finite-dimensional Lie group and $r-\frac{1}{2} \dim M > 0$, then $\Hom_r(M, G)$ is a Banach (even Hilbert) Lie group under point-wise multiplication \cite[Sect.4]{Ale-dirac}. 

We are particularly interested in the loop group $L G := \Hom_r(S^1, G)$ and the based loop group $\Omega G:=\{\tau \in LG| \tau(0)=e\}$ for a fixed $r \in \Z^{\ge 1}$. In addition to the advantage of having a Banach manifold, using the version of based loops with Sobolev completion is also helpful as  sometimes our construction of face or degeneracy maps via concatenation does not result in a smooth map, but instead a map in the Sobolev completion.

Given a domain $U$ in $\R^n$, the map sending $f\in C^\infty(U)$ to the evaluation of its $m$-th derivative  $f^{(m)}(x)$ at point $x\in U$ can be extended to a bounded linear (thus smooth) map  on $H_r(U)$ if $r>m$.  Therefore, the de Rham differentiation operator is  bounded linear (thus smooth) $H_{r+1}(U) \to H_{r}(U)$. 

\section{Universal integration $\int \g_\bullet$ and truncations}\label{ap:int-trun}
Using the idea of Sullivan's space realisation and a suitable truncation, Henriques shows in \cite{henriques} a procedure to integrate a Lie $n$-algebra, i.e. an $n$-term $L_\infty$-algebra, to a Lie $n$-group\footnote{Notice that the truncation procedure creates a possible obstruction to this integration procedure. That is, although a finite dimensional Lie algebra can be always integrated to a Lie group, not every Lie $n$-algebra can be integrated to a Lie $n$-group.}. In general, the Lie $n$-groups obtained by this method are infinite dimensional. Here we recall his construction in the special case of a Lie algebra.

Let $\g=(\g, [\cdot,\cdot])$ be a Lie algebra and $\CE(\g)$ its Chavelley-Eilenberg differential complex, that is $\CE(\g)=\wedge^\bullet \g^*$ with differential
\[
d_{CE}\xi (x_1, \dots, x_k) =\sum_{i<j} (-1)^{i+j-1} \xi ([x_i, x_j], x_1, \dots, \hat{x}_i, \dots, \hat{x}_j, \dots, x_k).
\]

The universal object integrating it, $\int \g_\bullet$, was constructed in \cite{Getzler04, henriques} in the following way (we also follow the treatment in \cite{henriques, Ale-dirac} of the differential structure, which is also similar to that in \cite{brahic-zhu}). For any $k\in\Z^{\ge 0}$ consider the standard $k$-simplex 
\begin{equation*}
    \Delta^k=\{(t_0, \dots, t_k)\in \R^{k+1}|\sum_{i=0}^k t_i=1\}.
\end{equation*}
Denote by $\Omega^\bullet(\Delta^k)$  subspaces of  de Rham forms on $\Delta^k$ with Sobolev class $r$ (to be a Banach algebra we need $r>\frac{1}{2}k$ as in Appendix \ref{ap-sob}) defined by
\[
\Omega^\bullet(\Delta^k):=\{\alpha|\ \alpha = \sum_{I=i_1<\dots <i_k}\alpha^I dt_I  \text{ where } \alpha^I  \; \text{is of Sobolev class}\; r, \text{ and } d\alpha \; \text{is also in this form}\footnote{This means that either $\alpha$ has Sobolev class $r+1$ or $d\alpha=0$.} \}. 
\]
This makes $(\Omega^\bullet(\Delta^k), d)$ into a differential graded algebra, which we abbreviate to d.g.a..
Denote by $\Hom_{d.g.a.}$  the set of morphisms between  d.g.a.s.   Then $\Hom_{d.g.a.}(\CE(\g), \Omega^\bullet(\Delta^k))$ carries a natural Banach manifold structure \cite[Theorem 5.10]{henriques}. Notice that $C^\infty$ functions do not form a Banach space, but its completion under a Sobolev norm to a Sobolev space is a Banach space.

It can be helpful to view elements in $\Hom_{d.g.a.}\big(\CE(\g), \Omega^\bullet (\Delta^k) \big)$ as Lie algebroid morphisms from $T\Delta^k$ to $\g$, and we define
\[\Hom_{\algd}(T\Delta^k, \g):=\Hom_{d.g.a.}\big(\CE(\g), \Omega^\bullet (\Delta^k) \big).\]
More precisely, a vector bundle morphism $\psi:T\Delta^k \to \g$ can be written explicitly as $\psi=\sum_{i=0}^k \psi_i dt_i$ with $\psi_i \in C^{r+1}(\Delta^k, \g)$. Moreover, $\psi$ defines an element in $\Hom_{\algd}(T\Delta^k, \g)$ if it is further a Lie algebroid morphism, that is it satisfies the Maurer-Cartan equation
\[
\frac{d\psi_i}{dt_j}- \frac{d\psi_j}{dt_i}=[\psi_i, \psi_j], \quad \forall i, j \in \{0, \dots, k\}. 
\]

The Banach manifolds  $\Hom_{\algd} (T\Delta^\bullet, \g)$ form a simplicial manifold, denoted by $\int \g_\bullet$ in \cite{henriques}, with face and degeneracy maps induced by the natural ones between $\Delta^k$ and $\Delta^{k-1}$. The simplicial manifold $\int \g_\bullet$ is conjectured \cite[Section 7]{henriques-v1} to be a universal integration object of $\g$, i.e. the $L_\infty$-group integrating $\g$, partially because its 1-truncation is the universal\footnote{which is only universal among Lie (1-)groups integrating $\g$} (connected and simply connected) Lie group $G$ integrating $\g$. Moreover, $\int$ is an exact functor with respect to a class of distinguished fibrations---``quasi-split fibrations'' \cite[Section 9]{Rogers-Zhu:2016}. Such fibrations include acyclic fibrations as well as fibrations that arise in string-like extensions. In particular, $\int$ sends $L_\infty$ quasi-isomorphisms to weak equivalences, quasi-split fibrations to Kan fibrations, and preserves acyclic fibrations as well as pullbacks of acyclic/quasi-split fibrations.

Now let us give some details on truncations.


\begin{definition}\label{def:truncation} [See e.g. \cite{may} for the case of simplicial sets]
Given a simplicial manifold $X_\bullet$ the {\bf $n$-truncation} $\tau_n(X)_\bullet$ is the simplicial set defined as
\[
\tau_n(X)_k= S_k, \quad \text{if }\; k\le n-1, \quad \tau_n(X)_k=X_k/\sim^n_{k+1} \text{ if }\; k\ge n, 
\]
where $x_k\sim^n_{k+1} x'_k$ if and only if they are simplicial homotopic relative to $(n-1)$-skeleton. In other words, $sk_{n-1}(x_k)=sk_{n-1}(x'_k)$ and   $\exists\ x_{k+1}\in \hom(\Delta[k]\times \Delta[1], X_\bullet)$  such that
\begin{equation*}
     x_{k+1}|_{sk_{n-1}(\Delta[k])}=sk_{n-1}(x_k)=sk_{n-1}(x'_k), \quad  x_{k+1}|_{\Delta[k]\times 0}=x_k\quad\text{and}\quad x_{k+1}|_{\Delta[k]\times 1}=x'_k.
\end{equation*}
\end{definition}

We pay special attention to the $n$-truncations for the simplicial manifold $\int \g_\bullet$ when $n=1, 2$. In this case, the homotopies that we mod out can be understood in the following way: $\psi^i\in \int\g_k$ with $\psi^0|_{T\partial\Delta^k}=\psi^1|_{T\partial\Delta^k}$ for $i=0,1$ are homotopic if there exists $\Psi\in \Hom_{\algd}(T(\Delta^{k}\times I),\g)$ such that
$$\Psi(x, 0)=\psi^0, \quad \Psi(x, 1)=\psi^1, \quad \Psi(y, t)=\psi^0(y)=\psi^1(y) \text{ for } y\in sk_{n-1} \Delta^k, k\ge n.  $$
Notice in particular that when $k=n$, we have $sk_{n-1} \Delta^k=\partial \Delta^k$.
 

The $n$-truncation of an arbitrary simplicial manifold needs not be a simplicial manifold as there are quotients involved. In our concrete case,  \cite[Theorem 7.5]{henriques} implies that $\tau_1(\int \g)_\bullet$ and $\tau_2(\int \g)_\bullet$ are simplicial manifolds because $\pi_2(G)=0$ for a finite dimensional Lie group. More explicitly we have the following.
\begin{proposition}(\cite[Example 7.2]{henriques})\label{1-trun}
Let $G$ be the connected and simply connected Lie group integrating $\g$. Then $\tau_1(\int \g)_\bullet=NG_\bullet$.
\end{proposition}

\begin{proposition}[\cite{henriques}]\label{2truncation-thm}
Let $G$ be the connected and simply connected Lie group integrating $\g$. Then the Lie 2-group $\tau_2(\int\g)_\bullet$ is equal to $\GG_\bullet$ where $\GG_k=\Hom_e( sk_1 \Delta^k, G)$, which is the space of maps of Sobolev class  $r+1$ sending $(0, \dots, 0, 1)\in sk_1 \Delta^k$ to $e$.  The face and degeneracy maps are those induced from the simplices.
\end{proposition}
\begin{proof}
The calculation is more or less done in \cite[Sect.8]{henriques}. We summarise here: let us denote by $\Hom_e(\Delta^k, G)$ the space of maps of Sobolev class ${r+1}$ which send $(0, \dots, 0, 1)\in \Delta^k$ to $e\in G$. When $k=1$, we also write $P_eG :=\Hom_e(\Delta^1, G)$. Then according to \cite[Example 5.5]{henriques} or \cite[Remark 3.8]{brahic-zhu}, 
$\Hom_{\algd}(T\Delta^k, \g) = \Hom_e(\Delta^k, G) $. Thus $\int \g_1=P_e G$.

The equivalence relation $\sim_k$ we take in truncation is exactly the homotopy relative to the boundary in $G$. Thus $\GG_k=\Hom_e( sk_1 \Delta^k, G)$. 
\end{proof}

\section{Explicit formulas for the Lie $2$-group $\GG_\bullet$}\label{GGtanget}

 The tangent of the Lie $2$-group $\GG_\bullet$ is also a Lie $2$-group\footnote{In fact the group structure is compatible with the vector bundle structure and hence it is more than just a Lie $2$-group.} and is given by 
$$T\GG_\bullet=\cdots \Omega TG\ \aaar P_{(e,0_e)} TG\rightrightarrows pt.$$
The 3rd  level  is also important and is 
\begin{equation}\label{TG3}
     \begin{split}
     T\GG_3=\{( a_0, a_1, a_2)\in(\Omega TG)^{\times 3}\ |\ \text{ for } t\in[0,\frac{1}{3}], \ a_0(t)=a_1(t),  \\ 
     a_0(t+\frac{2}{3})=a_2(\frac{1}{3}-t),\ a_1(t+\frac{2}{3})=a_2(t+\frac{2}{3}) \}.
     \end{split}
\end{equation}
The face and degeneracies are obtained by taking variations of equations \eqref{2faces}, \eqref{3faces}, and \eqref{1degeneraci}. More explicitly, the faces $Td_i:\Omega TG\to P_{(e,0_e)}TG$ are given by 
\begin{equation}\label{T2faces}
    \begin{split}
       &Td_0(a)=a(\frac{t}{3}), \quad Td_1(a) = a(1-\frac{t}{3}), \\
       &Td_2(a)=R_{\tau(\frac{1}{3})^{-1}}a(\frac{1+t}{3})-L_{\tau(\frac{t+1}{3})\tau(\frac{1}{3})^{-1}} R_{\tau(\frac{1}{3})^{-1}} a(\frac{1}{3}). 
    \end{split}
\end{equation}
The faces $Td_i:T\GG_3\to \Omega TG$ are given by
\begin{equation}\label{T3faces}
     Td_i(a_0,a_1, a_2)=a_i\quad \text{ for } 0\leq i\leq 3,
\end{equation}
where 
\begin{equation}\label{V3}
    a_3(t)=\left\{\begin{array}{ll}
        R_{\tau_0(\frac{1}{3})^{-1}}a_0(t+\frac{1}{3})-L_{\tau_3(t)}R_{\tau_0(\frac{1}{3})^{-1}}a_0(\frac{1}{3}) & t\in[0,\frac{1}{3}],  \\
         R_{\tau_0(\frac{1}{3})^{-1}}a_2(t)-L_{\tau_3(t)}R_{\tau_0(\frac{1}{3})^{-1}}a_0(\frac{1}{3})& t\in[\frac{1}{3},\frac{2}{3}], \\
         R_{\tau_0(\frac{1}{3})^{-1}}a_1(\frac{4}{3}-t)-L_{\tau_3(t)}R_{\tau_0(\frac{1}{3})^{-1}}a_0(\frac{1}{3})& t\in[\frac{2}{3},1]. 
    \end{array}\right. 
\end{equation}

The degeneracies $Ts_i:P_{(e,0_e)}TG\to \Omega TG$ are given by
\begin{equation}\label{Tdegeneracies}
    Ts_0(u)(t)=\left\{
    \begin{array}{ll}
         u(3t) & t\in[ 0,\frac{1}{3}],  \\
         u(1) & t\in[\frac{1}{3},\frac{2}{3}],\\
         u(3-3t)&t\in[\frac{2}{3},1], 
     \end{array}\right.
\quad 
    Ts_1(u)(t)=\left\{\begin{array}{ll}
        u(0) &t\in[ 0,\frac{1}{3}],  \\
        u(3t-1)&  t\in[\frac{1}{3},\frac{2}{3}],\\
        u(3-3t)& t\in[\frac{2}{3},1].
     \end{array}\right.
\end{equation}

Here we also include a direct proof that Segal's 2-form is multiplicative.  
\begin{proposition}\label{P-Closed}
The $2$-shifted $2$-form $\omega_\bullet$ on $\GG_\bullet$ is closed, that is,
$$d\omega=0\qquad \text{and}\qquad \delta\omega=0.$$
\end{proposition}
\begin{proof}
Since $\omega$ is a symplectic form on $\Omega G$ we know that $d\omega=0$,  so it remains to see that $\delta\omega=0$.  We need to compute the pullback of $\omega$ by the face maps \eqref{3faces}. Pick $\tau\equiv(\tau_0,\tau_1, \tau_2)\in\GG_3$ and $a\equiv(a_0,a_1,a_2), b\equiv(b_0,b_1,b_2)\in T_\tau\GG_3$ write $\widehat{a}_i(t)=L_{\tau_i(t)^{-1}}a_i(t)$ and similarly for $b$. Then using \eqref{3faces} and \eqref{T3faces} we obtain 
$$(d_i^*\omega)_\tau(a,b)=\omega_{\tau_i}(a_i,b_i)=A_i+B_i+C_i\quad \text{for } i=0,\cdots,3,$$
where $A_i$ correspond to the first third of the loop, $B_i$ to the second third and $C_i$ to the last third. Since $C^\infty$ maps are dense in Sobolev space, we only need to verify our result at smooth loops. In order to compare the terms $A_i, B_i$, and $C_i$ for different values of $i$, we use \eqref{TG3} and the face maps given by \eqref{T3faces}. Moreover, equation \eqref{V3} implies 
\begin{equation}\label{v3inv}
\left\{
    \begin{array}{ll}
        a_0(t)=R_{\tau_0(\frac{1}{3})}a_3(t-\frac{1}{3})+L_{\tau_3(t-\frac{1}{3})} a_0(\frac{1}{3}) & t\in[\frac{1}{3},\frac{2}{3}],  \\
         a_2(t)=R_{\tau_0(\frac{1}{3})}a_3(t)+L_{\tau_3(t)}a_0(\frac{1}{3}) & t\in[\frac{1}{3},\frac{2}{3}],  \\
         a_1(t)=R_{\tau_0(\frac{1}{3})}a_3(\frac{4}{3}-t)+L_{\tau_3(\frac{4}{3}-t)}a_0(\frac{1}{3}) & t\in[\frac{1}{3},\frac{2}{3}],  \\
    \end{array}\right.
\end{equation}
and the same identities holds  for $b$. 

For $i=0$ we see that the first term is given by
\begin{equation*}
    A_0= \int_0^{\frac{1}{3}}\langle \widehat{a}'_0(t),\widehat{b}_0(t)\rangle dt 
    =  \int_0^{\frac{1}{3}}\langle\widehat{a}'_1(t),\widehat{b}_1(t)\rangle dt = A_1, 
    \end{equation*}
 where the middle equality follows from the definition of $T\GG_3$. The second term can be rewritten using \eqref{v3inv} as
 \begin{equation*}
         B_0= \int_{\frac{1}{3}}^{\frac{2}{3}}\langle\widehat{a}'_0(t),\widehat{b}_0(t)\rangle dt 
         =  \int_{\frac{1}{3}}^{\frac{2}{3}}\langle\alpha,\beta\rangle
 \end{equation*}
 with 
 \begin{equation*}
        \alpha=  \Big(L_{\tau_0(\frac{1}{3})^{-1}\tau_3(t-\frac{1}{3})^{-1}}\big(R_{\tau_0(\frac{1}{3})}a_3(t-\frac{1}{3})+L_{\tau_3(t-\frac{1}{3})}a_0(\frac{1}{3})\big)\Big)'
         =Ad_{\tau_0(\frac{1}{3})^{-1}}\widehat{a}'_3(t-\frac{1}{3}) 
 \end{equation*}
 and
 \begin{equation*}
        \beta=    L_{\tau_0(\frac{1}{3})^{-1} \tau_3(t-\frac{1}{3})^{-1}}\Big(R_{\tau_0(\frac{1}{3})}b_3(t-\frac{1}{3})+L_{\tau_3(t-\frac{1}{3})}b_0(\frac{1}{3})\Big)
         =  Ad_{\tau_0(\frac{1}{3})^{-1}}\widehat{b}_3(t-\frac{1}{3})+\widehat{b}_0(\frac{1}{3}).
 \end{equation*}
 Setting $s=t-\frac{1}{3}$, and noticing that $\langle -, -\rangle$ is adjoint invariant,  we conclude that
 \begin{equation*}
    B_0=  \int_0^\frac{1}{3}\langle\widehat{a}'_3(s),\widehat{b}_3(s)\rangle +\langle Ad_{\tau_0(\frac{1}{3})^{-1}}\widehat{a}'_3(s),\widehat{b}_0(\frac{1}{3})\rangle ds 
     = A_3+ \int_0^\frac{1}{3}\langle \widehat{a}'_3(s),R_{\tau_0(\frac{1}{3})^{-1}}b_0(\frac{1}{3})\rangle ds.
 \end{equation*}
 Finally the third term
 \begin{equation*}
        C_0=   \int_{\frac{2}{3}}^1\langle\widehat{a}'_0(t),\widehat{b}_0(t)\rangle dt 
          =\int_{\frac{2}{3}}^1\langle\widehat{a}'_2(1-t),\widehat{b}_2(1-t)\rangle dt 
          =-\int_0^{\frac{1}{3}}\langle\widehat{a}'_2(s),\widehat{b}_2(s)\rangle ds=-A_2,
 \end{equation*}
 where in the middle equation we use the definition of $T\GG_3$ and in the last we make the change of variable $s=1-t$. By similar computations one shows that
 $$ \begin{array}{rl}
     B_1= & -C_3-\int_\frac{2}{3}^1\langle \widehat{a}'_3(s),R_{\tau_0(\frac{1}{3})^{-1}}b_0(\frac{1}{3})\rangle ds, \\
      C_1=&C_2, \\
      B_2=&B_3+\int_\frac{1}{3}^\frac{2}{3}\langle  \widehat{a}'_3(s), R_{\tau_0(\frac{1}{3})^{-1}}b_0(\frac{1}{3})\rangle ds. 
 \end{array}$$
 Thus we may conclude that 
 \begin{equation*}
    \delta\omega_\tau(a,b)= \sum_{i=0}^3(-1)^i (d_i^*\omega)_\tau(a,b)
     =\sum_{i=0}^3(-1)^i( A_i+B_i+C_i) 
     =\int_0^1\langle \widehat{a}'_3(s),R_{\tau_0(\frac{1}{3})^{-1}}w_0(\frac{1}{3})\rangle ds =0,
 \end{equation*}
 where in the last step we are using Stokes' Theorem and the fact that $\widehat{a}_3(0)=\widehat{a}_3(1)=0_e.$
\end{proof}

\section{Transgression map}\label{Ap-trans}
In this appendix we recollect some useful formulas for the transgression map going from  $k$-forms on a manifold to $(k-1)$-forms on its path space, which should be well known to experts.

Let $M$ be a manifold and denote by $PM$ its path space, i.e. 
$PM=\{\gamma:[0,1]\to M \ |\ \gamma \text{ of class } H_r \}$, 
and consider $PM$ as a Banach manifold. The evaluation map $ev: [0,1]\times PM\to M,\ ev(t,\gamma)=\gamma(t)$ at the level of forms produces the map $$ev^*:\Omega^k(M)\to \Omega^k([0,1]\times PM),$$
and since the interval has dimension $1$ we can give the more explicit description
\begin{equation}\label{formsPM}
   \Omega^k([0,1]\times PM) =\Omega^0([0,1])\otimes\Omega^k(PM)\oplus \Omega^1([0,1])\otimes\Omega^{k-1}(PM).
\end{equation}
The {\bf transgression map} is defined as the composition of evaluation with integration:
\begin{equation}
    \begin{array}{cccl}
        \tr: & \Omega^k(M)&\to&\Omega^{k-1}(PM) \\
         & \omega& \mapsto & \tr(\omega)=\int_0^1 ev^*\omega.
    \end{array}
\end{equation}

Now we give short proofs for the properties of the transgression map that we use in the article.

\begin{proposition}\label{Trans-DR} 
The transgression map and de Rham differential satisfy the following relation:
$$ \tr(d\omega)=ev_1^*\omega-ev_0^*\omega - d_{PM} \tr(\omega),$$
where $ev_t(\gamma)=\gamma(t)$ for $t\in[0,1].$
\end{proposition}
\begin{proof}
\begin{equation*}
    \begin{split}
        \tr(d\omega)=&\int_0^1 ev^*d\omega=\int_0^1 d_{[0,1]\times PM}(ev^*\omega)
        =\int_0^1 d_{[0,1]}ev^*\omega-\int_0^1 d_{PM}(ev^*\omega)\\
        =& ev^*_1\omega- ev^*_0\omega-d_{PM}\tr(\omega),
    \end{split}
\end{equation*}
where we use the decomposition \eqref{formsPM} and Stokes' theorem.
\end{proof}

\begin{proposition}\label{trans-expli}
In terms of vectors the transgression map has the following explicit formula:
$$\tr(\omega)_\gamma(v_1,\cdots, v_{k-1})=\int_0^1\omega_{\gamma(t)}\big(\gamma'(t), v_1(t),\cdots, v_{k-1}(t)\big) dt,$$
where $\omega\in \Omega^k(M)$ and $v_i\in T_\gamma PM.$
\end{proposition}
\begin{proof}
With decomposition \eqref{formsPM}, we  write
\begin{equation}
    \label{eq:ev*omega} ev^*\omega = \omega_0\otimes \omega_k + \omega_1 \otimes \omega_{k-1}, 
\end{equation} 
where $\omega_k\in \Omega^k(PM) $, $\omega_{k-1}\in \Omega^{k-1}(PM) $, $\omega_1 \in \Omega^1([0,1])$ and $\omega_0\in \Omega^0([0,1])$.  Then contracting with $[0,1]$ via $\int_0^1$, only the second term survives. Thus for tangent vectors $x_i\in T_\gamma PM$, $i=1, \dots, k-1$, at $\gamma\in PM$, we have
\begin{equation} \label{eq:intIev}
    \int_0^1 ev^* \omega_\gamma (v_1, \dots, v_{k-1}) = \int_0^1 \omega_1 \otimes \omega_{k-1} (v_1, \dots, v_{k-1}). 
\end{equation}
At the same time, for a tangent vector $(w, v) \in T_t[0,1]\times T_\gamma PM$, we have $T_\gamma ev(w,v)=w\gamma'+v$. This can be seen by taking a variation (a small path) $(t+w\epsilon,\gamma^\epsilon)$ representing $(w,v)$ (thus $\gamma^0=\gamma$). Then
$$
T_\gamma ev(w, v) = \frac{d}{d\epsilon}\bigg|_{\epsilon=0} (\gamma^\epsilon(t+w\epsilon)) = \frac{d}{d\epsilon}\bigg|_{\epsilon=0} \gamma^0(t+w\epsilon) +\frac{d}{d\epsilon}\bigg|_{\epsilon=0} \gamma^\epsilon(t) = w\gamma'(t) + v(t).
$$
Thus
\[
\begin{split}
   &ev^*\omega|_{(t,\gamma)}((w_1,v_1),…,(w_k,v_k))=\omega_{\gamma(t)}(T_\gamma ev(w_1,v_1),…,T_\gamma ev(w_k,v_k)) \\
= &\omega_{\gamma(t)}(v_1(t),…,v_k(t))+\omega_{\gamma(t)}(w_1\gamma'(t),v_2(t),…,v_k(t))+c.p.. 
\end{split}
\]
Thus comparing with \eqref{eq:ev*omega}, we have $\omega_k=\omega$, $\omega_0=1$, $\omega_{k-1}=\iota(\gamma')\omega$, and $\omega_1=dt$. Combining with \eqref{eq:intIev} obtains the desired formula.
\end{proof}

Suppose that we have a simplicial manifold $X_\bullet$. Then the path space $PX_\bullet$ is again a simplicial manifold given by
$$ (PX)_n=P(X_n),\quad (Pd)_i(\gamma)(t)=(d_i\circ\gamma)(t),\quad \text{and}\quad (Ps)_i(\gamma)(t)=(s_i\circ\gamma)(t).$$

\begin{proposition}\label{trans-delta}
Let $X_\bullet$ be a simplicial manifold. Then the transgression commutes with the simplicial differentials, i.e. 
$$ \tr(\delta^X\omega)=\delta^{PX}\tr(\omega).$$
\end{proposition}
\begin{proof}
In order to prove this, we first observe that if $X_\bullet$ is a simplicial manifold then $TX_\bullet$ is also a simplicial manifold with faces and degeneracies given by the corresponding tangent maps. It then follows that there is a canonical identification between $P(TX)_\bullet$ and $T(PX)_\bullet$ as simplicial manifolds. Using this canonical identification and the explicit formula of the transgression given by Proposition \ref{trans-expli}, for $\omega\in\Omega^k(X_{n-1}),$ $\gamma\in PX_{n},$ and $v_j\in T_{\gamma}PX_n$ we have
\begin{equation*}
    \begin{split}
       &\big((Pd_i)^*\tr(\omega)\big)_\gamma(v_1,\cdots, v_{k-1})=\tr(\omega)_{Pd_i(\gamma)}\Big(TPd_i(v_1),\cdots, TPd_i(v_{k-1})\Big)dt\\
       &=\int_0^1\omega_{Pd_i(\gamma)(t)}\Big(Pd_i(\gamma)'(t),TPd_i(v_1)(t),\cdots, TPd_i(v_{k-1})(t)\Big)dt\\
       &=\int_0^1\omega_{d_i(\gamma(t))}\Big(Td_i(\gamma(t))',PTd_i(v_1)(t),\cdots, PTd_i(v_{k-1})(t)\Big)dt\\
       &=\int_0^1\omega_{d_i(\gamma(t))}\Big(Td_i(\gamma'(t)),Td_i(v_1(t)),\cdots, Td_i(v_{k-1}(t))\Big)\\
       &=\int_0^1(d^{*}_i\omega)_{\gamma(t)}\Big(\gamma'(t),v_1(t),\cdots, v_{k-1}(t)\Big)dt = \tr(d^{*}_i\omega)_\gamma(v_1,\cdots, v_{k-1}).
    \end{split}
\end{equation*}
Since the simplicial differentials are alternating sums of face maps, this prove the statement.
\end{proof}

\section{IM-form (by Florian Dorsch)} \label{Ap:IM}
For completeness of this article, we now give proofs for some very useful statements in the unpublished lecture note \cite{Lesdiablerets}, which are stated on page 82-83 without proof.

\begin{lemma}\label{lem:im-form}
Let $\alpha_\bullet$ be a $m$-shifted 2-form on a Lie $n$-groupoid $X_\bullet$. Then
\begin{enumerate}
\item \label{itm:0-s} the IM-form $\lambda^{\alpha_\bullet}$ vanishes on degenerate vectors, that is, for $m\ge 1$, for any point $x\in X_0$ we have 
\begin{equation*}
    \lambda^{\alpha_\bullet}_x(Ts_i u, w) =0, \quad \forall u \in T_{x} X_{p-1},\ w\in T_x X_q, \quad  0\le i \le p-1.
\end{equation*} For $m=0$, this is an empty condition. 
    \item \label{itm:mul} 
    When $m\ge 1$,  for $v\in \huaT_{p+1}(X_\bullet)$, $w\in \huaT_{q}(X_\bullet)$ with $0\le p \le m-1$ and $p+q=m$, we have
\begin{equation}\label{eq:inf-mul-app}
    \lambda^{\alpha_\bullet}(\partial v, w) + (-1)^{p+1} \lambda^{\alpha_\bullet}( v, \partial w )=\lambda^{D\alpha_\bullet}(v, w); 
\end{equation}and when $m=0$, for $v\in \huaT_1(X_\bullet)$, $w\in \huaT_0(X_\bullet)=T_0X_0$ we have 
\begin{equation}  \label{eq:inf-mul-0-app}
    \lambda^{\alpha_\bullet}(\partial v, w)=\lambda^{D\alpha_{\bullet}}(v,w). 
\end{equation} If $\alpha_{\bullet}$ is closed, these are equivalent to \eqref{eq:inf-mul} with the correct interpretation in extreme cases of indices explained in Remark \ref{remark:extreme-case};
    \item \label{itm:gauge} The IM-pairing $\lambda_{x}^{\alpha_{\bullet},\#}:\huaT_{\bullet}|_{x}X\rightarrow \huaT_\bullet^{\ast}|_{x}X[-m]$ is invariant under gauge transformations up to chain homotopy. That is, there exists a chain homotopy 
        \begin{gather*}
\lambda_{x}^{\alpha_{\bullet}+D\phi_{\bullet},\#}\simeq\lambda_{x}^{\alpha_{\bullet},\#}::\huaT_{\bullet}|_{x}X\rightarrow \huaT_\bullet^{\ast}|_{x}X[-m] 
       \end{gather*}
       for any $(m-1)$-shifted 2-form $\phi_{\bullet}$. 
\end{enumerate}
\end{lemma}
\begin{proof}
\eqref{itm:0-s} To show that $\lambda^{\alpha_\bullet}$ vanishes on degeneracies, it is enough to verify that 
\begin{equation} \label{blabla}
    \alpha_m(Ts_{\pi(p+q-1)}\dots Ts_{\pi(p)}Ts_iu,Ts_{\pi(p-1)}\dots Ts_{\pi(0)}w)=0
\end{equation}
for all $u \in T_{x}X_{p-1}$, $w\in T_{x}X_q$ and $s_i:X_{p-1}\rightarrow X_p,\,\,0\leq i\leq p-1$. We begin by making the following observation. Let
\begin{gather*}
    i_{p-1}>\dots> i_j=\lambda>\dots>i_0\:\:\:\text{and}\\
    i_{p+q-1}>\dots>i_{p+l}>\lambda>i_{p+l-1}>\dots>i_p
\end{gather*} be indices such that
\begin{gather*}
    Ts_{i_{p-1}}\dots Ts_{\lambda}\dots Ts_{i_0}w\:\:\:\text{and}\:\:\:
    Ts_{i_{p+q-1}}\dots Ts_{i_{p+l}}Ts_{\lambda}Ts_{i_{p+l-1}}\dots Ts_{i_p}u
\end{gather*}
are well-defined tangent vectors in $T_{x}X_m$. Then by simplicial identities we have 
\begin{equation}\label{eq:observation}
\begin{split}
    &\OO_m(Ts_{i_{p+q-1}}\dots Ts_{i_{p+m}}Ts_{\lambda}Ts_{i_{p+m-1}}\dots Ts_{i_p}u,Ts_{i_{p-1}}\dots Ts_{\lambda}\dots Ts_{i_0}w)\\
     =&\OO_m(Ts_{\lambda}Ts_{i_{p+q-1}-1}\dots Ts_{i_{p+m}-1}Ts_{i_{p+m-1}}\dots Ts_{i_p}u,Ts_{\lambda} Ts_{i_{p-1}-1}\dots Ts_{i_{j+1}-1} \dots Ts_{i_0}w)\\
    =&s_{\lambda}^{\ast}\OO_m(Ts_{i_{p+q-1}-1}\dots Ts_{i_{p+m}-1}Ts_{i_{p+m-1}}\dots Ts_{i_p}u,Ts_{i_{p-1}-1}\dots Ts_{i_{j+1}-1} \dots Ts_{i_0}w) \\=& 0
\end{split}
\end{equation}
as $\OO_m$ is normalized. The rest of argument is essentially to reduce the situations to \eqref{eq:observation}.\\
First suppose that $i=\pi(p+j)\in \{\pi(p+q-1),\dots,\pi(p)\}$ for $0\leq j\leq q-1$. Since $i=\pi(p+j)>\dots>\pi(p)$, it follows from the simplicial identities that 
\begin{gather*}
    Ts_{\pi(p+q-1)}\dots Ts_{\pi(p)}Ts_iu= Ts_{\pi(p+q-1)}\dots Ts_{\pi(p+j+1)}Ts_{i+j+1}Ts_{\pi(p+j)}\dots Ts_{\pi(p)}u.
\end{gather*}
If $\pi(p+j+1)>i+j+1$, the index $i+j+1$ is not contained in $\{\pi(p+q-1),\dots,\pi(p)\}$, so $i+j+1\in \{\pi(p-1),\dots,\pi(0)\}$ and \eqref{blabla} follows from \eqref{eq:observation}.\\
Otherwise, we distinguish between the following two cases:
\begin{enumerate}
    \item[a)] There exists a minimal $l\in \{j+2,\dots,q-1\}$ such that $\pi(p+l)>i+l$. Then $i+l\in$ $\{\pi(p-1),\dots,\pi(0)\}$, so \eqref{eq:observation} applies.
    \item[b)] For all $l \in \{q-1,\dots,j+1\}: \pi(p+l)=i+l$. Then $p+q-1\geq i+q>\pi(p+q-1)$ and $i+q\in \{\pi(p-1),\dots,\pi(0)\}$, so \eqref{eq:observation} applies.
\end{enumerate}
Thus $\OO_m(Ts_{\pi(p+q-1)}\dots Ts_{\pi(p)}Ts_iu,Ts_{\pi(p-1)}\dots Ts_{\pi(0)}w)$ vanishes for $i \in \{\pi(p+q-1),\dots,\pi(p)\}$. The case when $i \in \{\pi(p-1),\dots,\pi(0)\}$ works similarly.\\
\\
\eqref{itm:mul}\: We first look at the case when $m=0$, and we want to prove Eq. \eqref{eq:inf-mul-0-app}. 
Notice that 
\[\delta \alpha_0(v, Ts_0 w) = \alpha_0 (Td_0 v, Td_0 Ts_0 w) - \alpha_0 (Td_1 v, Td_1 Ts_0 w). 
\]Since $Td_0v =0$,  $d_1s_0=\id$ and $D\alpha=\delta \alpha_0$, the desired equation \eqref{eq:inf-mul-0-app} is proven.

When $m\ge 1$,  since the de Rham degree 2 component of $D\alpha$ is given by $ \delta \alpha_m$, we have
\begin{equation*}
\begin{split}
   \lambda^{D\alpha_\bullet}(v, w)  
   = &\sum_{\pi \in\Sh(p+1,q)}\text{sgn}(\pi)\,\delta\OO_m(Ts_{\pi(p+q)}\dots Ts_{\pi(p+1)}v,Ts_{\pi(p)}\dots Ts_{\pi(0)}w)\\
   = &\sum_{i=0}^{m+1}(-1)^i\sum_{\pi \in\Sh(p+1,q)}\text{sgn}(\pi)\, \OO_m(Td_{i}Ts_{\pi(p+q)}\dots Ts_{\pi(p+1)}v, Td_{i}Ts_{\pi(p)}\dots Ts_{\pi(0)}w).
\end{split}
\end{equation*}
We begin by showing that 
\begin{gather*}
      \sum_{\pi \in\Sh(p+1,q)}\text{sgn}(\pi)\, \OO_m(Td_{i}Ts_{\pi(p+q)}\dots Ts_{\pi(p+1)}v, Td_{i}Ts_{\pi(p)}\dots Ts_{\pi(0)}w)=0
\end{gather*}
for $i \in \{0,\dots,m\}$.
For a fixed $i \in \{0,\dots,m\}$ and a shuffle $\pi \in \Sh(p+1,q)$, the index $i$ either lies in $\{\pi(p+q),\dots,\pi(p+1)\}$ or in $\{\pi(p),\dots,\pi(0)\}$. First, consider the case $i=\pi(p+1+j)\in  \{\pi(p+q),\dots,\pi(p+1)\}$. From the simplicial identities,  it follows that 
\begin{gather*}
    Td_{i}Ts_{\pi(p+q)}\dots Ts_{\pi(p+1)}v=Ts_{\pi(p+q)-1}\dots Ts_{\pi(p+1+j+1)-1}Ts_{\pi(p+j)}\dots Ts_{\pi(p+1)}v.
\end{gather*}Similarly, we have
\begin{gather}\label{eq:2star}
    \underbrace{Td_{i}Ts_{\pi(p)}\dots Ts_{\pi(0)}w}_{(\ast\ast)}=\begin{cases}Ts_{\pi(p)}\dots Ts_{\pi(0)}Td_{i-(p+1)}w,\:\:\:\:\:\:\text{if}\:\:i>1+\pi(p),\\
    \\
    Ts_{\pi(p)-1}\dots Ts_{\pi(0)-1}Td_{i}w, \:\:\:\:\:\:\text{if}\:\:i<\pi(0),\\
    \\
    Ts_{\pi(p)-1}\dots Ts_{\pi(l+1)-1}Ts_{\pi(l)}\dots Ts_{\pi(0)}Td_{i-(l+1)}w, \\
   \text{if}  \:\:\pi(l+1)>i>\pi(l)+1\:\:\text{for}\:\:l\in \{0,\dots,p-1\},\\
   \\
   Ts_{\pi(p)-1}\dots Ts_{\pi(l+1)-1}Ts_{\pi(l-1)}\dots Ts_{\pi(0)}w, \\\text{if}\:\:\:
   i=1+\pi(l)\:\:\text{for}\:\:l\in \{0,\dots,p\}.
   \end{cases}
\end{gather}We consider in each different situation:
\begin{enumerate}
    \item[a)] if $i>1+\pi(p)$ then $i-(p+1)\leq p+q-(p+1)=q-1$, so $Td_{i-(p+1)}w=0$ and $(\ast\ast)$ vanishes,
    \item[b)] if $i<\pi(0)$ then $i\leq p+q-(p+1)=q-1$, so $Td_{i}w=0$ and $(\ast\ast)$ vanishes,
    \item[c)] if $\pi(l+1)>i>\pi(l)+1$ for $l \in \{0,\dots,p-1\}$, then $i-(l+1)\leq q+l-(l+1)=q-1$, so $Td_{i-(l+1)}w=0$ and $(\ast\ast)$ vanishes.
\end{enumerate}
Thus we are left with the last situation in \eqref{eq:2star}. In this case, $i=1+\pi(l),\:l\in \{0,\dots,p\}$, thus we have $\pi(l)=\pi(p+1+j)-1$. We define a new $(p+1,q)$-shuffle $\Tilde{\pi}$ by 
\begin{gather*}
    \tilde{\pi}(k)=\begin{cases}i-1=\pi(l), \:\:\:\:\:\:\text{if}\:\:k=p+1+j,\\
    i=\pi(p+1+j), \:\:\:\:\:\:\text{if}\:\:k=l,\\
    \pi(k), \:\:\:\:\:\:\text{otherwise,}\end{cases}
\end{gather*} which can be illustrated as 
\begin{equation*}
\scriptsize{
    \begin{pmatrix}
0 & \dots & \mathbf{l} & \dots & p, & p+1 & \dots & \mathbf{p+1+j} & \dots & p+q\\
\pi(0)< &\dots & <\mathbf{\pi(p+1+j)-1}< & \dots & <\pi(p), & \pi(p+1)< & \dots & <\mathbf{\pi(l)}< & \dots & <\pi(p+q)
\end{pmatrix}}.
\end{equation*}
Then clearly $\text{sgn}(\pi)=-\text{sgn}(\tilde{\pi})$ and 
\begin{gather*}
Td_iTs_{\pi(p)}\dots Ts_{\pi(0)}w=Td_iTs_{\tilde{\pi}(p)}\dots Ts_{\tilde{\pi}(0)}w,\\
    Td_iTs_{\pi(p+q)}\dots Ts_{\pi(p+1)}v=Td_iTs_{\tilde{\pi}(p+q)}\dots Ts_{\tilde{\pi}(p+1)}v.
\end{gather*}
Notice that the terms
\begin{gather*}
    \text{sgn}(\pi)\,\OO_m(Td_iTs_{\pi(p+q)}\dots Ts_{\pi(p+1)}v,Td_iTs_{\pi(p)}\dots Ts_{\pi(0)}w),\\
    \text{sgn}(\tilde{\pi})\,\OO_m(Td_iTs_{\tilde{\pi}(p+q)}\dots Ts_{\tilde{\pi}(p+1)}v,Td_iTs_{\tilde{\pi}(p)}\dots Ts_{\tilde{\pi}(0)}w)
\end{gather*}
cancel with each other. The case $i\in \{\pi(p),\dots,\pi(0)\}$ can be treated similarly. Thus
\begin{gather*}
    \sum_{\pi \in\Sh(p+1,q)}\text{sgn}(\pi)\, \OO_m(Td_{m+1}Ts_{\pi(p+q)}\dots Ts_{\pi(p+1)}v, Td_{m+1}Ts_{\pi(p)}\dots Ts_{\pi(0)}w)=(-1)^{m+1} \lambda^{D\alpha_\bullet}(v, w).
\end{gather*}
We now distinguish between two types of $(p+1,q)$-shuffles: shuffles satisfying $\pi(p+q)=m$ and shuffles satisfying $\pi(p)=m$. There exists a 1-1 correspondence between $(p+1,q)$-shuffles $\pi$ with $\pi(p+q)=m$ and $(p+1,q-1)$-shuffles $\tau$ via $\tau(k)=\pi(k)$ for $k \in \{0,\dots,p+q-1\}$. Likewise there exists a 1-1 correspondence between $(p+1,q)$-shuffles $\pi$ with $\pi(p)=m$ and $(p,q)$-shuffles $\chi$ via
\begin{gather*}
    \chi(k)=\begin{cases}\pi(k)\:\:\:\:\:\text{if}\:\:k \in \{0,\dots,p-1\},\\ \pi(k+1)\:\:\:\:\:\text{if}\:\:k \in \{p,\dots,p+q-1\}.    \end{cases}
\end{gather*}
\noindent With these correspondences it follows that
\begin{equation} \label{eq:last-step}
    \begin{split}
    & (-1)^{m+1} \lambda^{D\alpha_\bullet}(v, w)  \\
    =&\sum_{\pi \in\Sh(p+1,q)}\text{sgn}(\pi)\, \OO_m(Td_{m+1}Ts_{\pi(p+q)}\dots Ts_{\pi(p+1)}v, Td_{m+1}Ts_{\pi(p)}\dots Ts_{\pi(0)}w)\\
    =&\sum_{\pi \in\Sh(p+1,q), \pi(p+q)=m}\text{sgn}(\pi)\, \OO_m(Ts_{\pi(p+q-1)}\dots Ts_{\pi(p+1)}v, Ts_{\pi(p)}\dots Ts_{\pi(0)}Td_qw)\\
    +&\sum_{\pi \in\Sh(p+1,q), \pi(p)=m}\text{sgn}(\pi)\, \OO_m(Ts_{\pi(p+q)}\dots Ts_{\pi(p+1)}Td_{p+1}v,Ts_{\pi(p-1)}\dots Ts_{\pi(0)}w)\\
    =&(-1)^q\sum_{\tau \in \Sh(p+1,q-1)}\text{sgn}(\tau)\,\OO_m(Ts_{\tau(p+q-1)}\dots Ts_{\tau(p+1)}v,Ts_{\tau(p)}\dots Ts_{\tau(0)}\partial w)\\
    +&(-1)^{p+1}\sum_{\chi \in \Sh(p,q)}\text{sgn}(\chi)(-1)^q\,\OO_m(Ts_{\chi(p+q-1)}\dots Ts_{\chi(p)}\partial v,Ts_{\chi(p-1)}\dots Ts_{\chi(0)} w)\\
    =&(-1)^q \lambda^{\alpha_\bullet}(v,\partial w)+(-1)^q(-1)^{p+1} \lambda^{\alpha_\bullet} (\partial v,w).
    \end{split}
\end{equation}
Thus \eqref{eq:inf-mul-app} is proven.\\
\\
\eqref{itm:gauge}\: item \eqref{itm:mul} tells us that $\lambda^{D\phi_\bullet, \#} =
\partial \lambda^{\phi_\bullet, \#} + \lambda^{\phi_\bullet, \#} \partial^*$. Thus $\lambda^{\phi_\bullet, \#}$ provides a homotopy between $\lambda_{x}^{\alpha_{\bullet}+D\phi_{\bullet},\#}$ and $\lambda_{x}^{\alpha_{\bullet},\#}$.

\end{proof}

\bibliographystyle{plain}
\bibliography{bibz, morerefs}

\end{document}